\documentclass[12pt, reqno, a4paper,oneside]{amsart}

\usepackage{lipsum}
\usepackage{amsfonts}
\usepackage{graphicx}
\usepackage{epstopdf}
\usepackage{algorithmic}
\usepackage{algorithm}
\usepackage{mathrsfs}
\usepackage{amsmath}
\usepackage{amssymb}
\usepackage{textcomp}
\usepackage{color}
\usepackage{geometry}
\usepackage[normalem]{ulem}
\usepackage{placeins}
\usepackage{bm}
\usepackage{wasysym}
\usepackage{stmaryrd}
\usepackage{float}
\usepackage{subfigure}

\usepackage{environments}
\numberwithin{theorem}{section}
\numberwithin{equation}{section}

\newcommand{\burg}{{\sf b}}

\newcommand\R{\mathbb{R}}
\newcommand\Rg{\mathcal{R}}
\newcommand\T{\mathcal{T}}
\newcommand\N{\mathbb{N}}
\newcommand\Z{\mathbb{Z}}
\newcommand\C{\mathbb{C}}
\newcommand\E{\mathcal{E}}

\newcommand\del{\delta}

\newcommand{\<}{\langle}
\renewcommand{\>}{\rangle}

\newcommand{\ulin}{u^{\rm lin}}

\newcommand{\Mrefine}{\mathcal{M}_{\rm r}}

\def\Xint#1{\mathchoice
{\XXint\displaystyle\textstyle{#1}}%
{\XXint\textstyle\scriptstyle{#1}}%
{\XXint\scriptstyle\scriptscriptstyle{#1}}%
{\XXint\scriptscriptstyle\scriptscriptstyle{#1}}%
\!\int}
\def\XXint#1#2#3{{\setbox0=\hbox{$#1{#2#3}{\int}$ }
\vcenter{\hbox{$#2#3$ }}\kern-.6\wd0}}
\def\mint{\Xint-}

\def\diam{{\rm diam}}

\def\diam{{\rm diam}}

\def\R{\mathbb{R}}
\def\N{\mathbb{N}}
\def\Z{\mathbb{Z}}
\def\C{\mathbb{C}}

\def\DD{{D}'}

\def\dx{\,{\rm d}x}

\def\dt{\,{\rm d}t}

\def\<{\langle}
\def\>{\rangle}

\def\mA{{\sf A}}

\def\del{\delta}

\def\a{{\rm a}}

\def\c{{\rm c}}
\def\h{{\rm h}}

\def\a{{\rm a}}
\def\c{{\rm c}}

\def\L{\Lambda}

\def\Us{\mathscr{U}}
\def\Usz{\Us_0}
\def\Ush{\dot{\Us}^{1,2}}

\def\E{\mathscr{E}}

\def\Eb{\E^{\rm b}}

\def\RO{\mathcal{R}}

\def\L{\Lambda}

\def\rcut{r_{\rm cut}}
\def\Rg{\mathscr{R}}

\def\T{\mathcal{T}}

\def\Th{\mathcal{T}_h}

\def\Ih{I_h}

\def\vor{\rm vor}

\begin{document}

\title[A Posteriori Error Estimates for Adaptive QM/MM]{A Posteriori Error Estimates for Adaptive QM/MM Coupling Methods}

\author{Yangshuai Wang}
\address{Yangshuai Wang\\
Institute of Natural Sciences, School of Mathematical Sciences, and MOE-LSC\\
Shanghai Jiao Tong University\\
Shanghai\\
China}
\email{yswang2016@sjtu.edu.cn}

\author{Huajie Chen}
\address{Huajie Chen\\
School of Mathematical Sciences\\
Beijing Normal University\\
Beijing\\
China}
\email{chen.huajie@bnu.edu.cn}

\author{Mingjie Liao}
\address{Mingjie Liao\\
Institute of Natural Sciences, School of Mathematical Sciences, and MOE-LSC\\
Shanghai Jiao Tong University\\
Shanghai\\
China}
\email{mingjieliao@sjtu.edu.cn}

\author{Christoph Ortner}
\address{Christoph Ortner\\
Mathematics Institute\\
University of Warwick\\
Coventry CV4 7AL\\
UK}
\email{c.ortner@warwick.ac.uk}

\author{Hao Wang}
\address{Hao Wang\\
School of Mathematics\\
Sichuan University\\
Sichuan\\
China}
\email{wangh@scu.edu.cn}

\author{Lei Zhang}
\address{Lei Zhang\\
Institute of Natural Sciences, School of Mathematical Sciences, and MOE-LSC\\
Shanghai Jiao Tong University\\
Shanghai\\
China}
\email{lzhang2012@sjtu.edu.cn}

\thanks{YW, ML, and LZ are partially supported by National Natural Science Foundation of China 
(NSFC 11871339, 11861131004). HC is supported by National Natural Science 
Foundation of China (NSFC 11971066) and the National Key Research and 
Development Program of China (2019YFA0709601). CO is supported by EPSRC Grant 
EP/R043612/1 and by the Leverhulme Trust under Grant RPG-2017-191. CO and LZ are 
further supported by the SJTU-Warwick Joint Seed Fund 2019/20. HW is 
supported by National Natural Science Foundation of China (NSFC 11971336, 
11501389).}

\begin{abstract}
Hybrid quantum/molecular mechanics models (QM/MM methods) are widely used in material and molecular simulations when MM models do not provide sufficient accuracy but pure QM models are computationally prohibitive. Adaptive QM/MM coupling methods feature on-the-fly classification of atoms during the simulation, allowing the QM and MM subsystems to be updated as needed. In this work, we propose such an adaptive QM/MM method for material defect simulations based on a new residual based {\it a posteriori} error estimator, which provides both lower and upper bounds for the true error. We validate the analysis and illustrate the effectiveness of the new scheme on numerical simulations for material defects.
\end{abstract}


\maketitle

\section{Introduction}
\label{sec:introduction}
Quantum mechanics and molecular mechanics (QM/MM) coupling methods have been widely used for simulations of large systems in materials science and biology \cite{bernstein09, csanyi04, gao02, kermode08, ogata01, zhang12}.
In QM/MM simulations, the computational domain is partitioned into QM and MM regions. The region of primary interest (e.g., a material defect) is described by a QM model, which is embedded in an ambient environment (e.g., bulk crystal) that is described by an MM model.
In this manner, QM/MM methods can in principle combine the accuracy of a quantum mechanical description with the efficiency of classical molecular mechanics.

A fundamental challenge for QM/MM methods is how to optimally assign each atom to QM or MM subsystems so that a (quasi-)optimal balance between accuracy and efficiency can be achieved.
{\it A priori} choices, even when they are feasible, typically lead to sub-optimal distribution of computational resources.
Adaptive QM/MM coupling methods offer an automatic partition of QM/MM subsystems on the fly according to some error indicators during the simulation process. In addition to the optimisation of computational cost, this allows an adaption to {\em moving} regions of interest.
Adaptive QM/MM methods have been proposed, e.g., for the study of molecular fragments in macromolecules, monitoring molecules entering/leaving binding sites and tracking proton transfer via the Grotthuss mechanism (see \cite{duster17} and references therein).
Most adaptive QM/MM methods are for solute-solvent system, and are based on different (heuristic) criteria, such as distance to active sites \cite{heyden2007, kerdcharoen1996, watanabe2014};  Hamiltonian conservation \cite{boereboom2016};  density based adaptivity \cite{waller2014}; number adaptivity \cite{Tekenaka:2012}; local atomic stress \cite{Glukhova:2014}.
For materials with defects, \cite{csanyi04, kermode08} propose the criterion of distance to defect. Closely related ideas can be found in the quasi-continuum method for density-functional theory~\cite{Gavini:2007, Ponga:2016}.

The various {\it a posteriori} error estimators proposed in \cite{boereboom2016, heyden2007, Glukhova:2014, kerdcharoen1996, waller2014, watanabe2014} provide not only estimates and theoretical bounds for the solution error in a specified metric, but also naturally lead to a QM/MM partitioning criterion of the atomic sites.

Inspired by classical adaptive finite element methods~\cite{Verfurth:1996a, Dorfler1996, Wolfgang:2013}, \cite{CMAME} introduced the idea of using {\it a posteriori} error indicators for the QM/MM model residual. Using a weighted $\ell^2$-norm on the QM/MM force error leads to a simple and practical scheme, but makes it impossible to guarantee lower bounds on the error, which is important to guarantee the efficiency of the algorithm. There have also been investigations of related {\it a posteriori} error estimates for atomistic/continuum (A/C) coupling methods, which share many similarities~\cite{arndtluskin07c,Ortner:qnl.1d,OrtnerWang:2014,prud06, oden08, Shenoy:1999a, Wang:2017, Liao2018, HW_SY_2018_Efficiency_A_Post_1D}.


In the present work, we construct a reliable and efficient {\it a posteriori} error estimator based on a natural {\em dual norm} of the model residual. The dual norm is itself not computable since it requires the evaluation of the residual forces as well as the solution of an auxiliary Poisson problem on the whole space. We therefore construct a computable {\em approximate estimator} by truncating the Poisson problem to a finite domain and finite-dimensional approximation space. We then estimate the errors we committed in this additional step and demonstrate that this leads to a practical, yet still {\em reliable and efficient} estimator up to a ``data oscillation term''. We propose an adaptive QM/MM algorithm for material defects, based on this new estimator. Aside from providing both upper and lower bounds the new estimator moves us closer to our goal of a fully adaptive QM/MM scheme without requiring any {\it a priori} input from the user.

As a proof of concept, we will restrict ourselves to the tight binding model as the quantum mechanical model and a prototypical QM/MM model~\cite{chen15b}, as well as geometry equilibration problems (statics) of a defect in a homogeneous simple lattice crystal.

\subsubsection*{Outline}
In Section \ref{sec:pre}, we briefly describe the tight binding model, the variational formulation for the equilibration of crystalline defects, and the QM/MM coupling methods that we consider.
In Section \ref{sec:posteriori}, we construct the {\it a posteriori} error estimator based on a finite element approximation of the residual of the QM/MM solution and establish both lower and upper bounds of the approximation error.
In Section \ref{sec:algo}, we describe the adaptive QM/MM algorithm in detail, including the adaptive algorithm to control the approximation error of the approximate estimator.
In Section \ref{sec:numer}, we present several numerical examples of point defects and an edge dislocation by our adaptive algorithm. In Section \ref{sec:conclusion} we provide a summary and outlook.

\subsubsection*{Notation}
We use the symbol $\langle\cdot,\cdot\rangle$ to denote an abstract duality
pair between a Banach space and its dual space. The symbol $|\cdot|$ normally
denotes the Euclidean or Frobenius norm, while $\|\cdot\|$ denotes an operator
norm.
For second order tensors $A$ and $B$, we denote $A:B = \sum_{i,j}A_{ij}B_{ij}$ and $A\otimes B$ the standard kronecker product.
For the sake of brevity, we will denote $A\backslash\{a\}$ by
$A\backslash a$, and $\{b-a~\vert ~b\in A\}$ by $A-a$.
For functional $E \in C^2(X)$, the first and second variations are denoted by
$\<\delta E(u), v\>$ and $\<\delta^2 E(u) v, w\>$ for $u,v,w\in X$, respectively.
For a finite set $A$, we will use $\#A$ to denote the cardinality of $A$.
The closed ball with radius $r$ and center $x$ is denoted by $B_r(x)$.
The symbol $C$ (or $c$) denotes generic positive constant that may change from one line
of an estimate to the next. When estimating rates of decay or convergence, $C$
will always remain independent of approximation parameters such as the system size, the configuration of the lattice and the test functions. The dependence of $C$ will be clear from the context or stated explicitly. To further simplify notation we will often write $\lesssim$ to mean $\leq C$ as well as $\eqsim$ to mean both $\lesssim$ and $\gtrsim$.
We use the standard definitions and notations $L^p$, $W^{k,p}$, $H^k$ for Lebesgue and Sobolev spaces. In addition we define the homogeneous Sobolev spaces
$
	\dot{H}^k(\Omega) := \big\{ f \in H^k_{\rm loc}(\Omega) \,|\,
										\nabla^k f \in L^2(\Omega) \big\}.
$

\clearpage

\section{QM/MM Coupling for Crystalline Defects}
\label{sec:pre}
\setcounter{equation}{0}

\subsection{The tight binding model}
\label{sec:tb}

\def\Rc{R_{\rm cut}}
\def\Nn{N}

Consider a many-particle system consisting of $\Nn$ atoms.
Let $d\in\{2,3\}$ be the space dimension and $\Pi\subset\R^d$ be an {\it index set}  (or a {\it reference configuration}), with $\#\Pi=\Nn$.
An atomic configuration is a map $y : \Pi\to\R^d$ satisfying
\begin{equation} \label{eq:non-interpenetration}
|y(\ell)-y(k)| \geq \mathfrak{m}|\ell-k| \qquad\forall~\ell,k\in\Pi
\end{equation}
with {\em accumulation parameter} $\mathfrak{m} > 0$. We will use $r_{\ell k}:=|y(\ell)-y(k)|$ for  brevity of notation.

For the sake of notational simplicity we restrict the presentation to orthogonal two-centre tight binding models \cite{goringe97,Papaconstantopoulos15}, with a single orbital per atom. All results and algorithms can be extended directly to general linear and some nonlinear (self-consistent) tight binding models, using the techniques described in \cite[\S~2 and Appendix A]{chen15a} and in~\cite{Thomas2020nl}.

Our model is formulated in terms of a discrete Hamiltonian with the matrix elements
\begin{eqnarray}\label{tb-H-elements}
\Big(\mathcal{H}(y)\Big)_{\ell k}
=\left\{ \begin{array}{ll} 
h_{\rm ons}\left(\sum_{j\neq \ell}
\varrho\big(|y({\ell})-y(j)|\big)\right),
& {\rm if}~\ell=k, \\[1ex]
h_{\rm hop}\big(|y(\ell)-y(k)|\big), & {\rm if}~\ell\neq k,
\end{array} \right.
\end{eqnarray}
in which $h_{\rm ons} \in C^{\mathfrak{n}}([0, \infty))$ is the on-site term,
$\varrho \in C^{\mathfrak{n}}([0, \infty))$ represents the charge density
with $\varrho(r) = 0~\forall r\in[\Rc,\infty)$ where $\Rc>0$ stands for the cutoff radius, and
$h_{\rm hop} \in C^{\mathfrak{n}}([0, \infty))$ is the hopping term with
$h_{\rm hop}(r)=0~\forall r\in[\Rc,\infty)$.
We assume throughout that $\mathfrak{n}\geq 4$.

Let $\psi_s, \varepsilon_s$, $s=1,2,\cdots,N$, be the solutions of the eigenvalue problem $\mathcal{H}(y)\psi_s = \varepsilon_s\psi_s, \|\psi_s\| = 1,$ then we define the {\em band energy} to be 
\begin{eqnarray}\label{e-band}
E^\Pi(y)= {\textstyle \sum_{s=1}^N} \mathbf{f}(\varepsilon_s)\varepsilon_s,
\end{eqnarray}
where $\mathbf{f}(\varepsilon) = ( 1 + e^{\beta (\varepsilon-\mu)} )^{-1}$ is the Fermi-Dirac distribution function for the energy states of a particle system obeying the Pauli exclusion principle. The inverse Fermi-temperature, $\beta$, and the chemical potential, $\mu$, are fixed throughout (see \cite{chen16} for a rigorous justification of this choice).
%

The starting point for the QM/MM we discuss below is a spatial partition of the energy~\cite{finnis03},
\begin{eqnarray}\label{E-El}
E^\Pi(y)=\sum_{\ell\in\Pi} E_{\ell}^{\Pi}(y)
\qquad{\rm with}\qquad
E_{\ell}^\Pi(y) := \sum_{s}\mathbf{f}(\varepsilon_s)\varepsilon_s
\left|[\psi_s]_{\ell}\right|^2,
\end{eqnarray}
which formally defines a site energy $E_{\ell}^{\Pi}(y)$ and provides a connection between the tight-binding model and classical interatomic potentials (molecular mechanics). To make this connection quantitative we now review the {\it regularity and locality} results for $E_{\ell}^\Pi$ from \cite{chen15a}:
Suppose $\L$ is a countable index set or reference configuration and $\Pi\subset\Lambda$ is a finite subset.
We denote by $E_\ell^\Pi$ the site energy with respect to the subsystem $\Pi \subset \Lambda$.
For a continuous domain $A \subset \R^d$, we use the short-hand $E_\ell^{A} := E_\ell^{A \cap \Lambda}$.
%
The following lemma from \cite[Theorem 3.1 (i)]{chen15a} implies the existence of the {\it thermodynamic} limit of $E_\ell^{\Pi}$ as $\Pi \uparrow \Lambda$ and guarantees that $E_{\ell}^{\Pi}$ defined in \eqref{E-El} can be taken as a proper (approximate) site energy.

\begin{lemma}\label{lemma-thermodynamic-limit}
If $y:\L\rightarrow\R^d$ is a configuration satisfying \eqref{eq:non-interpenetration}, then,
	\begin{itemize}
		\item[(i)] {\rm (regularity and locality of the site energy)}
		$E^{\Pi}_{\ell}(y)$ possesses $j$-th order partial derivatives with
		$1 \leq j \leq \mathfrak{n}-1$, and there exist positive constants $C_j$ and $\eta_j$ such that
		\begin{eqnarray}\label{site-locality-tdl}
		\left|\frac{\partial^j E^{\Pi}_{\ell}(y)}{\partial [y(m_1)]_{i_1}
			\cdots\partial [y(m_j)]_{i_j}}\right|
		\leq C_j e^{-\eta_j\sum_{l=1}^j|y(\ell)-y(m_l)|}
		\end{eqnarray}
		with $m_k\in\Pi$ and $1\leq i_k\leq d$ for any $1\leq k\leq j$;
		\item[(ii)] {\rm (thermodynamic limit)}
		$\displaystyle E_{\ell}(y):=\lim_{R\rightarrow\infty} E^{B_R(\ell)}_{\ell}(y)$
		exists and satisfies (i).
	\end{itemize}
\end{lemma}

\subsection{Variational model for crystalline defects}
\label{sec:defects}
\def\Rdef{R^{\rm def}}
\def\Rg{\mathcal{R}}
\def\Rgnn{\mathcal{N}}
\def\rcut{R_{\rm c}}
\def\Lhom{\L^{\rm hom}}
\def\Ddef{D^{\rm def}}
\def\Ldef{\L^{\rm def}}
\def\Rcore{R_{\rm DEF}}
\def\Adm{{\rm Adm}}
\def\E{\mathcal{E}}
\def\L{\Lambda}
\def\UsH{{\mathscr{U}}^{1,2}}
\def\Usz{{\mathscr{U}^{\rm c}}}
\def\DD{{\sf D}}
\def\ee{{\sf e}}
A rigorous framework for geometry equilibration of crystalline defects was developed in \cite{chen19,2013-defects}, which formulates the equilibrium of crystal defects as a variational problem in a discrete energy space and establishes qualitatively sharp far-field decay estimates for the corresponding equilibrium configuration. This framework will be the backbone of our rigorous {\it a posteriori} error analysis.

Given $d \in \{2, 3\}$, $\mA \in \R^{d \times d}$ non-singular,  $\Lhom := \mA \Z^d$ is the homogeneous reference lattice which represents a homogeneous crystal formed from identical atoms and possessing no defects. A reference lattice with a single defect in a localized defect core region is denoted by $\L \subset \R^d$. We assume the defect is contained within a ball $B_{\Rcore}$, $\Rcore> 0$; that is,
%
$\L \setminus B_{\Rcore} = \Lhom \setminus B_{\Rcore}.$
%
The deformed configuration of the infinite lattice $\L$ is a map $y: \L\rightarrow\R^d$ which we decompose into
\begin{eqnarray} \label{y-u}
y(\ell) = x_0(\ell) + u_0(\ell) + u(\ell) = y_0(\ell) + u(\ell),
\end{eqnarray}
with $x_0:\L\rightarrow\R^d,~ x_0(\ell)=\ell$, $u_0:\Lambda\rightarrow\mathbb{R}^d$ a predictor prescribing the far-field boundary condition, and $u:\Lambda\rightarrow\mathbb{R}^d$ the corrector.
We require that the configuration $y_0(\ell)$ is ``near equilibrium'' far from the defect core. For point defects we achieve this by simply taking $u_0=0$. The derivation of $u_0$ for straight dislocations is reviewed in \S~\ref{sec:deri_u0}; see also \cite{2013-defects} for further details.

For a subset $\Rg \subset \Lambda-\ell$, we
define $D_\Rg u(\ell) := (D_\rho u(\ell))_{\rho\in\Rg}$, $D_\rho u(\ell) := u(\ell+\rho) - u(\ell)$, and $Du(\ell) := D_{\Lambda-\ell} u(\ell)$.
For $\gamma > 0$ we then define the (semi-)norms
\begin{eqnarray*}
	\big|Du(\ell)\big|_\gamma := \bigg( \sum_{\rho \in \L-\ell} e^{-2\gamma|\rho|}
	\big|D_\rho u(\ell)\big|^2 \bigg)^{1/2}
	\quad{\rm and}\quad
	\| Du \|_{\ell^2_\gamma} := \bigg( \sum_{\ell \in \L}
	|Du(\ell)|_\gamma^2 \bigg)^{1/2}.
\end{eqnarray*}
All (semi-)norms $\|\cdot\|_{\ell^2_\gamma}, \gamma > 0,$ are equivalent~\cite{2012-lattint} (also \cite[Appendix A]{chen15b}).
We can therefore define the natural function space of finite-energy displacements,
\begin{displaymath}
\UsH := \big\{ u : \L \to \R^d, \| Du \|_{\ell^2_\gamma} < \infty \big\}.
\end{displaymath}
For a displacement $u \in \UsH$ with associated configuration $y_0+u$ satisfying the accumulation condition \eqref{eq:non-interpenetration},  we define the energy-difference functional
\begin{align}
\label{energy-difference}
\mathcal{E}(u) := \sum_{\ell\in\Lambda}\Big(E_{\ell}(y_0+u)-E_{\ell}(y_0)\Big).
\end{align}
It was shown in \cite[Theorem 2.7]{chen19} (see also \cite{2013-defects}) that,
if $\del\E(0) \in (\UsH)^*$, then $\E$ is well-defined on the space $\Adm_0$ and in fact
$\E \in C^{\mathfrak{n}-1}(\Adm_0)$, where
\begin{multline*}
\Adm_{\mathfrak{m}} := \big\{ u \in \UsH:
\big|\big(y_0(\ell)+u(\ell)\big)-\big(y_0(m)+u(m)\big)\big| > \mathfrak{m} |\ell-m| \quad\forall~  \ell, m \in \L \big\}. \quad
\end{multline*}
Due to the decay imposed by the condition $u\in\UsH$, any displacement $u\in\Adm_0$ belongs to $\Adm_{\mathfrak{m}}$ with some constant $\mathfrak{m}>0$ \cite{chen15a}.

We can now rigorously formulate the equilibration problem, which serves as our benchmark application for the remainder of the paper, 
\begin{equation}\label{eq:variational-problem}
\bar{u} \in \arg\min \big\{ \E(u), u \in \Adm_0 \big\},
\end{equation}
where ``$\arg\min$'' is understood as the set of local minima. One can generalise this model to include more general equilibria, in particular saddle points~\cite{2018-uniform}, but for the sake of simplicity we will restrict ourselves to minima.

\subsection{QM/MM Coupling}
\label{sec:qmmm}
\def\Adm{{\rm Adm}}
\def\E{\mathcal{E}}
\def\L{\Lambda}
\def\DD{{\sf D}}
\def\ee{{\sf e}}
\def\LQM{\Lambda^{\rm QM}}
\def\LMM{\Lambda^{\rm MM}}
\def\LFF{\Lambda^{\rm FF}}
\def\Lbuf{\Lambda^{\rm BUF}}
\def\OQM{\Omega^{\rm QM}}
\def\OMM{\Omega^{\rm MM}}
\def\OFF{\Omega^{\rm FF}}
\def\Obuf{\Omega^{\rm BUF}}
\def\RQM{R_{\rm QM}}
\def\RMM{R_{\rm MM}}
\def\RFF{R_{\rm FF}}
\def\Rbuf{R_{\rm BUF}}
\def\VMM{V^{\rm MM}}
\def\Vb{V^{\rcut}_{\#}}
\def\EH{\mathcal{E}^{\rm H}}
\def\uH{\^{u}^{\rm H}}
\def\wD{\widetilde{D}}
\def\AH{\Adm^{\rm H}_0}
\def\Usx{\mathscr{U}^{\rm H}}
\def\uH{\bar{u}^{\rm H}}
\def\oscf{\mathrm{osc}_{\T^0}^\omega(\hat{f})}
To construct computational models for the variational problem \eqref{eq:variational-problem} we will restruct ourselves to the {\em consistent energy-based QM/MM models} of~\cite{chen15b}. However, our a posteriori error estimates and adaptive algorithms are largely agnostic about the underlying approximation scheme and we expect that most of our analysis and algorithms apply directly or can be generalized to other QM/MM methods (including force-mixing methods~\cite{bernstein09,chen15b}) and entirely different classes of coarse-graining or multi-scale methods.

\begin{figure}
        \centering
        \includegraphics[scale=0.8]{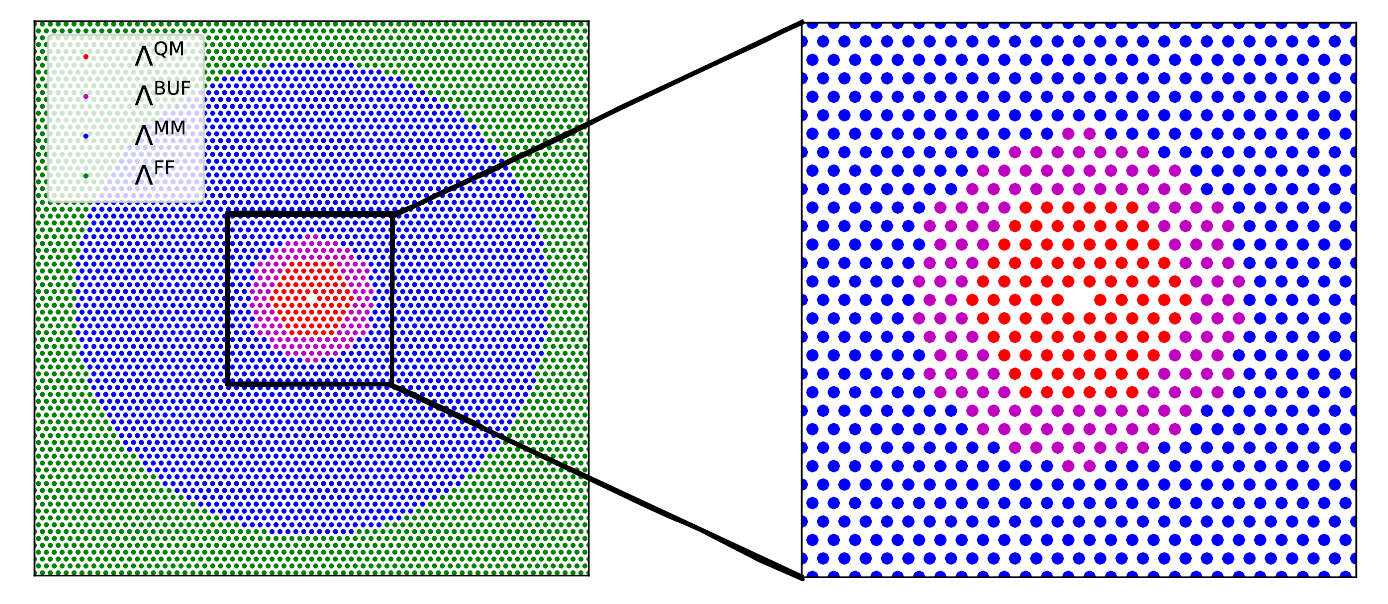}
        \caption{Domain decomposition in the QM/MM coupling scheme.}
        \label{fig:decomposition}
\end{figure}

\paragraph{Domain decomposition:} Closely following~\cite{chen15b} (where further details can be found), the first step in constructing a QM/MM approximation is to decompose the reference configuration $\L$ into three disjoint sets, $\Lambda = \LQM \cup \LMM\cup \LFF$, where $\LQM$ denotes the QM region containing the defect core, $\LMM$ denotes the MM region and $\LFF$ denotes the far-field region where atom positions will be clamped to the far-field predictor $y_0$. This yields the approximate admissible set
\begin{eqnarray}\label{e-mix-space}
\Adm_{0}^{\rm H} := \Adm_{0} \cap \Usx
\qquad \text{with} \qquad
\Usx := \left\{ u \in \UsH  ~\lvert~ u=0~{\rm in}~\Lambda^{\rm FF} \right\} .
\end{eqnarray}
In addition we specify a buffer region $\Lbuf\subset\LMM$ surrounding $\LQM$ such that all atoms in $\Lbuf\cup\LQM$ are involved in the evaluation of the site energies in $\LQM$ using the tight binding model. This decomposition is illustrated in Figure~\ref{fig:decomposition}.

\paragraph{Specification of the MM model:}
In the MM region, we approximate the tight binding site energy $E_{\ell}$ by an MM site energy $E^{\rm MM}_{\ell}$, which will be constructed such that it is cheap to evaluate, but provides an accurate representation of interatomic interaction at some distance from the defect core. These mild requirements motivate the use of a Taylor expansion~\cite[eq. (36)]{chen15b}: First, we fix some MM cutoff radius $\rcut$, to obtain a truncated QM site energy $E_\ell^{B_{\rcut}(\ell)}$. Next, we prescribe an accuracy parameter $k \geq 2$, to obtain the MM site energy as the $k$-order Taylor expansion,  
\begin{eqnarray}\label{taylor}
E^{\rm MM}_{\ell}(y_0 + u)
:= \sum_{j=0}^k \frac{1}{j!} \delta^j E_{\ell}^{B_{\rcut}(\ell)}(y_0)[
	\underset{\text{$j$ times}}{\underbrace{u, \dots, u}} ]
\end{eqnarray}
where $\delta^j E_{\ell}^{\Pi}(y_0)\left[u, \dots, u\right]$ denotes the $j$-th order variations of $E_\ell^{\Pi}$.
%
This construction is used throughout the numerical experiments in \S \ref{sec:numer}. Finally, we remark that for $|\ell| > \Rcore + \rcut$ the lattice becomes homogeneous in the ball $B_{\rcut}(\ell)$ and hence the Taylor-potential site-energies have the same coefficients, i.e., they are homogeneous as well.

\paragraph{The QM/MM hybrid model: }
The QM/MM hybrid energy functional approximating the QM energy difference functional $\E$ is given by 
\begin{multline}\label{eq:hybrid_energy}
\quad \E^{\rm H}(u) = \sum_{\ell\in \LQM}  \Big( E_{\ell}^{\rm BUF}(y_0+u) - E_{\ell}^{\rm BUF}(y_0) \Big) + \sum_{\ell\in \LMM\cup\LFF}  \Big( E^{\rm MM}_{\ell}(y_0+u) - E^{\rm MM}_{\ell}(y_0)\Big),
\qquad
\end{multline}
where the buffered QM site energy is given by $E_\ell^{\rm BUF} := E_\ell^{\Lbuf\cup\LQM}$. The fully discrete (computable) energy-based QM/MM scheme, as an approximation to \eqref{eq:variational-problem}, is now given by the finite dimensional minimization problem
\begin{eqnarray}\label{problem-e-mix}
    \bar{u}^{\rm H} \in \arg\min\big\{ \mathcal{E}^{\rm H}(u) ~\lvert~ u\in \AH \big\}.
\end{eqnarray}

\begin{remark}
    We have chosen a QM/MM model that is {\em consistent} with the reference QM model in the following sense: If $\bar{u}$ is a strongly stable solution of \eqref{eq:variational-problem}, i.e., $\delta^2 \mathcal{E}(\bar{u})$ is positive in $\UsH$, then for sufficiently large QM and buffer regions there exist equilibria $\uH$ solving \eqref{problem-e-mix}, such that 
\begin{eqnarray}\label{ass:convergence_QMMM}
    \lim_{\RQM, \Rbuf \rightarrow \infty} \|\uH-\bar{u}\|_{\UsH} = 0.
\end{eqnarray}
    We refer to \cite{chen15b} for a precise statement and sharp convergence rates, but emphasize that the Taylor expansion construction of the MM site potential about the far-field lattice state is a key ingredient. Such a result not only gives confidence in our scheme, but for the purpose of the present paper it also  allows us to relate {\it a posteriori} residual estimates to error estimates; cf. Proposition~\ref{lemma:res}.
\end{remark}



\def\Tr{\rm Tr}
\def\TH{T_{\rm H}}
\def\eo{\epsilon^{\Omega}}
\def\et{\epsilon^{\T}}
\def\diam{\mathrm{diam}}
\def\SO{\sigma^{\Omega}}
\def\uo{u^{\Omega}}
\def\phio{\phi^\Omega}
\def\Uec{\mathscr{U}_{0}}
\def\etao{\eta^{\Omega}}
\def\epso{\epsilon^\Omega(\uH)}
\def\etaoh{\eta^{\Omega}_{\rm h}}
\def\Uh{\mathscr{U}_{ \rm h}}

\def\uot{u^{\Omega}_{\T}}
\def\phioh{\phi^{\Omega}_{\rm h}}
\def\SOT{\sigma^{\Omega}_{\rm h}}
\def\Th{\mathcal{T}_h}
\def\Nh{\mathcal{N}_h}
\def\Ush{\Us_h}
\def\Ra{R^\a}
\def\Rb{R^{\rm b}}
\def\Eb{\E^{\rm b}}
\def\dof{{\rm DOF}}
\def\Omh{\Omega_h}
\def\Thr{{\T_{h,R}}}
\def\vor{\rm vor}
\def\Uhr{\Us_{h,R}}
\def\N{\mathcal{N}}
\def\eOh{\etaoh(\uH)}

\def\Th{\T^{\rm h}}
\def\epsh{\epsilon^{\rm h}(\uH)}
\def\epsi{\epsilon^{\rm i}(\uH)}
\def\phiof{\phi^{\Omega}_{\rm f}}
\def\Ih{\mathcal{I}}
\def\NMM{N_{\rm MM}}
\def\NQM{N_{\rm QM}}
\def\NBRc{N_{B_{R_c}}}
\def\NT{N_0}
\def\~{\tilde}
\def\-{\^}
\def\^{\hat}
\def\U{\mathscr{U}}
\def\d{{\rm d}}

\section{A Posteriori Error Estimates for QM/MM Coupling}
\label{sec:posteriori}
\setcounter{equation}{0}

In this section, we construct a negative-norm {\it a posteriori} error estimator for the QM/MM approximation $\uH$, and show that the estimator provides both lower and upper bounds of the approximation error.

\subsection{Lattice interpolants}
For technical purposes, it will be convenient to interpret the lattice $\Lambda$ as the vertex set of a simplicial grid $\T$, the {\em canonical partition}, as follows: first, we construct a regular (i.e., periodic) subdivision $\T^{\textrm{hom}}$ with nodes $\Lhom$ (the homogenous lattice); see e.g. \cite[Fig. 1]{2014-bqce} for concrete constructions. 
We then assume that the {\em canonical partition}, $\T$, coincides with $\T^{\textrm{hom}}$ outside the defect core region; that is, we assume that $\T \cap \T^{\textrm{hom}} \subset \{ T \in \T^{\textrm{hom}} : T \cap B_{\Rcore} = \emptyset\}$.


%
Let $\zeta_{\ell}(x)\in W^{1,\infty}(\R^d;\R)$ be the $\mathcal{P}_1$ nodal basis function associated with $\T$, then we extend all lattice displacements $u : \L \to \R^d$ to $\R^d$, via their nodal interpolants, 
\begin{equation}\label{eq:interp}
 	u(x) := \sum_{\ell\in\L} u(\ell) \zeta_{\ell}(x).
\end{equation}
%

We then have the following norm-equivalence, for constants $c$ and $C$ depending on $\gamma$, \cite{chen15b}
\begin{equation}\label{eq:norm-eq}
  c\|\nabla u\|_{L^2} \leq \| Du \|_{\ell^2_\gamma} \leq C \|\nabla u\|_{L^2}\qquad \forall \gamma > 0.
\end{equation}
We prefer to use $\|\nabla u\|_{L^2}$ as a semi-norm for $\UsH$, but employ $\|Du\|_{\ell^2_\gamma}$ primarily when estimating interactions, where the parameter $\gamma$ then becomes a measure of the interaction decay. In the same spirit, we now define the corresponding dual norm to be
\begin{equation}\label{eq:res-norm}
    \|\delta \E(u)\|_{(\UsH)^*} := \sup_{v \in \UsH \setminus \{ \text{constants} \}} \frac{\<\delta \E(u), v\>}{\|\nabla v\|_{L^2}}.
\end{equation}

\subsection{An abstract estimator}
\def\fR{\mathfrak{R}}
\def\etaR{\eta^{\fR}}
\def\etaS{\eta^{\mathfrak{S}}}
Under a suitable local stability condition, the {\em residual} $\delta\E(\uH)$ of a solution $\uH\in\Usx$ of \eqref{problem-e-mix} characterises its error (see, e.g., \cite[Lemma 3.1]{CMAME}).

\begin{proposition}\label{lemma:res}
	Let $\bar{u}$ be a strongly stable solution of \eqref{eq:variational-problem}. If the QM/MM method is consistent, \eqref{ass:convergence_QMMM} then for $\RQM, \Rbuf$ sufficiently large, there exists a QM/MM solution $\uH$ to \eqref{problem-e-mix} and constants $c, C$ independent of the approximation parameters such that 
	\begin{eqnarray}\label{res-bound}
	c\|\bar{u}-\uH\|_{\UsH} \leq \| {\delta\E(\uH)} \|_{(\UsH)^*} \leq C\|\bar{u}-\uH\|_{\UsH}.
	\end{eqnarray}
\end{proposition}

In light of this result we can focus entirely on the residual $\delta\E(\uH)$, which 
we express it in terms of the {\em residual forces}, implicitly defined by 
\begin{equation}
	\label{eq:forcenew}
	\< \delta\E(\uH), v \> = \sum_{\ell \in \L} f_{\ell}(\uH) \cdot v(\ell).
\end{equation}
Although the forces $f_\ell$ are not computable in practise, we will for now retain an idealised setting and assume we do have access to them.

To proceed, we define the {\em rescaled nodal interpolant}
\begin{eqnarray}\label{eq:QM_force}
	\hat{f}(\uH)(x):= \sum_{\ell\in\L} c_\ell f_{\ell}(\uH)\zeta_\ell(x),
	\qquad \text{where} \quad
	c_\ell := \frac{1}{\int_{\R^d} \zeta_\ell(x) \dx};
\end{eqnarray}
%
a continuum field $\hat{f} \in L^2(\R^d; \R^d)$ \cite{2012-lattint} representing the residual $\delta \E(\uH)$.
The rescaling through $c_\ell$ accounts for the fact that near defects the atoms are not arranged in a lattice and one needs to correct the ``volume'' assigned to them (see the proof of Lemma \ref{th:ctsphi}).
This allows us to obtain upper and lower bounds on $\| \delta\E(\uH) \|_{(\UsH)^*}$ in terms of the solution of a whole-space Poisson problem, which provides the starting point for the construction of our estimator.

\begin{lemma} \label{th:ctsphi}
	Up to a constant shift, there exists a unique $\phi \in \dot{H}^1(\R^d)$ such that
	\begin{equation} \label{eq:defnphi}
	    \int \nabla\phi \cdot \nabla v \,dx
	    = \int \hat{f} v \dx \qquad \forall v \in H^1(\R^d).
	\end{equation}
	Moreover, there exist constants $c_1, C_1$ such that
	\begin{equation}
		c_1 \| \delta \E(\uH) \|_{(\UsH)^*} \leq \| \nabla \phi \|_{L^2(\R^d)}
		\leq C_1 \| \delta \E(\uH) \|_{(\UsH)^*}.
		\label{eq:equiv}
	\end{equation}
\end{lemma}
\begin{proof}[Sketch of the proof]
    Let $\phi \in \dot{H}^1(\R^d)$ and $\phi_\a \in \UsH$ be the solutions of \eqref{eq:defnphi}. Testing, respectively, with $v\in \dot{H}^1(\R^d)$ and $v\in \UsH$, and employing Galerkin orthogonality it is straightforward to show that 
\begin{align*}
	\| \nabla \phi_\a \|_{L^2}
	\leq 
	\| \nabla \phi \|_{L^2}
	\leq \| \nabla \phi_\a \|_{L^2} 
	   +  \| \nabla \phi - \nabla \phi_\a \|_{L^2}
	\leq  \| \nabla \phi_\a \|_{L^2} + \| \hat{f} \|_{L^2},
\end{align*}
where the last inequality follows from standard elliptic regularity and finite element error estimates. Due to the discreteness of $\phi_\a$ and $\hat{f}$ we can then use suitable inverse estimates to establish bounds between $\| \hat{f} \|_{L^2}$, $\phi_\a$  and $\| f \|_{(\UsH)^*} = \|\delta\E(\uH)\|_{(\UsH)^*}$, which complete the proof. The details are given in Appendix~\ref{sec:apd_proof}.
\end{proof}
%


The potential $\phi$ is a Riesz representation of the residual $\delta \E(\uH)$, which is related to the stress error employed in \cite{2012-ARMA-cb, Wang:2017}, e.g., for constructing similar a posteriori estimators for A/C coupling schemes. We discuss this connection in more detail in \S~\ref{sec:stress}.

\subsection{Finite element Poisson solver}
\label{sec:discrete}
\def\RO{R_{\Omega}}
The idealised estimator $\|\nabla\phi\|_{L^2}$, derived in Lemma~\ref{th:ctsphi}, provides both upper and lower bounds for the residual, however, it cannot be computed because (i) the equation \eqref{eq:defnphi} cannot be solved explicitly; and (ii) the source term $\hat{f}(\uH)$ cannot be evaluated. To overcome this, we now discretise the Poisson problem~\eqref{eq:defnphi} and its source term: we truncate the infinite computational domain for $\phi$ to a finite domain $\Omega$; we approximate  $\hat{f}(\uH)$ with a coarse finite element interpolant which will require evaluating $f_\ell$ at few sites; and we discretise the Poisson problem with a finite element method.

To make this concrete, let $\Omega$ be a convex polygon or polyhedron in $\R^d$ with boundary $\Gamma = \partial\Omega$, chosen such that $\L^{\rm QM}, \L^{\rm MM} \subset \Omega$. Next, let  $\T^{\rm a}$ be the restriction of $\T$ to the QM and buffer region, i.e. the collection of triangles whose sites belong to $\L^{\rm QM} \L^{\rm BUF}$. We then extend $\T^{\rm a}$ with a {\em coarse partition} $\T^{\rm c}$, which we require to be a shape-regular simplicial partition of the region $\Omega\setminus \cup \T^{\rm a}$ whose nodes belong to $\L$. We denote the combined triangulation by $\T^0 := \T^{\rm a} \bigcup \T^{\rm c}$ and its set of nodes by $\N^0$. For efficiency we aim to have $\#\N^0 \ll\#(\L \cap \Omega)$. Figure~\ref{figs:plotMesh} provides an illustration of the triangulation $\T^0$. 

\begin{figure}
\begin{center}
	\includegraphics[scale=0.68]{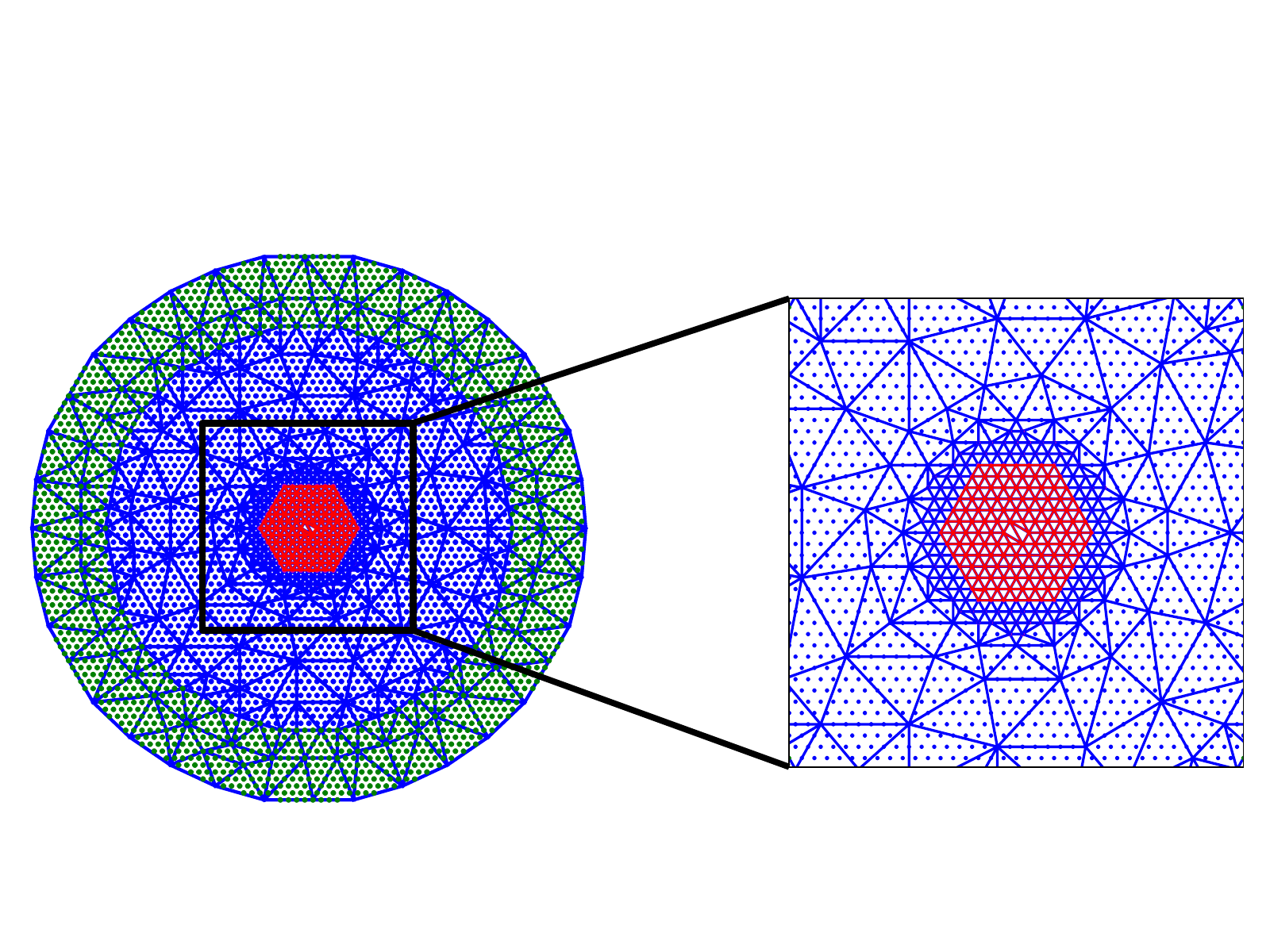}
	\caption{Illustration of the combined triangulation $\T^0$. Red and blue triangulations are $\T^\a$ and $\T^\c$ respectively. }
	\label{figs:plotMesh}
\end{center}
\end{figure}

Our next step is to replace the lattice interpolant $\hat{f}$ with an interpolant on the mesh $\T^0$. Interpolating $\hat{f}$ directly is not feasibly since $f_\ell$ is not computable. To overcome this we use the fact that the dependence of the force $f_{\ell}(\uH)$ on the environment decays exponentially fast; cf.~Lemma \ref{lemma-thermodynamic-limit}. Hence, we can approximate each residual force $f_{\ell}(\uH)$ in a finite size patch $B_{\rcut}(\ell)$, i.e., 
\begin{eqnarray}\label{eq:ftilde}
    \tilde{f}_\ell(\uH) := \frac{ \partial E^{B_{\rcut}(\ell)}(u)}{\partial u(\ell)}\Big|_{u = \uH},
\end{eqnarray}
committing an error that is exponentially small in $\rcut$. In practice, we choose $\rcut = 8 r_0$, where $r_0$ is interatomic spacing, a reliable value that was numerically found in \cite{CMAME}. 

Thus, we can now define the approximate residual 
\begin{equation}\label{eq:fhat_T0}
    \hat{f}_{\T^0}(\uH)(x):=\sum_{\ell \in \N^0} \tilde{f}_{\ell}(\uH)(\ell) \zeta^{\rm c}_{\ell}(x) + C_{\hat{f}_{\T^0}},
\end{equation}
where $C_{\hat{f}_{\T^0}}$ is chosen such that $\int_\Omega \hat{f}_{\T^0}(\uH)(x) \dx = 0$ and $\zeta^{\rm c}_{\ell} \in W^{1, \infty}(\R^d;\R)$ is a nodal basis function satisfying $\zeta^{\rm c}_{\ell}(\ell) = 1$ and $\zeta^{\rm c}_{\ell}(\ell')=0$ for all $\ell' \in \N^0 \setminus \{\ell\}$. We note that $\zeta^{\rm c}_{\ell}$ is chosen as the standard piecewise affine basis functions over $\T^0$, $\hat{f}_{\T^0}(\uH)(\ell) = \tilde{f}_\ell(\uH)(\ell) + C_{\hat{f}_{\T^0}}$ for $\ell\in\N^0$ and $\hat{f}_{\T^0}(\uH)(x) = 0$ for $x\in\R^d\setminus\Omega$. 

Given $\T^0$ and $\hat{f}_{\T^0}$, we then generate a refined triangulation $\Th$ of $\Omega$ by some adaptive refinement algorithm, which will be described in Section~\ref{sec:algo}. We denote the nodes of $\Th$ by $\N^{\rm h}$, and define the refined finite element space
\[
	\Uh:=\{ u_{\rm h} \in \mathcal{P}_1(\Th): u_\h = 0\text{ on }\Gamma \}.
\]
We can now obtain an approximation to the idealised estimator $\phi$ by solving for $\phi_{\rm h} \in \mathscr{U}_{\rm h}$ such that 
\begin{eqnarray}
\label{eq:stress_force_trun_T}
 \int_{\Omega}\nabla \phi_{\rm h}(\uH) \cdot \nabla v_{\rm h} \dx = \int_{\Omega} \hat{f}_{\T^0}(\uH) \cdot v_{\rm h} \dx \qquad \forall v_{\rm h} \in \mathscr{U}_{\rm h}.
\end{eqnarray}

\subsection{A posteriori error estimate for the Poisson problem.}
\label{sec:apost_phi}
To estimate the approximation error introduced by this discretization we particularly need to account for the truncation of the domain. We assume for the sake of technical convenience, that there exists a radius $\RO$ such that $B_{\RO} \subset\Omega\subset B_{2\RO}$; that is, $\Omega$ is approximately  a ball. This allows us to define a simple truncation operator, following \cite{2013-defects}, $T_{\RO}: \dot{H}^1(\R^d)\rightarrow H^1_0(\Omega)$, 
\begin{align}\label{eq:trmod}
T_{\RO}v(x) := \eta(x)\big(v(x)-a_{\RO}\big),
\qquad a_{\RO} = \mint_{B_{\RO}\setminus B_{\RO/2}}v(x) \dx
\end{align}
where $\eta$ is a $C^1$ cut-off function; $\eta = 1$ in $B_{\RO/2}$, $\eta = 0$ in $B_{\RO}^{\rm c}$ and $|\nabla \eta| \lesssim \RO^{-1}$. Following \cite{2013-defects}, for $v_{\RO} = T_{\RO} v$ and $a_{\RO}$ defined by \eqref{eq:trmod} we readily obtain the estimates
    \begin{align}\label{eq:trwgw}
	\| v - v_{\RO} - a_{\RO} \|_{L^2(\Omega)} &\lesssim \RO \| \nabla v - \nabla v_{\RO} \|_{L^2(\Omega \setminus B_{\RO})}, \qquad \text{and} \\
\label{eq:trwgw2}
	\| \nabla v - \nabla v_{\RO} \|_{L^2(\Omega)} &\lesssim  \| \nabla v \|_{L^2(\Omega \setminus B_{\RO/2})}.
   \end{align}

%

Applying this truncation operator in a suitable way and combining it with classical residual-based a posteriori error estimates~\cite{braess2007finite} we obtain the following {\it a posteriori} error estimate for $\phi_h$.

\begin{lemma}
\label{lem:approxEst_alt}
	Let $\phi$, $\phi_\h$ be, respectively, given by \eqref{eq:defnphi} and \eqref{eq:stress_force_trun_T},  then
\begin{align} \label{eq:combined_phi_phih_estimator_alt}
		\| \nabla \phi - \nabla \phi_{\rm h} \|_{L^2}
		 \leq &  \|\hat{f}-\hat{f}_{\T^0}\|_{(\dot{H}^1)^*} +  CR_\Omega \|\hat{f}_{\mathcal{T}^0}\|_{L^2(\Omega\setminus B_{R_\Omega/2})}
        +  C\|\nabla \phi_\h\|_{L^2(\Omega\setminus B_{R_\Omega/2})}\nonumber \\
		 &
		+ C \bigg( \sum_{T\in\T^\h} h^2_T \|\Delta \phi_\h+\hat{f}_{\T^0}\|^2_{L^2(T)} + \sum_{e\in \Gamma^\h}h_e\left\|\left[\frac{\partial \phi_\h}{\partial n}\right]\right\|^2_{L^2(e)}\bigg)^{1/2},
\end{align}
where $\Gamma^\h$ contains the edges of the elements $T\in\T^\h$ which lie in the interior of $\Omega$, $h_T$ denotes the diameter of $T$ and $h_e$ is the length of the edge $e$.
\end{lemma}

The a posteriori error estimate \eqref{eq:combined_phi_phih_estimator_alt} should be grouped into three components:
\begin{enumerate}
	\item The ``data oscillation'' $\|\hat{f}-\hat{f}_{\T^0}\|_{(\dot{H}^1)^*}$ arises due to the approximate evaluation and interpolation of the residual force. We claim that this term can be neglected in practise, and will give a detailed justification for this in Section \ref{sec:dataosc}.

	\item The group  $CR_\Omega \|\hat{f}_{\mathcal{T}^0}\|_{L^2(\Omega\setminus B_{\Omega/2})}
        +  C\|\nabla \phi_\h\|_{L^2(\Omega\setminus B_{R_\Omega/2})}$ gives an estimate for the error due to truncating the computational domain. 
        The first of the two terms will be analyzed together with the data oscillation term in  \ref{sec:dataosc}. Although the second term could simply be absorbed into the QM/MM a posteriori error estimate \eqref{eq:APET}, we will still retain it to control the size of the computational domain $\Omega$ for $\phi_\h$.

	\item The remaining group is the standard residual-based a posteriori error estimator on a finite domain, measuring how accurately $\phi_\h$ solves the poisson problem. We explain in Appendix \ref{sec:stableL2} why we expect that it can also be absorbed into $\| \nabla \phi_\h \|_{L^2}$, however, we also find in numerical experiments that a well-resolved estimator $\phi_\h$ provides significantly better estimates on the QM/MM model error, hence we keep the present form in our adaptive algorithm.
\end{enumerate}

\begin{proof}[Proof of Lemma \ref{lem:approxEst_alt}]
    By \eqref{eq:defnphi}, for $v\in\dot{H}^1$, we have
	\[
		\| \nabla \phi - \nabla \phi_{\rm h} \|_{L^2}
		 = \sup_{\| \nabla v \|_{L^2} = 1}
\int_{\R^d} \Big( \nabla \phi \cdot \nabla v - \nabla \phi_{\rm h} \cdot \nabla v \Big) \dx
		= \sup_{\| \nabla v \|_{L^2} = 1}
\int_{\R^d} \Big( \hat{f} \cdot v - \nabla \phi_{\rm h} \cdot \nabla v \Big) \dx.
	\]
Let $\hat{f}_{\T^0}$ be given by \eqref{eq:fhat_T0} and $v_R := T_{R_\Omega} v$ be given by \eqref{eq:trmod}, then we split the residual into four groups, 
\begin{align}
	\int_{\R^d} \Big( \hat{f} \cdot v - \nabla \phi_{\rm h} \cdot \nabla v  \Big) \dx
		=& \int_{\R^d}  \Big(\hat{f} - \hat{f}_{\mathcal{T}^0}\Big) \cdot v \dx + \int_{\R^d} \hat{f}_{\mathcal{T}^0} \cdot (v-v_R) \dx \nonumber \\
		& + \int_{\R^d} \Big(\hat{f}_{\mathcal{T}^0} \cdot v_R - \nabla \phi_{\rm h} \cdot \nabla v_R\Big)\dx
		   - \int_{\R^d} \nabla \phi_{\rm h} \cdot (\nabla v -\nabla v_R)\dx \nonumber \\
		=&: T_1 + T_2 + T_3 + T_4. 
\end{align}

The term  $T_1$ is simply estimated by
\begin{align}\label{eq:al2}
	T_1 = \int_{\R^d} (\hat{f} - \hat{f}_{\mathcal{T}^0}) \cdot v \dx
	\leq  \|\hat{f}-\hat{f}_{\T^0}\|_{(\dot{H}^1)^*} \|\nabla v\|_{L^2}.
\end{align}

As for the term $T_2$, we have
\begin{align}\label{eq:EstT2}
T_2 =& \int_{\R^d} \hat{f}_{\mathcal{T}^0} \cdot (v-v_R) \dx  \nonumber \\
=& \int_{\R^d} \hat{f}_{\mathcal{T}^0} \cdot (v-v_R-a_R) \dx + \int_{\R^d} \hat{f}_{\mathcal{T}^0} \cdot a_R \dx \nonumber \\
\leq& CR_\Omega \|\hat{f}_{\mathcal{T}^0}\|_{L^2(\Omega\setminus B_{R_\Omega/2})}\|\nabla v\|_{L^2}.
\end{align}
where the last inequality follows from \eqref{eq:trwgw} and the fact that $\int_{\Omega}\hat{f}_{T^0}\dx=0$, $B_{R_\Omega/2}$ is a ball that containing the region where $v=v_R$. 

Next, since $v_R \in \dot{H}_0^1(\Omega)$, $\hat{f}_{\mathcal{T}^0} = \phi_\h = 0$ in $\R^d\setminus\Omega$ and  \eqref{eq:trwgw},  $\|\nabla v_R\|_{L^2(\Omega)}\leq \|\nabla v\|_{L^2(\Omega)}$, we can estimate $T_3$ by the standard arguments of residual-based a posteriori error analysis (e.g., following \cite[Theorem 8.1, \S~III.8]{braess2007finite}) to obtain 
\begin{equation}
\label{eq:EstT3}
T_3	\leq
	C\bigg\{
			\sum_{T\in\T^\h} h^2_T \|\Delta \phi_\h+\hat{f}_{\T^0}\|^2_{L^2(T)}
			+
			\sum_{e\in\Gamma^\h} h_e  \left\|\left[\frac{\partial \phi_\h}{\partial n}\right]\right\|^2_{L^2(e)} \bigg\}^{1/2}
		\|\nabla v\|_{L^2}.
\end{equation}

Finally, applying \eqref{eq:trwgw2} to $T_4$, and using the fact that $\nabla v = \nabla v_R$ in $B_{R_\Omega/2}$ we have
\begin{align}\label{eq:EstT4}
T_4 = \int_{\R^d} \nabla \phi_{\rm h} \cdot (\nabla v -\nabla v_R)\dx \leq C\|\nabla \phi_\h\|_{L^2(\Omega\setminus B_{R_\Omega/2})}\|\nabla v\|_{L^2}.
\end{align}

Combining \eqref{eq:al2}, \eqref{eq:EstT2}, \eqref{eq:EstT3} and \eqref{eq:EstT4}, we obtain the stated result.
\end{proof}

\subsection{QM/MM a posteriori error estimate}
We are now in the position to define the \emph{approximate error estimator} of a solution $\uH$ of the QM/MM scheme~\eqref{problem-e-mix}, by
\begin{equation}\label{eq:APET}
	\eta_\h(\uH) := ||\nabla \phi_{\h}(\uH)||_{L^2(\Omega)},
\end{equation}
where $\phi_\h(\uH)$ is the solution to \eqref{eq:stress_force_trun_T}. The quality of the estimator $\phi_\h$ is characterized, in Lemma~\ref{lem:approxEst_alt}, by
\begin{align}\label{eq:rhoh_alt}
\rho^2_\h(\uH) :=& \sum_{T\in\T^\h} h^2_T \|\Delta \phi_\h+\hat{f}_{\T^0}\|^2_{L^2(T)} + \sum_{e\in \Gamma^\h}h_e\left\|\left[\frac{\partial \phi_\h}{\partial n}\right]\right\|^2_{L^2(e)} \nonumber \\
		  &+ R_\Omega^2 \|\hat{f}_{\mathcal{T}^0}\|_{L^2(\Omega\setminus B_{R_\Omega/2})}^2
        +  \|\nabla \phi_\h\|_{L^2(\Omega\setminus B_{R_\Omega/2})}^2.
\end{align}
%
%

We summarize the results of the foregoing sections in the following main theorem, demonstrating the equivalence of the \emph{idealised residual estimate} $\|\delta \E(\uH)\|_{(\UsH)^*}$ and the \emph{approximate estimator} $\eta_\h(\uH)$. The equivalence constants are determined by an {\em oscillation factor},
\[
		F_{\rm osc} :=
		\frac{\| \hat{f} - \hat{f}_{\T^0} \|_{(\dot{H}^1)^*}}{ \| \hat{f} \|_{(\dot{H}^1)^*}},
\]
set to $F_{\rm osc} = 0$ if $\| \hat{f} \|_{(\dot{H}^1)^*} = 0$, which determines how well $\hat{f}_{\T^0}$ approximates the true residual.

\begin{theorem} \label{th:mainresult}
	There exists constants $c, C$ such that
	\begin{equation} \label{eq:main:first}
		c (1 + F_{\rm osc})^{-1} \eta_\h(\uH)
		\leq
		\| \delta \E(\uH) \|_{(\UsH)^*}
		\leq
		C \Big( (1 + F_{\rm osc}) \eta_\h(\uH)
					+ \rho_{\rm h}(\uH) \Big),
	\end{equation}
\end{theorem}
\begin{proof}
From Lemma \ref{th:ctsphi}, we already know that $ \| \delta \E(\uH) \|_{(\UsH)^*} \eqsim \| \nabla \phi \|_{L^2}$. Moreover, we recall that $\| \nabla \phi \|_{L^2} =  \| \hat{f} \|_{(\dot{H}^1)^*}$. Since $\phi$ and $\phi_\h$ are the solutions of \eqref{eq:defnphi} and \eqref{eq:stress_force_trun_T}, respectively, we use Galerkin orthogonality to write
\begin{align}
    \label{eq:phi-phih}
    \| \nabla \phi \|^2_{L^2} - \| \nabla \phi_{\h}\|^2_{L^2}  =& \<\nabla \phi - \nabla \phi_{\h}, \nabla \phi - \nabla \phi_{\h}\> + 2\<\nabla \phi - \nabla \phi_{\h}, \nabla \phi_\h\> \nonumber \\
    =& \| \nabla \phi - \nabla \phi_{\h}\|^2_{L^2} + 2\<\hat{f}-\hat{f}_{\T^0}, \phi_\h\>, \qquad \text{where} \\ 
    \notag 
	\big|\<\hat{f} - \hat{f}_{\T^0}, \phi_\h\> \big| \leq& \| \hat{f} - \hat{f}_{\T^0} \|_{(\dot{H}^1)^*}  \|\nabla \phi_\h \|_{L^2}
	\leq F_{\rm osc} \| \nabla\phi \|_{L^2} \|\nabla \phi_\h \|_{L^2}.
\end{align}

To obtain an upper bound for $\| \nabla \phi \|_{L^2}$ we use Cauchy's inequality to estimate
\begin{align*}
	\| \nabla \phi \|^2_{L^2}
	&\leq
	\| \nabla \phi_{\h}\|^2_{L^2}
	+
	\| \nabla \phi - \nabla \phi_{\h}\|^2_{L^2}
	+
	2F_{\rm osc} \| \nabla \phi \|_{L^2} \|\nabla \phi_\h \|_{L^2}   \\
	&\leq
	\big(1 + F_{\rm osc}^2\big) \| \nabla \phi_{\h}\|^2_{L^2}
	+
	\| \nabla \phi - \nabla \phi_{\h}\|^2_{L^2}
	+
	{\textstyle \frac14} \|\nabla \phi \|_{L^2}^2.
\end{align*}
Rearranging and applying Lemma~\ref{lem:approxEst_alt} we deduce
\begin{equation} \label{eq:proofmain:upperbound}
	{\textstyle \frac34} \| \nabla \phi \|^2_{L^2}
	\leq
	\big(1 + F_{\rm osc}^2\big) \eta_\h^2
	+
	C \rho_\h^2.
\end{equation}
This establishes the upper bound.

To obtain an lower bound we can use an analogous argument. Starting again from \eqref{eq:phi-phih} we have
\begin{align*}
	\| \nabla \phi_{\h}\|^2_{L^2}
	& \leq
	\| \nabla \phi \|^2_{L^2}
	+
	F_{\rm osc} \| \nabla \phi \|_{L^2} \|\nabla \phi_\h \|_{L^2} \\
	&\leq
	(1 + F_{\rm osc}^2) \| \nabla \phi \|^2_{L^2}
	+
	{\textstyle \frac{1}{4}} \|\nabla \phi_\h \|_{L^2}^2,
\end{align*}
which can be rearranged to yield
\begin{equation} \label{eq:proofmain:lowerbound}
	(1 + F_{\rm osc}^2)^{-1} \eta_\h^2 \leq
	{\textstyle \frac{4}{3}} \| \nabla \phi \|^2_{L^2}.
\end{equation}
Noting that $1+F_{\rm osc}^2 \eqsim (1 + F_{\rm osc})^2$, and combining \eqref{eq:proofmain:upperbound}, \eqref{eq:proofmain:lowerbound} and Lemma \ref{th:ctsphi} we obtain the stated result.
\end{proof}

\section{Implementation and Numerical Tests}
\label{sec:numer}
\setcounter{equation}{0}
\def\Mqm{\mathcal{M}_{\rm QM}}
\def\Mmm{\mathcal{M}_{\rm MM}}
We now describe an adaptive QM/MM algorithm leveraging the model error estimator $\eta_\h(\uH)$, and
present numerical examples for point defects and an edge dislocation.



\subsection{Adaptive Algorithms}
\label{sec:algo}
\def\dist{\textrm{dist}}
\def\rhoth{\rho_{T}^{\rm h}}
\def\rhoh{\rho^{\rm h}}
\def\rhO{\rho^{\Omega}}
\def\fT{\hat{f}_{\T^0}}
We propose a two-layer adaptive strategy, consisting of an {\em outer} Algorithm \ref{alg:main} driving the QM/MM model selection and an {\em inner} Algorithm~\ref{alg:adaptMesh} to compute the estimator $\phi_\h$. Both algorithms follow the established SOLVE-ESTIMATE-MARK-REFINE loop~\cite{Dorfler1996}. To choose where to refine the model we split the estimator $\eta_\h$ into local contributions, 
\begin{equation}
    \label{eq:rho}
    \eta_{\h,T} := \frac{\|\nabla \phi_\h(\uH) \|_{L^2(T)}}{\eta_\h(\uH)}, \quad \forall T\in\Th,
\end{equation}
such that $\sum_{T\in\Th} \eta^2_{\h,T} = \eta_\h^2(\uH)$. We begin by describing the outer algorithm, following by a detailed discussion of the individual steps.

\begin{algorithm}[H]
\caption{Adaptive QM/MM algorithm}
\label{alg:main}
{\bf Prescribe} $\LQM, \LMM$, termination tolerance $\eta_{\rm tol}$, refinement tolerance $\tau_{\rm D}$.

\begin{algorithmic}[1]
\REPEAT
	\STATE{ \textit{Solve}: Solve \eqref{problem-e-mix} to obtain $\uH$. }
	\STATE{  \textit{Estimate}: Apply Algorithm \ref{alg:adaptMesh} to compute $\eta_\h(\uH)$ and $\eta_{\h,T}$ (cf. \eqref{eq:APET}, \eqref{eq:rho}). } 		
	\STATE{\textit{Mark}: Use D\"{o}rfler strategy with $\tau_{D}$ to mark elements for refinment.}
	\STATE{ \textit{Refine:} Construct new $\LQM$ and $\LMM$ regions.}
\UNTIL{$\eta_\h(\uH) < \eta_{\rm tol}$}
\end{algorithmic}
\end{algorithm}

The {\it Solve} step requires no further comments, while the {\it Estimate} step is the subject of Algorithm \ref{alg:adaptMesh} below. We therefore discuss the {\it Mark} and {\it Refine} steps first:


%
{\bf Mark.} We employ D\"{o}rfler's strategy~\cite{Dorfler1996}, which is a widely used marking strategy to ensure error reduction. Given $0 < \tau_{\rm D} < 1$, we construct the minimal set $\Mrefine \subset \L$ such that the following D\"{o}rfler properties are satisfied: $\sum_{T\subset \Mrefine}\eta_{\h,T} \geq \tau_{D} \sum_{T\in \mathcal{T}^\h} \eta_{\h,T}$ and mark all the sites in $\Mrefine$. The default parameter $\tau_{\rm D} = 0.3$ is used in all experiments reported below.

\def\dist{\mathrm{dist}}

{\bf Refine.} Once we have marked elements for refinement, we must construct a ``refined'' QM/MM partitioning. We present a simple strategy that has worked well in all our tests, but has restrictions that we discuss in detail in the Conclusion. 
We divide the marked elements into two subsets $\Mqm$ and $\Mmm$. The set $\Mqm$ contains those elements connected with the QM/MM interface by a path whose elements all belong to $\Mqm$. The remaining elements belong to the subset denoted as $\Mmm$. 
We define $\dist: \L\to\R$ to be the function mapping the atoms to their distances to the defect core $B_{\Rcore}$, and let the new QM region has a radius $\RQM = \max_{\ell\in \Mqm\cap \L} \dist(\ell)$. See Figure \ref{figs:refi} for an illustration. 
If $\Mmm\neq \emptyset$, then we analogously enlarge the MM region to absorb elements of $\Mmm$ into the MM region. 


{\bf Estimate:}  Finally we turn towards the details of the {\em Estimate} step. 
As indicated by Lemma~\ref{lem:approxEst_alt} we construct an estimator $\eta_\h$ controlling its accuracy using an adaptive finite element method; specifically we prescribe a relative tolerance $\tau_{\rm est} > 0$ and require as a termination criterion that 
\[
	\rho_\h \leq \tau_{\rm est} \| \nabla \phi_\h \|_{L^2}.	
\]
To drive the mesh refinement we define the truncation error indicator,
\begin{align}\label{eq:rho_omega}
	\rho_{\h,\Omega} := R_\Omega \|\hat{f}_{\mathcal{T}^0}\|_{L^2(\Omega\setminus B_{R_\Omega/2})} + \|\nabla \phi_\h\|_{L^2(\Omega\setminus B_{R_\Omega/2})},
\end{align}
which controls adaption of the computational domain $\Omega$, and the local residual indicator,
\begin{align}
\label{eq:residual_estimate}
\rho_{\h,T} := \left[ h^2_{T}\|\Delta \phi_\h(\uH) + \fT(\uH) \|^2_{L^2(T)} + \frac{1}{2}\sum_{e \subset \partial T} h_{e}\Big\|\left[\frac{\partial \phi_\h(\uH)}{\partial n}\right]\Big\|^2_{L^2(e)}\right]^{\frac{1}{2}},  \quad \forall T \in \T^{k},
\end{align}
where $\T^{k}$ denotes the triangulation of the $k$-th loop of the following algorithm, and the initial configuration $\T^{0}$ is introduced in Section \ref{sec:discrete}. This leads to the following adaptive finite element algorithm, which completes the specification of our adaptive QM/MM scheme.

\begin{algorithm}[H]
\caption{Adaptive algorithm for the estimator $\eta_{\rm h}(\uH)$}
\label{alg:adaptMesh}
{\bf Prescribe} $\uH$, $\Omega$, $\T^0$, $k = 0$,  $\tau_{\rm est}$, refinement parameter $\tau_{\rm D}$ $\qquad\qquad\qquad$.

\begin{algorithmic}[1]
\STATE{ Compute $\fT(\uH)$ according to \eqref{eq:fhat_T0}.} 
\REPEAT
	\STATE{ \textit{Solve}:  Solve \eqref{eq:stress_force_trun_T} on $\T^k$ for $\phi_\h(\uH)$ and compute $\eta_\h(\uH) = \|\nabla \phi_\h(\uH)\|_{L^2(\Omega)}$.} 
	\STATE{  \textit{Estimate}: Compute $\rho_{\h,\Omega}$ and $\rho_{\h,T}$ from \eqref{eq:rho_omega} and \eqref{eq:residual_estimate}.}
	\STATE{ Compute $\rho_{\h, \T} = \sum_T \rho_{\h,T}$ and $\rho_\h = \rho_{\h, \T} + \rho_{\h, \Omega}$.}
	\IF{$\rho_{\h, \T} \leq \rho_{\h, \Omega}$}
        \STATE{Increase $R_{\Omega}$ to $(1+\theta)R_\Omega$, expand $\T^0$ by constructing mesh for the incremental domain; let $\hat{f}_{\T^0}(\uH)(\ell)=0$ for $\ell$ belonging to the incremental mesh.}
	\ENDIF
	\STATE{ \textit{Mark}: Use D\"{o}rfler strategy with parameter $\tau_{\rm D}$ to mark elements for refinement.}
	\STATE{ \textit{Refine:} Bisect the selected elements to generate a new mesh $\T^{k+1}$.}
	\STATE{ Set $k \leftarrow k+1$.}
\UNTIL{$\rho_\h < \tau_{\rm est} \eta_\h$}
\end{algorithmic}
\end{algorithm}

According to \cite[Theorem 2]{2013-defects}, 
the truncation error $\rho_{\h,\Omega}(R)$ is approximately $CR^{-d/2}$. If we increase $R$ to $(1+\theta)R$, 
\begin{equation*}
	\rho_{\h,\Omega}((1+\theta)R)\simeq \frac{C\big((1+\theta)R\big)^{-d/2}}{CR^{-d/2}}CR^{-d/2}= (1+\theta)^{-d/2}\rho_{\h,\Omega}(R).
\end{equation*}
We can take $\theta=\Theta^{-d/2}-1$ in order to reduce $\rho_{\h,\Omega}(R)$ to $\Theta\rho_{\h,\Omega}(R)$ for $0<\Theta<1$.


\subsection{Test problems.}
\label{sec:numexp}
Our numerical tests will be performed with a tight-binding toy model that retains the qualitative properties of more realistic tight-binding models, but enables rapid experimenting on large computational domains with limited resources. The Hamiltonian is given by \eqref{tb-H-elements}, with  the onsite term is $h_{\rm ons} = 0$, and the hopping term given by the Morse potential \cite{morse29} (scaled to have a minima at $r = 1$),
$
	h_{\rm hop} (r) =  e^{-2\alpha(r-1)} - 2e^{-\alpha(r-1)},
$
with $\alpha=2.0$, which is the same model as was used in the numerical results in \cite{chen15b, CMAME}.

\begin{figure}
    \label{fig:meshes}
	\centering
	\subfigure[Point defect]{
		\label{fig:pdmesh}
		\includegraphics[height=3.5cm]{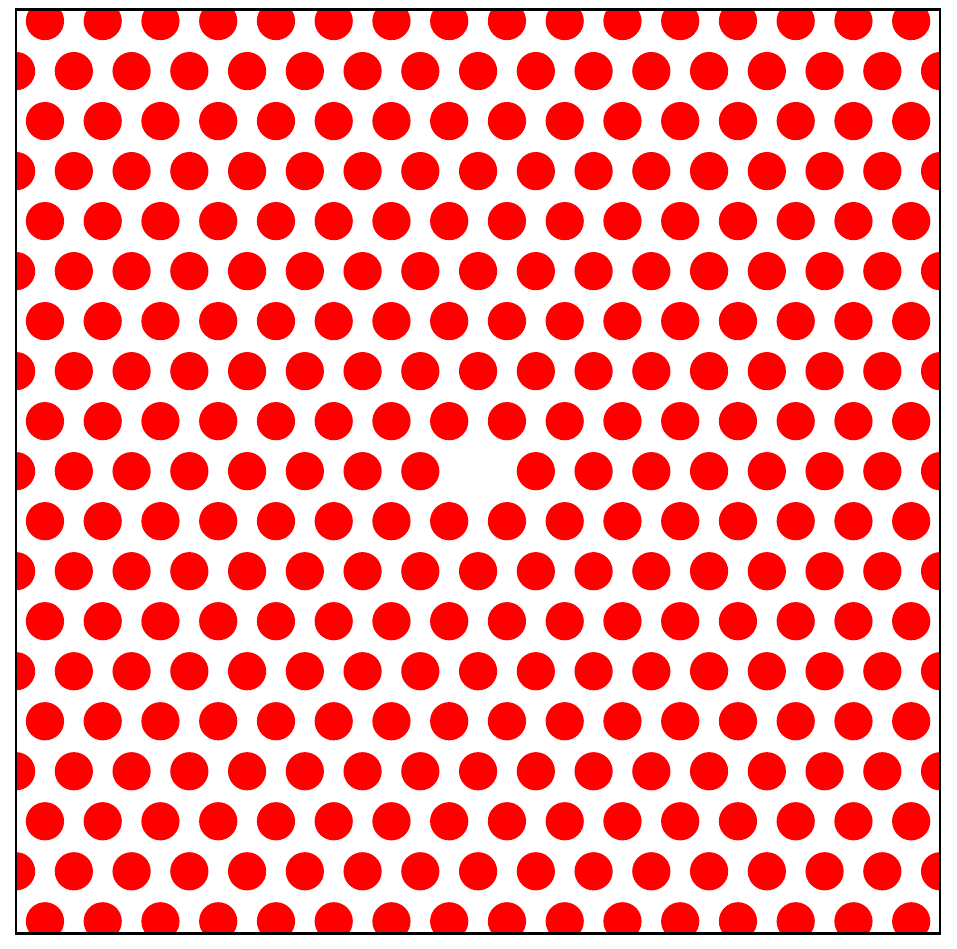}}
	\hspace{0.6cm}
	\subfigure[Micro-crack]{
		\label{fig:mcmesh}
		\includegraphics[height=3.5cm]{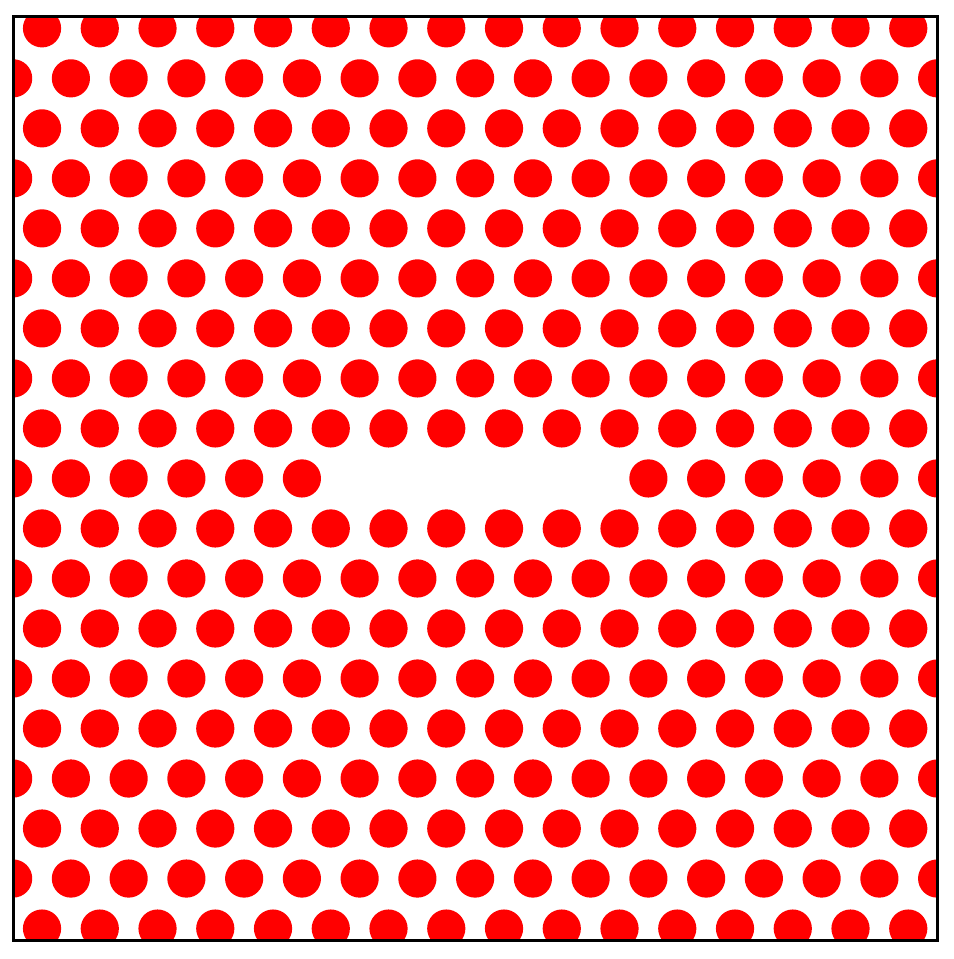}}
	\hspace{0.6cm}
	\subfigure[Edge-dislocation]{
		\label{fig:edmesh}
		\includegraphics[height=3.5cm]{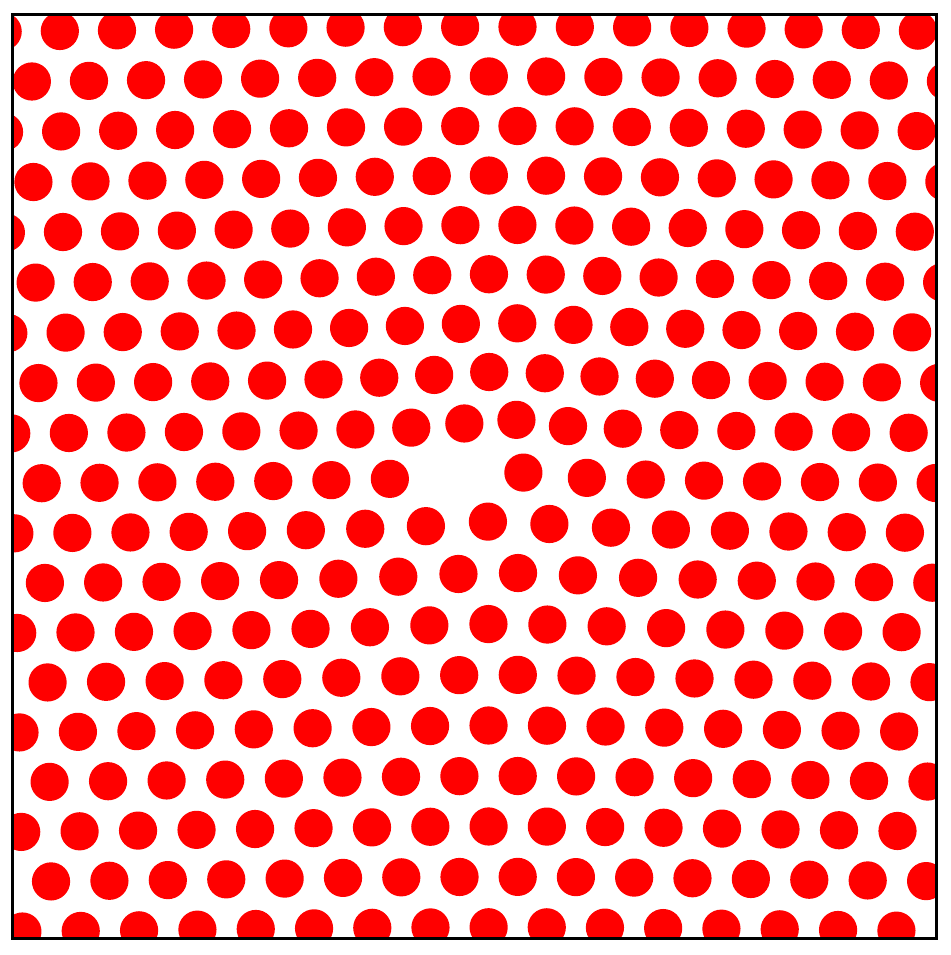}}
	\caption{Defect cores for the three test cases introduced in section \S~\ref{sec:numexp}, serving as benchmark problems for the numerical tests.}

\end{figure}

We consider three prototypical examples of localised defects; their core geometries are visualised in Figure~\ref{fig:meshes}:
\begin{itemize}
	\item {\it Point defect:} (Fig.~\ref{fig:pdmesh}) a single vacancy located at the origin; defined by $\L := \Lhom\backslash\{\pmb 0\}$; 
	\item {\it Micro-crack:} a row of five adjacent vacancies; while this is not technically a ``crack'', it serves as an example of a localised defect with an anisotropic shape;
	\item {\it Edge-dislocation:} a straight edge dislocation with dislocation line orthogonal to the plane (see~\S~\ref{sec:dislocation}); this is a paradigm example of a topological defect with long-range elastic field.
\end{itemize}

\subsection{Adaptive algorithm study}
\label{sec:numres}
In this section, we perform a detailed study of the behaviour of our adaptive algorithm for the point defect case. Analogous studies for the other two cases obtain very similar results; see Appendix \ref{apped:ns}. Applying our adaptive algorithm to the QM/MM coupling method for the point defect results in the QM/MM decomposition given in Figure \ref{fig:pdgeom}.

\paragraph{Algorithm \ref{alg:main}, {\bf Mark} and {\bf Refine} steps:} 
Figure \ref{figs:refi} visualises a prototypical {\bf Mark} step in Algorithm \ref{alg:main}, highlighting the marked elements. We observe that only element close to the QM/MM interface and close to the MM/far-field interface are marked for refinement (i.e., model refinement or domain enlargement). 
The marked elements close to the QM/MM interface (top-right) belong to $\Mqm$, while the remaining elements close to the MM/FF interface belong to $\Mmm$. The {\bf Refine} step sets updates $\RQM = \max_{\ell\in \Mqm\cap \L} \dist(\ell)$ and $\RMM = \max_{\ell\in \Mmm\cap \L} \dist(\ell)$ respectively.

\begin{figure}
	\begin{center}
		\subfigure[QM/MM Decomposition]{
			\label{fig:pdgeom}
			\includegraphics[height=4.5cm]{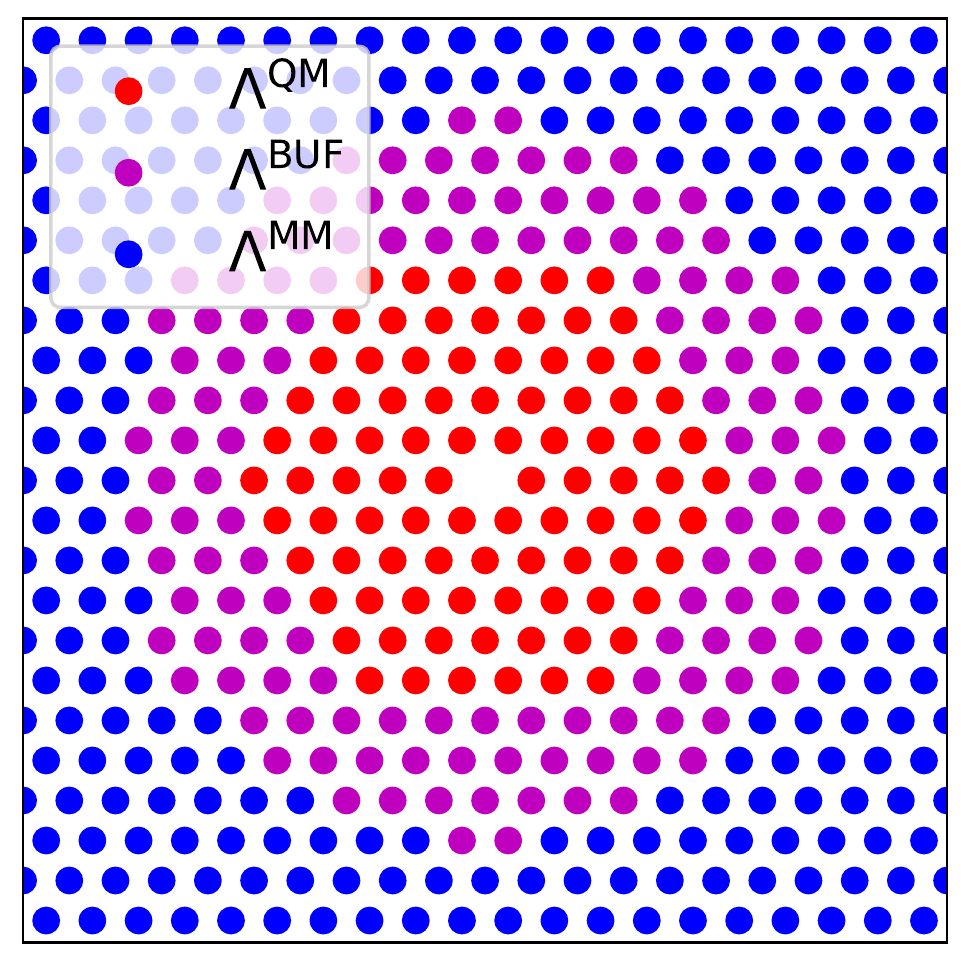}
		}
		\hspace{0.5cm} 
		\subfigure[Mark-Refine Steps]{
			\label{figs:refi}
			\includegraphics[height=4.5cm]{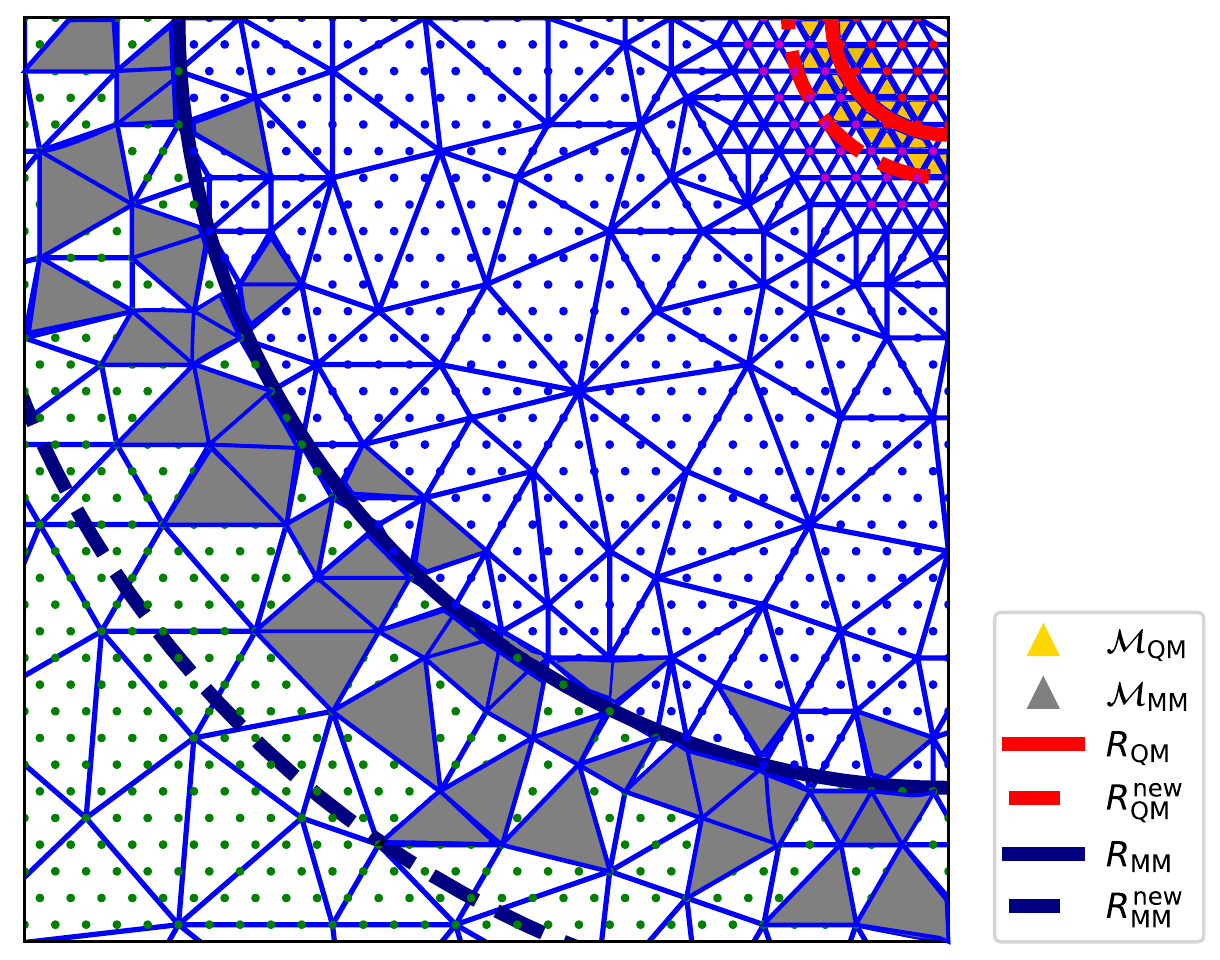}
		}
		\caption{(a) Adaptively constructed QM/MM decomposition for the vacancy defect example. (b) Illustration of $\Mqm$ and $\Mmm$: The elements marked gray close to the QM/MM interface (top-right) belong to $\Mqm$, and the other elements belong to $\Mmm$.}
	\end{center}
\end{figure}

\paragraph{Convergence of Algorithm~\ref{alg:adaptMesh}:} Next, we study the convergence of Algorithm~\ref{alg:adaptMesh}, which computes the estimator $\phi_\h$. First, in Figure \ref{fig:meshsteps} we visualise the evolution of the finite element mesh during refinement; observing that the mesh becomes more and more concentrated around the QM/MM and MM/FF interfaces.

\begin{figure}
	\centering
	\includegraphics[height=3.0cm]{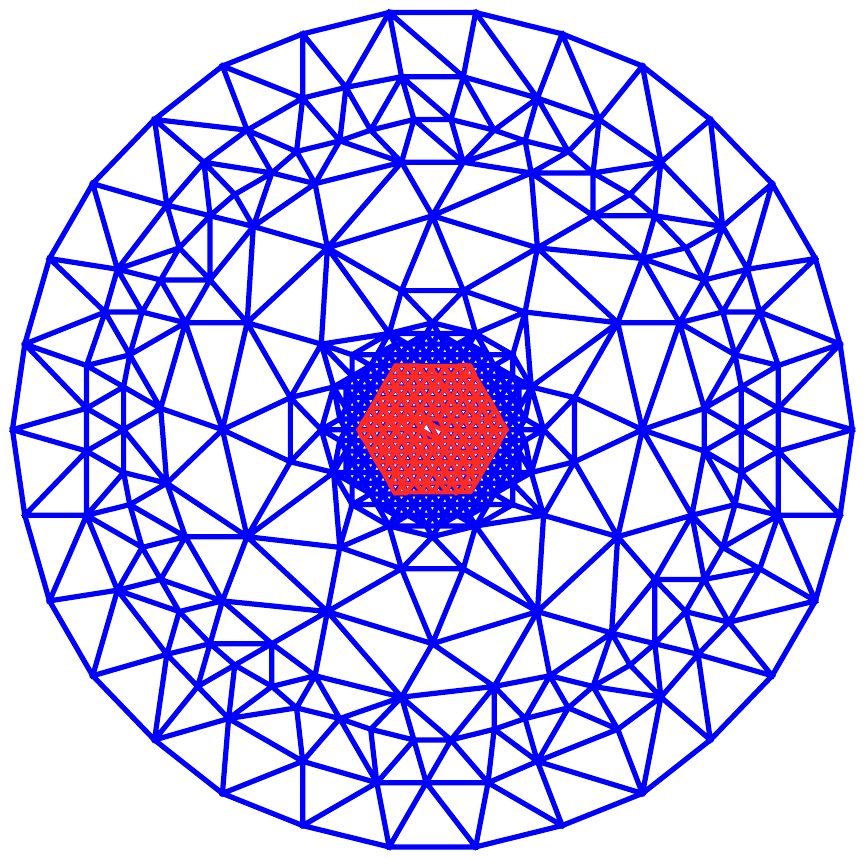}
	\includegraphics[height=3.0cm]{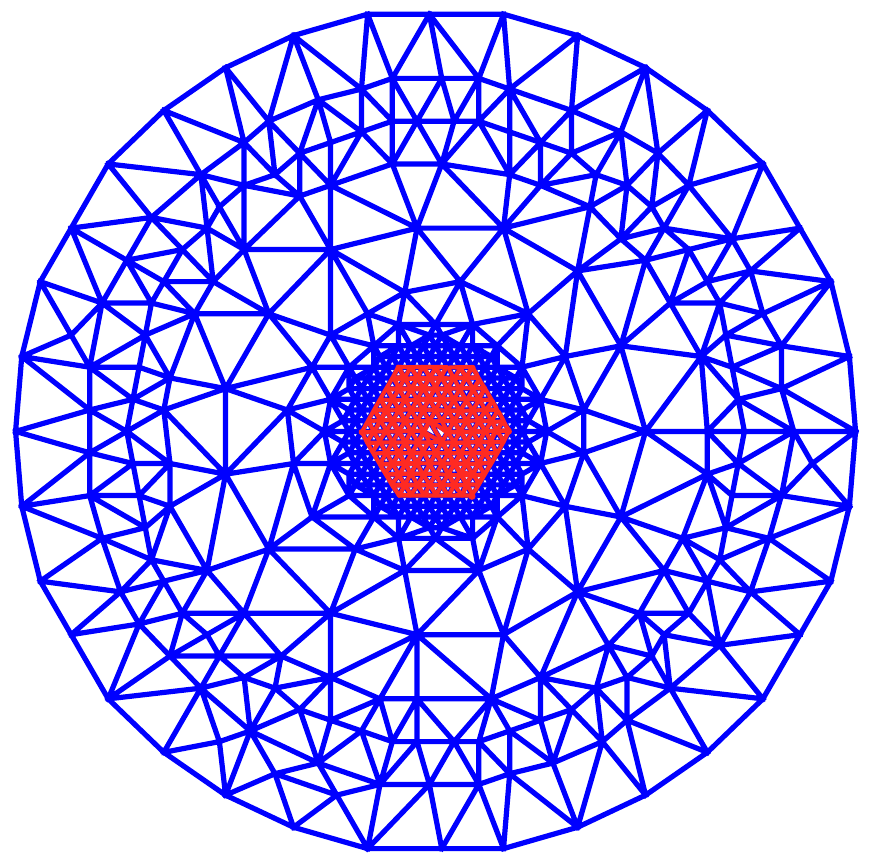}
	\includegraphics[height=3.0cm]{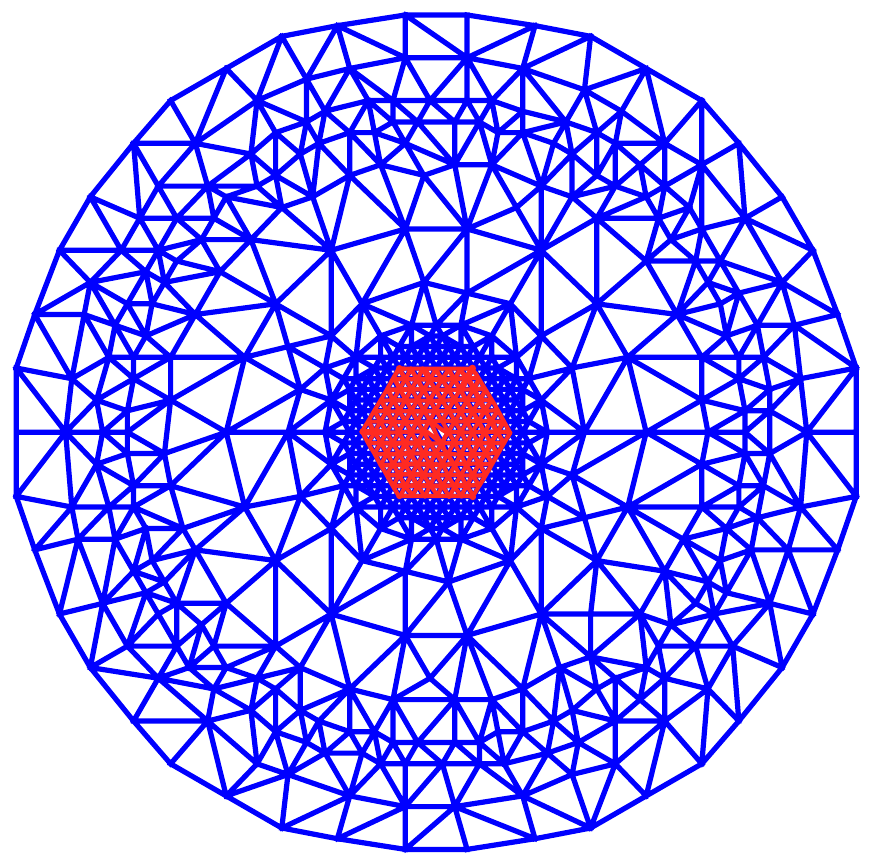}
	\includegraphics[height=3.0cm]{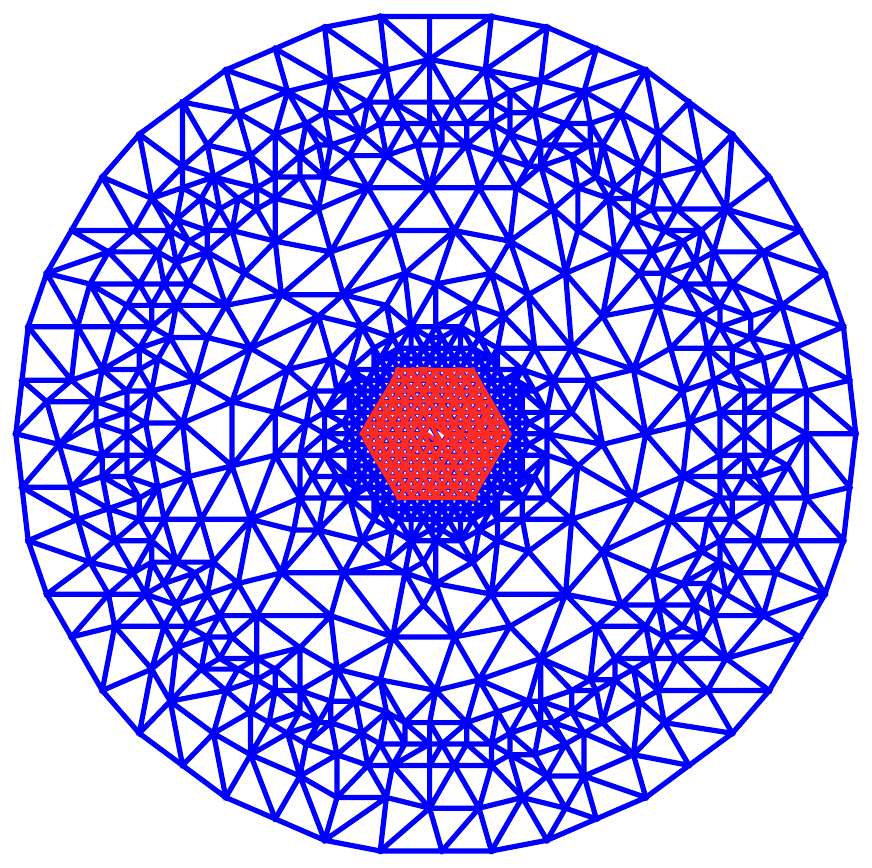}
	\caption{Evolution of QM/MM configuration during Algorithm \ref{alg:adaptMesh}.}
	\label{fig:meshsteps}
\end{figure}

\begin{figure}
	\centering
	\subfigure[Relative error $\rho_{\h,\Omega}/\eta_\h$ with increasing $R_\Omega$ while fixing $\RQM$ and $\RMM$.]{
		\label{fig:trun}
		\includegraphics[height=5cm]{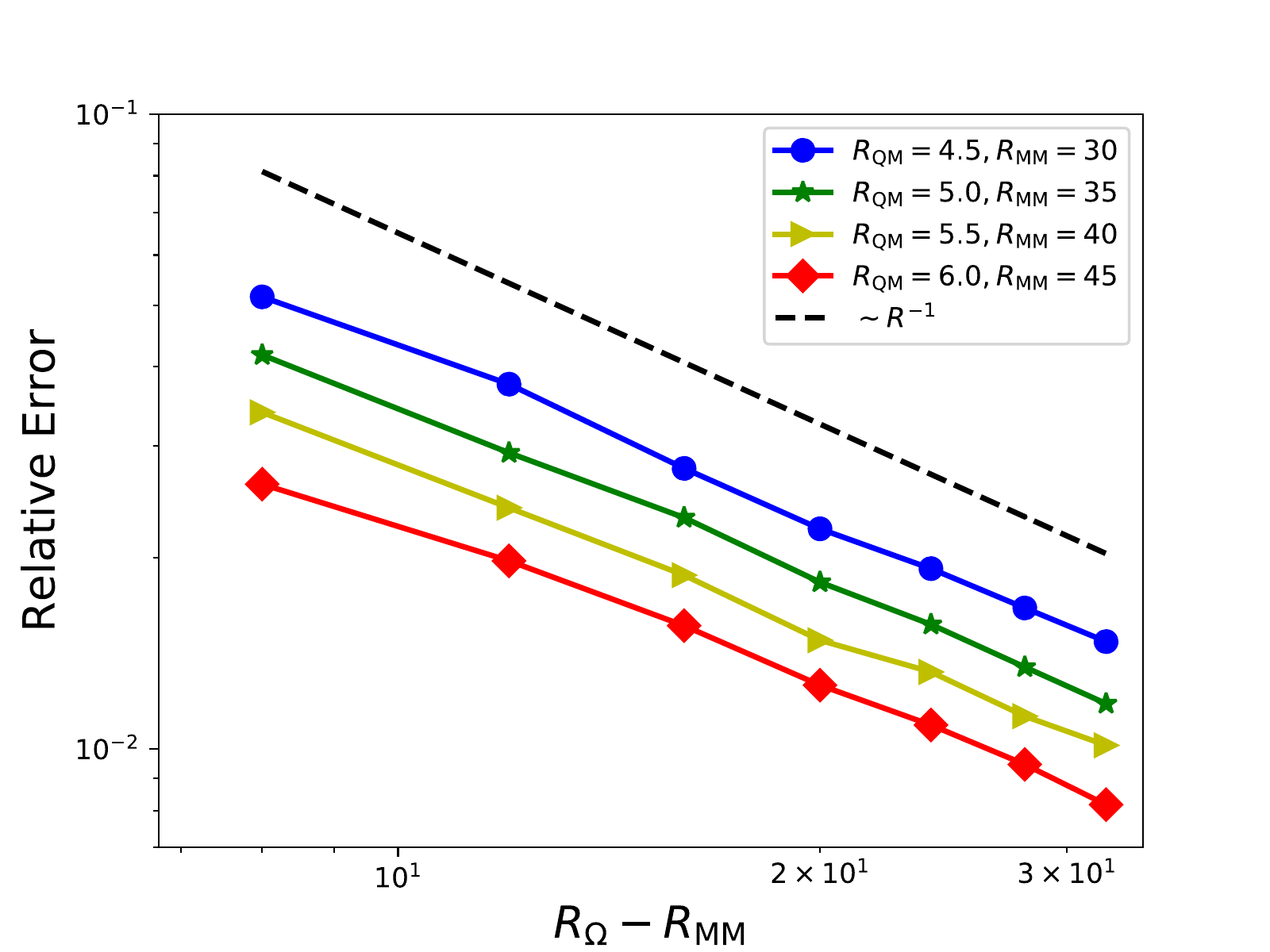}}
	\hspace{0.2cm}
	\subfigure[The values of residual indicator $\rho_{\h, \T}$, the estimator $\eta_{\h}$ and the data-oscillation during the mesh refinement.]{
		\label{fig:cg}
		\includegraphics[height=5cm]{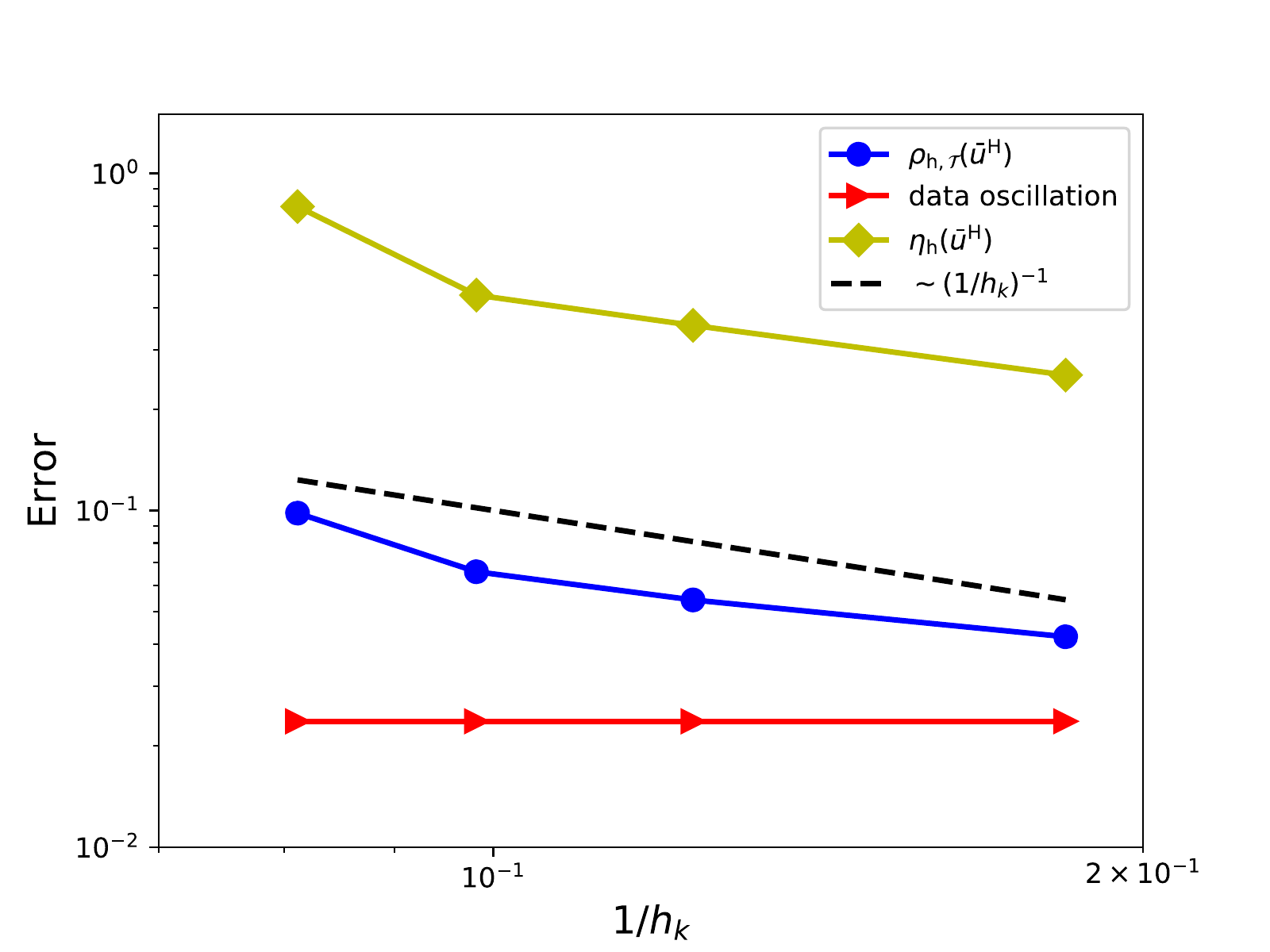}}
	\caption{Comparison of the truncation indicator, the residual indicator and the data-oscillation to the estimator. The data-oscillation error remains approximately constant since the mesh $\T^0$ is not refined.}
\label{fig:eoet}
\end{figure}

In Figure~\ref{fig:trun}, we compare the truncation error and the discretization error with respect to the error estimator $\eta_\h(\uH)$. We clearly observe the expected behaviour that this ratio tends to zero as we increase the radius of the computational domain for $\phi_\h$. Thus we can be confident to reduce the truncation error for $\phi_\h$ by increasing $R_\Omega$ while keeping $R_{\rm MM}$ fixed, as we did in Algorithm \ref{alg:adaptMesh}. 


Figure \ref{fig:cg} describes the evolution of several error indicators as the mesh is refined, indicated by $h_k:=\max\{\text{diam}(T), T\in \T^{k}\}$. 
We observe that the error indicator decays roughly linearly in $h_k$, which is consistent with the use of a linear finite element scheme. Secondly, we observe that the estimator $\rho_\h$ is significantly  smaller (on the order $15\%$) then the estimator $\eta_\h$ which clearly indicates that its contribution to the QM/MM model error can be neglected. 
Finally, we observe that the ``data-oscillation'', i.e., the approximation of $\hat{f}$ by $\hat{f}_{\T^0}$ is even smaller and may therefore also be neglected. In particular this provides a numerical verification of our analysis in \S~\ref{sec:dataosc}. In combination, these observations suggest that the QM/MM a posteriori error estimator is both efficient and reliable, in theory as well as in practise.

\paragraph{The estimator tolerance $\tau_{\rm est}$:} Finally, we study the sensitivity of the estimator $\eta_\h$ to the refinement tolerance $\tau_{\rm est}$, which is the primary input parameter into Algorithm \ref{alg:adaptMesh}. To that end we wish to compare $\eta_\h$ against the idealised estimator $\|\nabla\phi\|_{L^2}$. As it is not computable we compare instead against an estimator $\eta_{\rm a}$ computed analogously to $\eta_\h$ but where $\T^\h = \T^0$ coincide with the atomistic mesh. That is, the only error that remains is the domain truncation error which we have already shown to be small compared to the discretisation error.

While {\em in theory} we have found that $\eta_\h$ provides a reliable and efficient bound, we have found that in practise it is important to obtain an finer resolution to obtain an accurate estimate on the model error. Figure \ref{fig:Singleape_difftau} shows that, with $\tau_{\rm est} = 1.0$, the accuracy of the discretized Poisson solver effects the adaptive process significantly. The remaining panels in Figure~\ref{fig:Singleape_difftau} demonstrate that only mild reductions in $\tau_{\rm est}$ lead to excellent agreement between the coarse and idealised estimators.

\begin{figure}
	\centering
	\subfigure[$\tau_{\rm est}=1.0$]{
		\label{fig:tau1}
		\includegraphics[height=3.4cm]{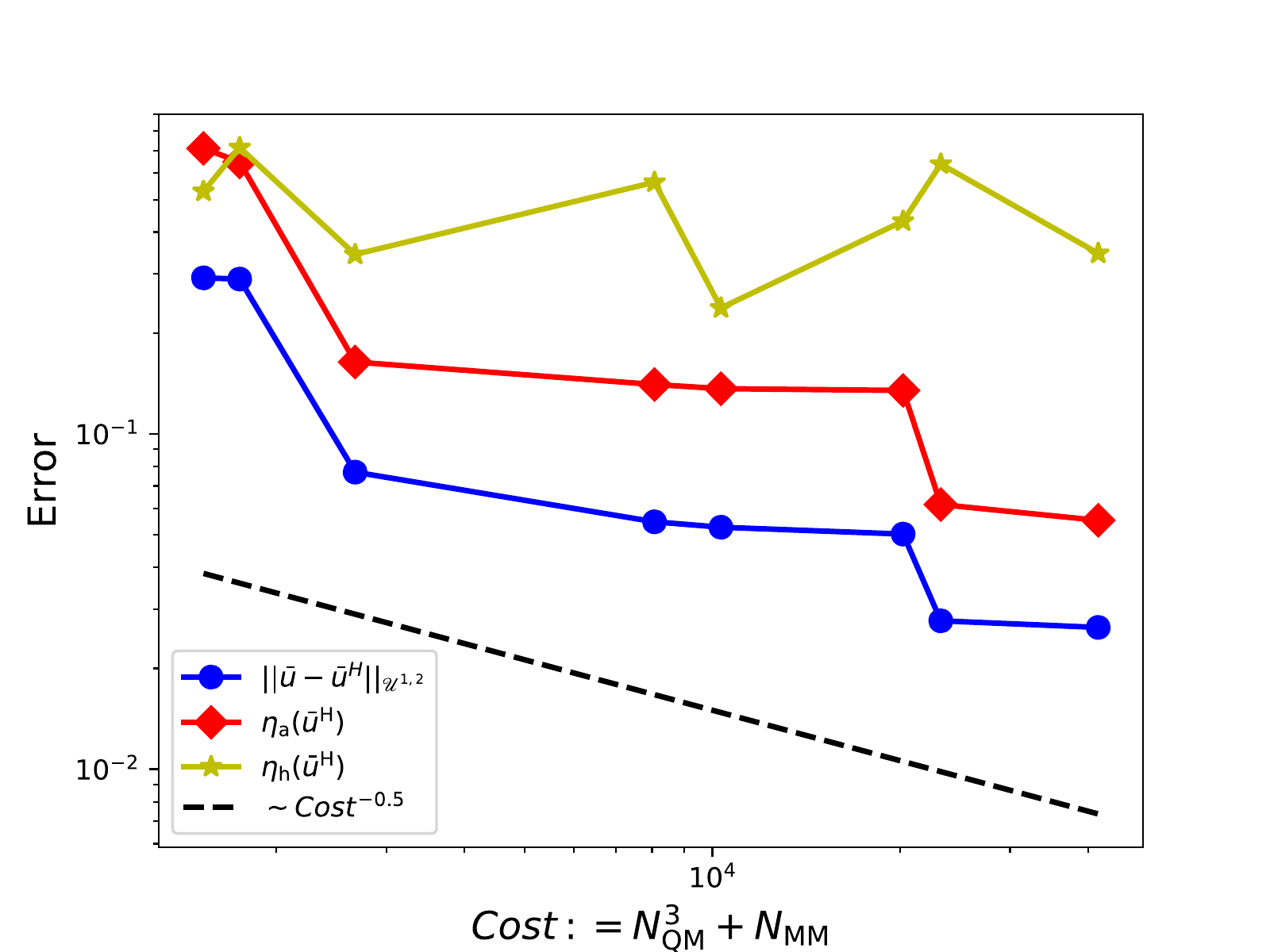}}
	\subfigure[$\tau_{\rm est}=0.3$]{
		\label{fig:tau03}
		\includegraphics[height=3.4cm]{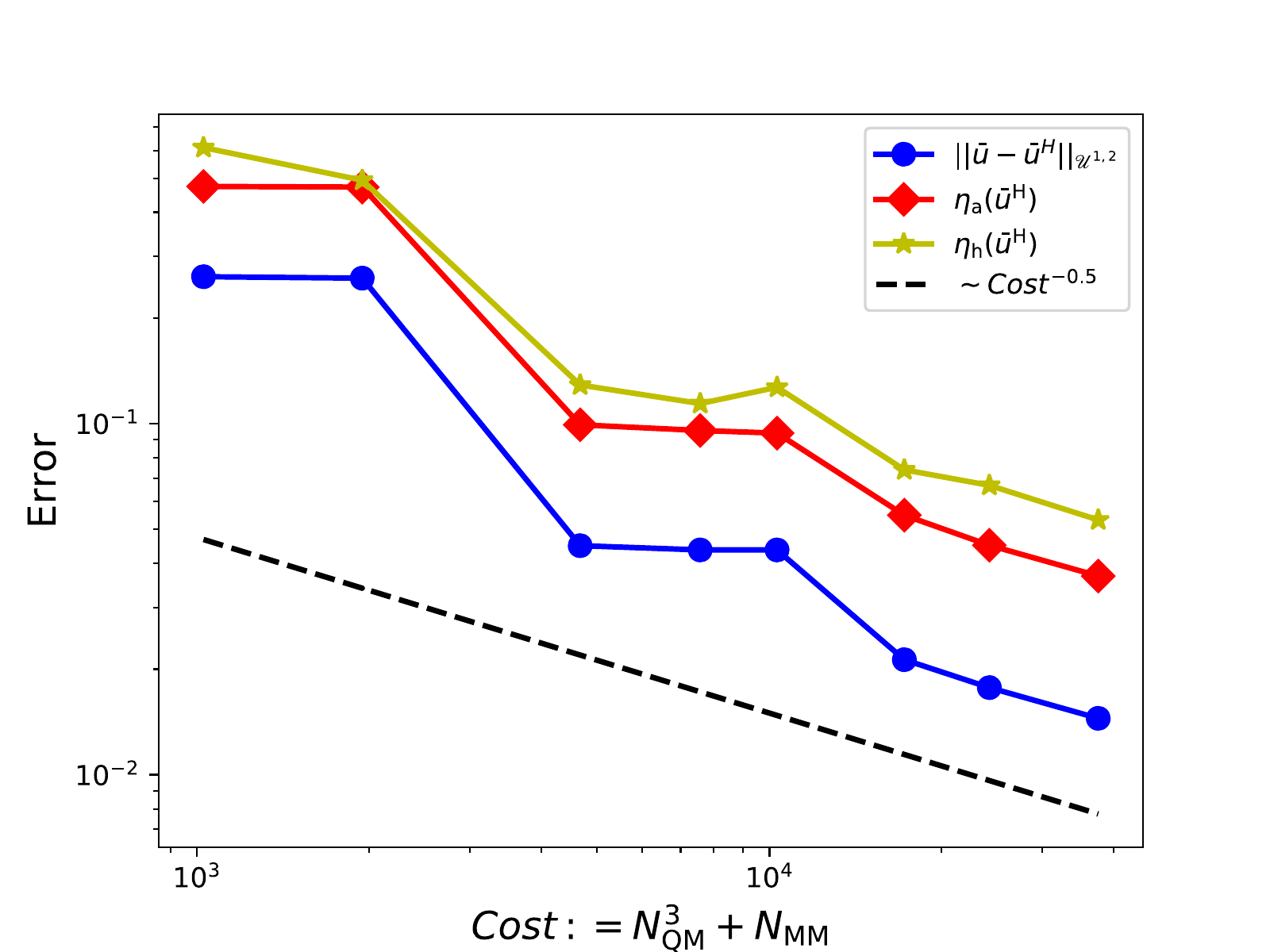}}
	\subfigure[$\tau_{\rm est}=0.1$]{
		\label{fig:tau01}
		\includegraphics[height=3.4cm]{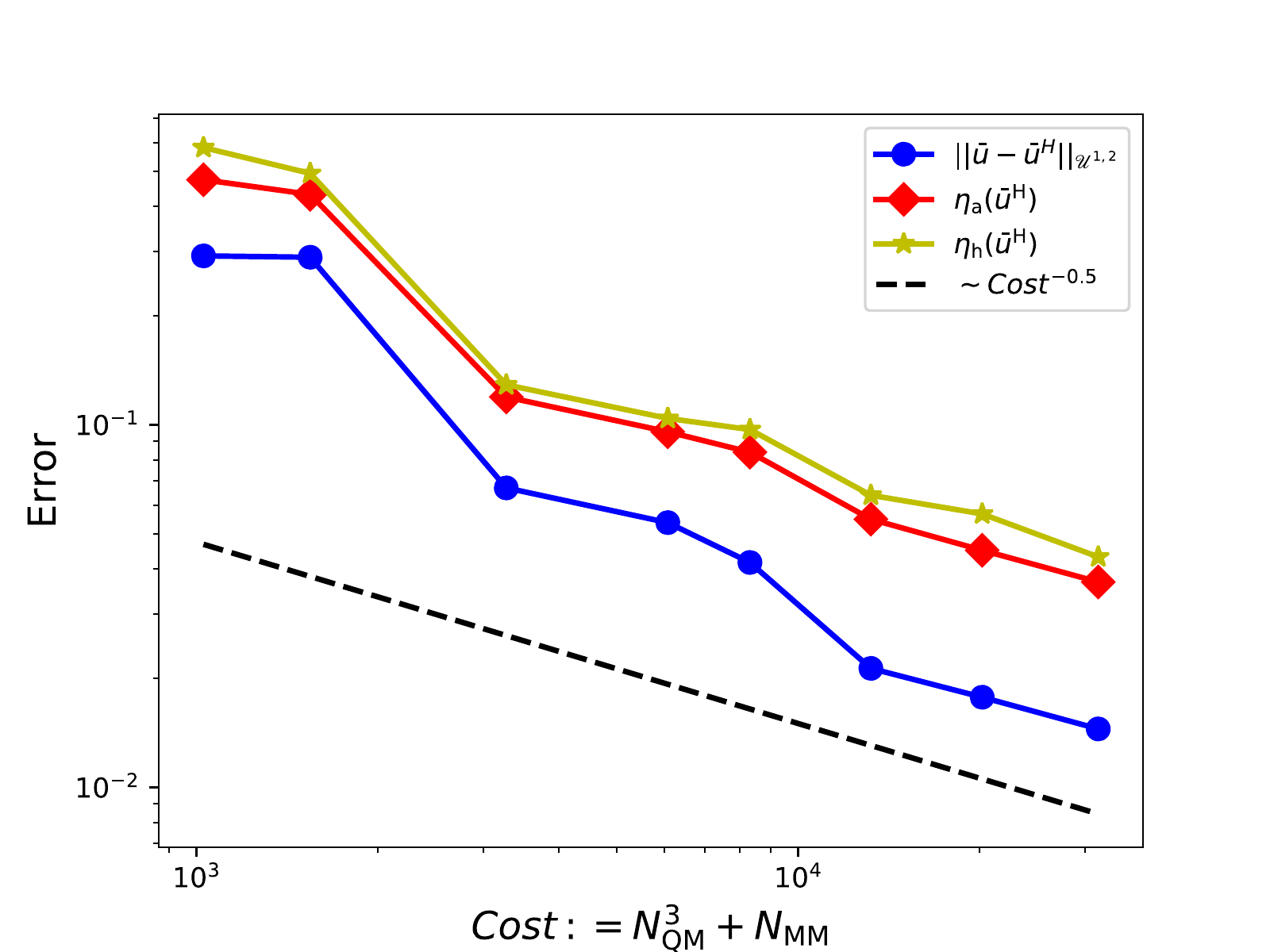}}
	\caption{QM/MM errors and error indicators, plotted against steps in the adaptive QM/MM Algorithm~\ref{alg:main}. Different stopping criterion $\tau_{\rm est}$ employed in Algorithm~\ref{alg:adaptMesh} can lead to qualitatively different behaviour in the QM/MM model refinement.}
	\label{fig:Singleape_difftau}
\end{figure}

\subsection{Convergence rates}
\label{sec:numconv}
Finally, we study the convergence of the Algorithm~\ref{alg:main}, for all three benchmark problems introduced in \S~\ref{sec:numexp}. 
Let $\NQM$ and $\NMM$ represent the numbers of atoms in the QM and MM regions respectively,
In each {\em solve} for $\uH$, the computational cost is $O(\NQM^3 + \NMM)$, as the cost to solve the QM (tight binding) model scales cubically and the cost to solve the MM model scales linearly with respect to the number of atoms. 
In Figure \ref{fig:Conv}(a,b,c) we therefore plot the approximation error $\|u-\uH\|_{\UsH}$, and the estimators $\eta_\h(\uH), \eta_\a(\uH)$ against this estimate of computational cost with $\tau_{\rm est}=0.3$. 
We observe two things: Firstly, difference between the practise estimator $\eta_\h$ and the (nearly) ideal estimator $\eta_{\rm a}$ is marginal, confirming our analysis and preliminary experiments that $\eta_\h$ provides an efficient and reliable estimator for the QM/MM model residual. 
Secondly, we observe that the estimators follow the trend of the approximation error fairly closely, but overestimate by anything between a factor that ranges from 2.74(Point defect), 3.57(Micro-crack) and 2.34(Edge-dislocation).

In Figure~\ref{fig:Conv}(d,e,f), the ratio of $N_{\rm QM}$ and $N_{\rm MM}$ during the adaptation process is shown, demonstrating that our adaptive algorithm automatically approaches the {\em quasi-optimal} cost splitting between QM and MM regions predicted by the {\it a priori} error analysis~\cite{chen15b}.

\begin{figure}
	\centering
	\subfigure[Point defect]{
		\label{fig:qmmmape}
		\includegraphics[height=3.4cm]{tau03.pdf}}
	\subfigure[Micro-crack]{
		\label{fig:qmmmapec}
		\includegraphics[height=3.42cm]{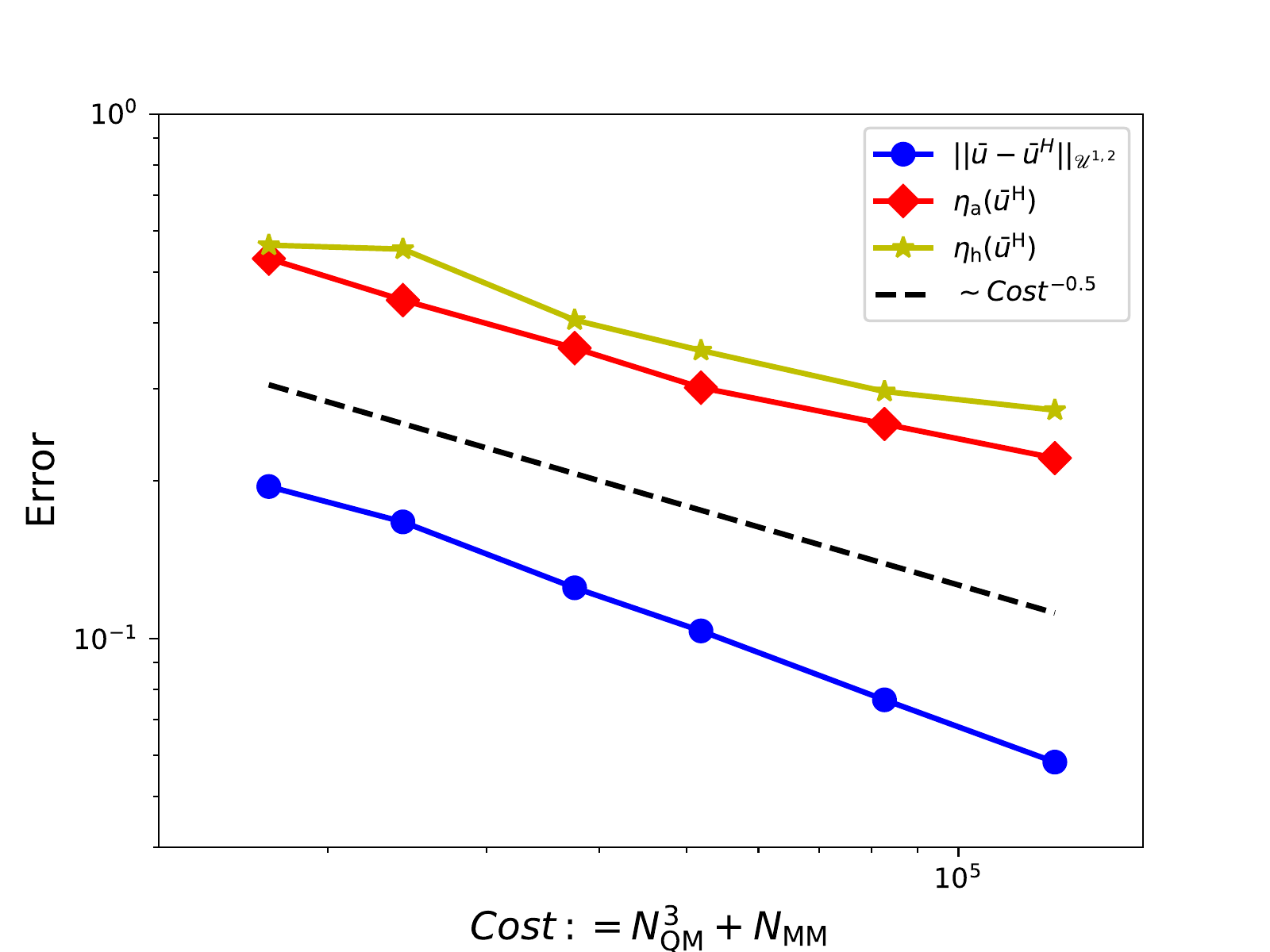}}
	\subfigure[Edge-dislocation]{
		\label{disape}
		\includegraphics[height=3.52cm]{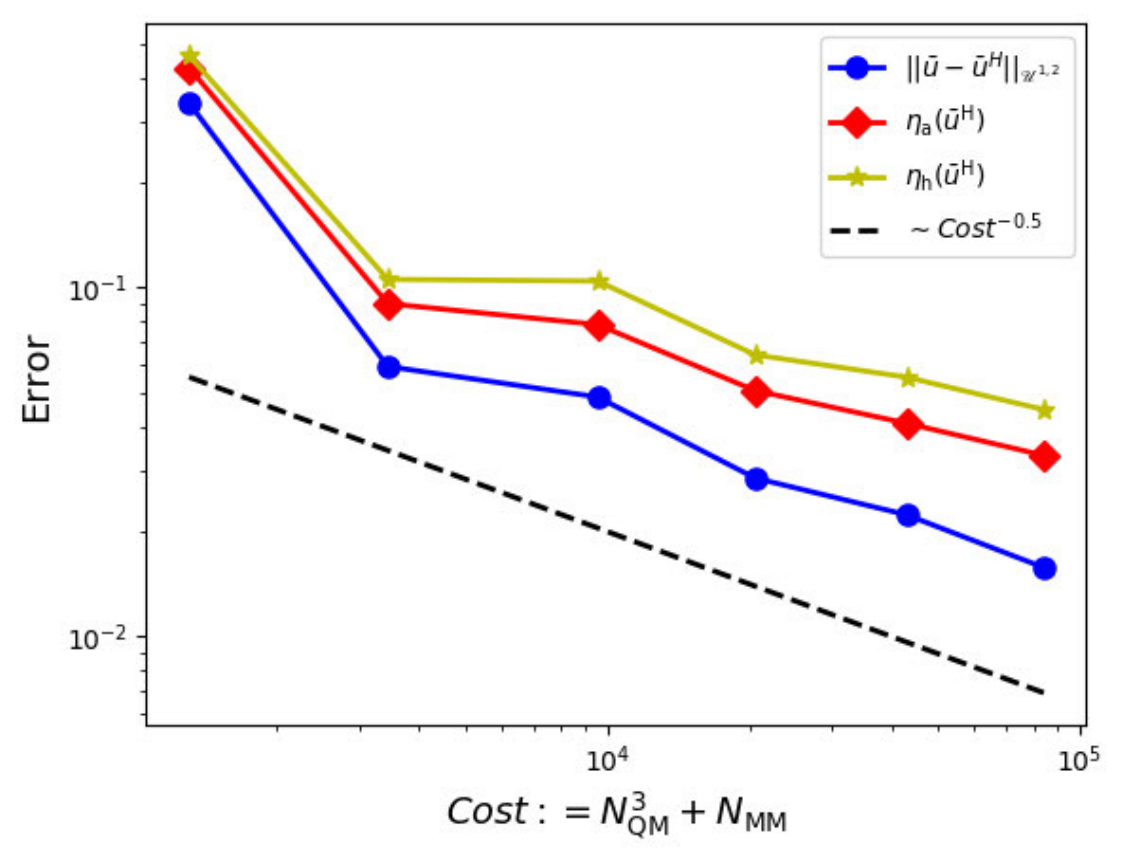}}
	\subfigure[Point defect]{
		\label{fig:qmmmrela}
		\includegraphics[height=3.67cm]{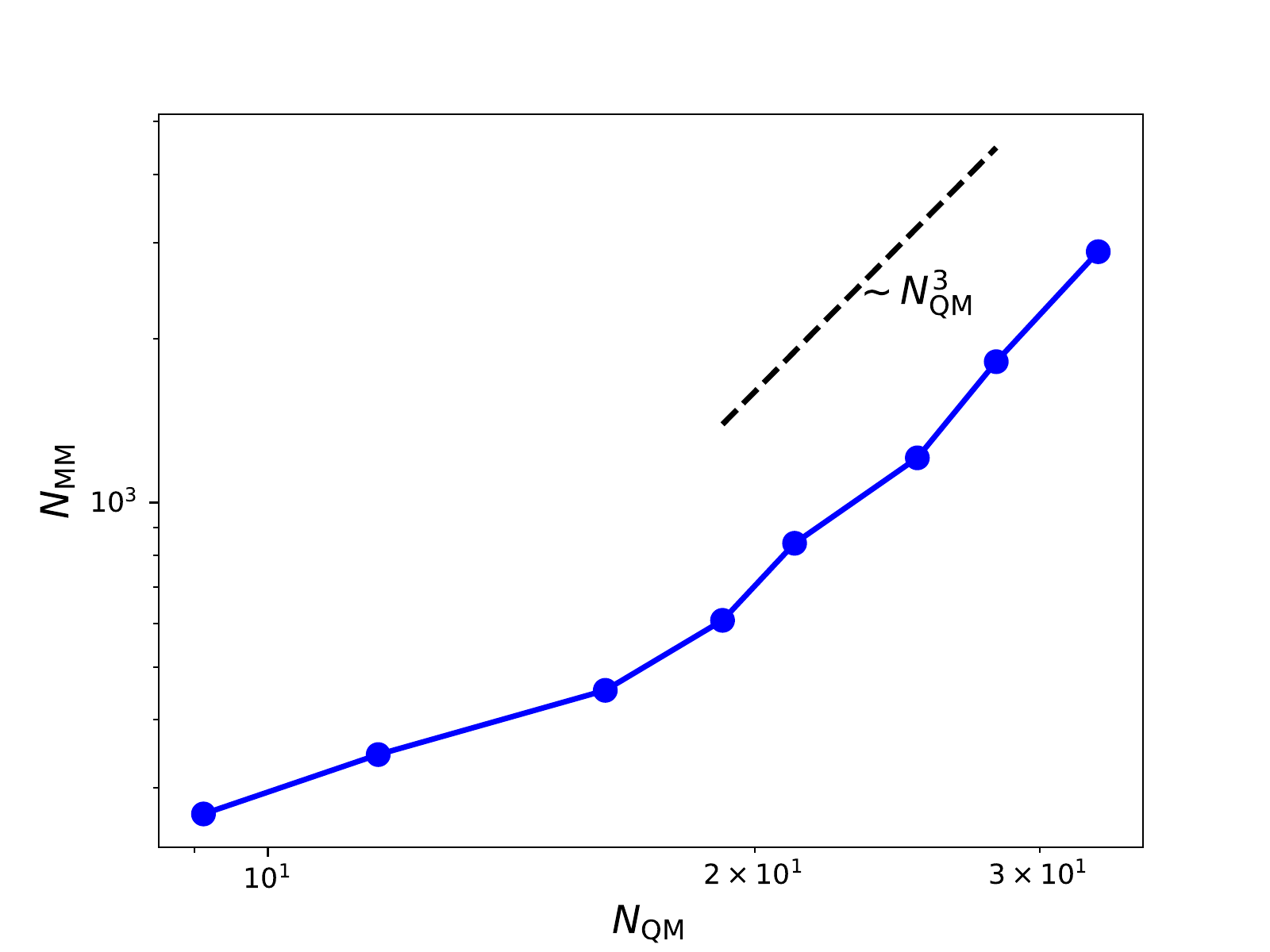}}
	\subfigure[Micro-crack]{
		\label{fig:qmmmrelac}
		\includegraphics[height=3.30cm]{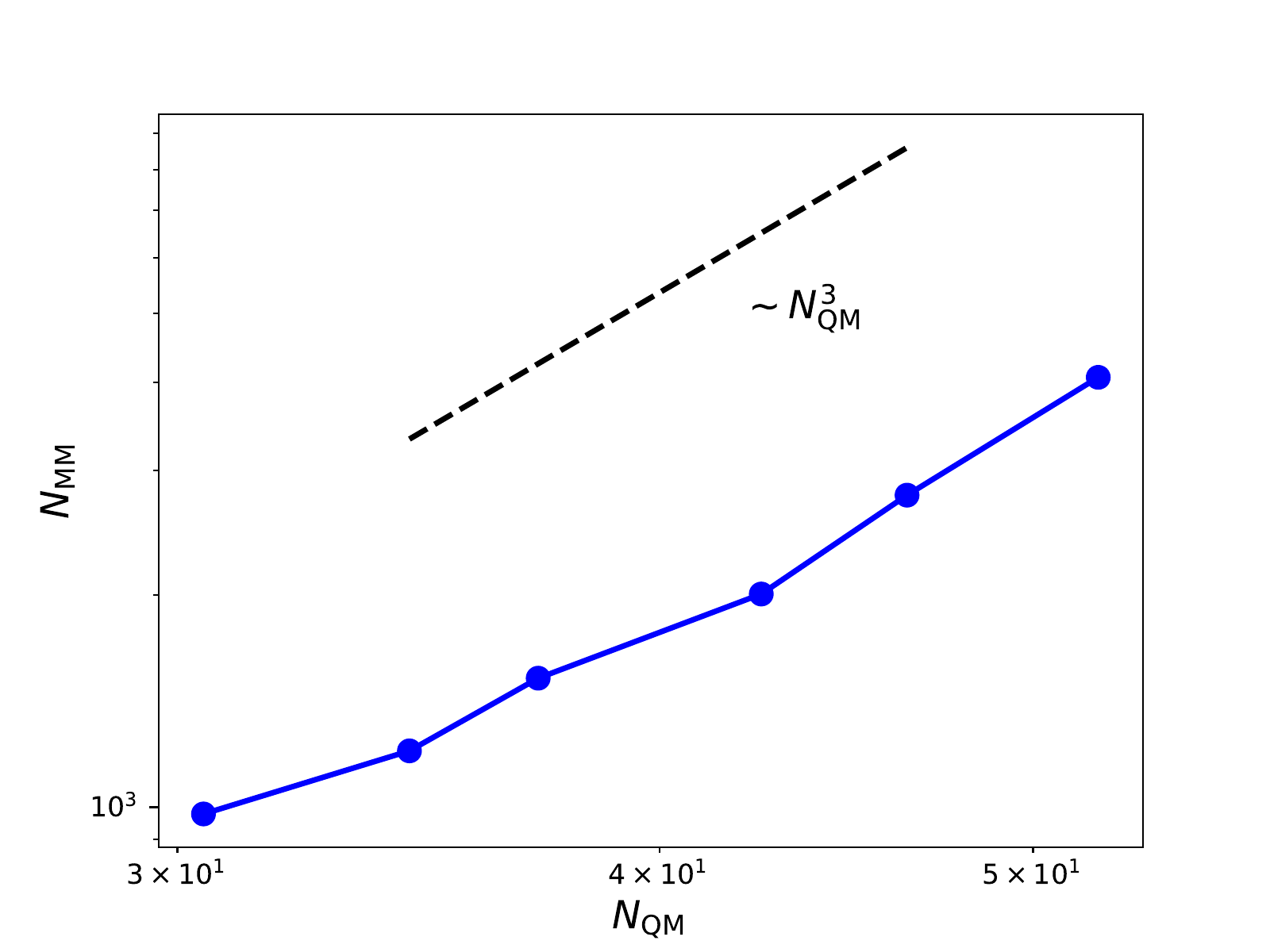}}
	\subfigure[Edge-dislocation]{
		\label{fig:dispath}
		\includegraphics[height=3.30cm]{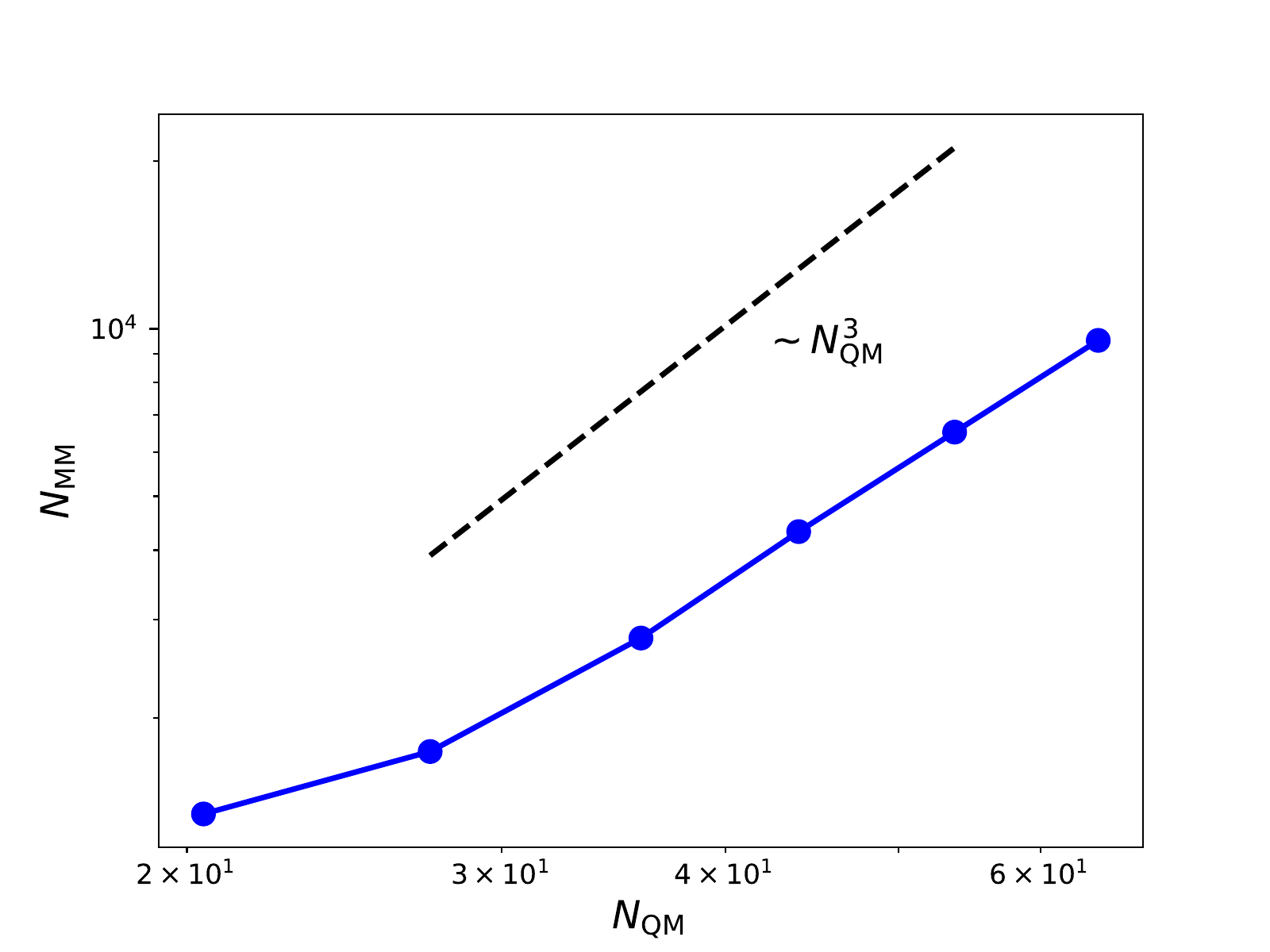}}
	\caption{Convergence of the adaptive Algorithm \ref{alg:main} and the scaling of QM and MM degrees of freedom during the adaptation process.}
	\label{fig:Conv}
\end{figure}

\section{Extensions and Remarks}
After establishing the main theoretical ideas in \S~\ref{sec:posteriori} and demonstrating their utilities in \S~\ref{sec:numer}, we now provide additional analyses of additional issues that are helpful to understand the properties of our adaptive algorithm. 

\def\Rg{\mathcal{R^{\rm BUF}_{\ell}}}
\def\uu{u_0^{\rm H}}
\def\LI{\Lambda^{\rm I}}

\subsection{Data Oscillation}
\label{sec:dataosc}
In this section, we analyse the {\em oscillation factor} $F_{\rm osc}$ in Theorem \ref{th:mainresult} with respect to the decay of the residual force $|f_{\ell}(\uH)|$, under two additional assumptions:
\begin{itemize}
    \item The initial triangulation $\T^0$ grows at most linearly with respect to the distance from the defect core, that is, 
    \begin{eqnarray}\label{eq:mesh}
    h_0(x)\lesssim |x|, \quad  \text{with}~ h_0(x):= \textrm{diam}(T)~\text{for}~x\in T\in\T^0. 
    \end{eqnarray}
    \item There exists a $C^{2,1}$ interpolant $\check{f}$ such that
\begin{align}\label{eq:decay}
	|\nabla^j \check{f}| \lesssim 
	\left\{\begin{aligned}
		|x|^{-\alpha_1-j}, & \quad x \in \Omega^{\rm FF} \\ 
		|x|^{-\alpha_2-j}, & \quad x \in \Omega^{\rm MM}
	\end{aligned}
	\right. \quad \textrm{for}~j=0,1,2,
\end{align}
    where $\alpha_1, \alpha_2 >0$ are the decay rates of $|f_\ell|$ in FF region and MM region respectively. This assumption is motivated by the results in \cite{2013-defects,chen15b,chen19}.
\end{itemize}

Under these assumptions we obtain the following general result, which can be interpreted in the context of the benchmark examples in \S~\ref{sec:numer}.

\begin{proposition}\label{prop:dataosc}
If \eqref{eq:mesh} and \eqref{eq:decay} are satisfied, then 
\begin{align}\label{eq:prop51}
\|\hat{f}-\hat{f}_{\T^0}\|_{(\dot{H}^1)^*} \lesssim (\log\RO)^t\Big(\RO^{-\beta_1}+\RQM^{-\beta_2} + \RMM^{-\beta_1}\Big)\|\hat{f}\|_{(\dot{H}^1)^*}+R_\Omega^{1+d/2}\exp(-\eta_1\rcut),
\end{align}
where $\beta_{i} = \alpha_{i} - d/2 - 1$ and $t = 1$ if $d=2$, $t=0$ if $t \geq 3$. 
\end{proposition}
\begin{proof} 
    See Appendix~\ref{sec:dataosc_proof}.
\end{proof}
 
The general consequence of this result is that, if $\rcut$ is sufficiently large such that the second term on the right hand side of \eqref{eq:prop51} is negligible or at least proportional to the first term, then we have
\begin{align*}
    F_{\rm osc}:=\frac{\|\hat{f}-\hat{f}_{\T^0}\|_{(\dot{H}^1)^*}}{\|\hat{f}\|_{(\dot{H}^1)^*}} \lesssim (\log\RO)^t\Big(\RO^{-\beta_1}+\RQM^{-\beta_2} + \RMM^{-\beta_1}\Big).
\end{align*}
 
We now discuss assumption \eqref{eq:decay} for the test problems in \S~\ref{sec:numexp}.
According to \cite{chen15b, chen19} and the proof in Appendix \ref{appdx:force_decay}, under the reasonable additional assumption that $|D^j\uH| \eqsim |D^j\bar{u}|$ for $j=0,1$, we have the following a priori estimates for $\ell \in \LMM$ and $\ell \in \LFF$:
\begin{enumerate}
\item Point defects: $d=2$ or $d=3$, $k=2$ in the definition of MM site energy \eqref{taylor}, \begin{itemize}
    \item $|f_{\ell}(\uH)| \eqsim (1+|\ell|)^{-3d}$ for 
    $\ell \in \LMM$, 
    \item $|f_\ell(\uH)| \eqsim \exp(-\gamma|\ell|)$ for $\ell \in \LFF$.
    \end{itemize}
\item Dislocation: $d=3$, $k=3$ in \eqref{taylor}, 
\begin{itemize}
    \item $|f_{\ell}(\uH)| \eqsim |\ell|^{-4}$ for $\ell \in \LMM$, 
    \item $|f_{\ell}| \eqsim |\ell|^{-3}$ for $\ell \in \LFF$.
    \end{itemize}
\end{enumerate}
In both cases, it therefore follows that $F_{\rm osc}$ is controlled (provided $\rcut$ is sufficiently large).

\begin{remark}
    We have made a simplifying assumption in the foregoing analysis by assuming that the decay of the residual forces is uniform across the MM/FF interface region. Due to a loss of symmetry this need not be true; see the discussion in Appendix \ref{appdx:force_decay} and the jump for the residual forces around the interface in \cite[Figure 4(b)]{CMAME}. However, in practice, we select a finer finite element mesh around the MM/FF interface to construct $\T^0$, as depicted in Figure \ref{fig:meshsteps}, to reduce the contribution of the force oscillation at the interface.
\end{remark}



\subsection{Stable $L^2$-projection}
\label{sec:stableL2}

Under stronger requirements on the regularity of the finite element mesh $\T^\h$ we can use stability of the $L^2$-projection to obtain an alternative {\it a posteriori} bound on $\|\nabla \phi - \nabla \phi_\h\|_{L^2}$, which entirely removes the term $\rho_\h$ from the resulting QM/MM residual bound.  Concretely, let 
\[
    M_\h^{L_2} := \sup_{v \in H^1_0(\Omega) \setminus \{0\}} 
    \frac{ \| \nabla \Pi_\h v \|_{L^2(\Omega)} }{\| \nabla v \|_{L^2(\Omega)}}
\]
be the operator norm of the $L^2$-projection. Estimating $M_\h^{L_2}$ is a classical and well-studied problem in numerical analysis. For example, if the mesh is locally quasi-uniform and the volume of neighbouring elements does not change too drastically, then one can obtain uniform bounds on $M_\h^{L_2}$ \cite{Bramble:2001}. Here, we shall not be further concerned with the precise conditions under which such bounds are obtained but only study the consequences.

Let $v \in \dot{H}^1(\R^d)$, and let $v_{R} = T_{\RO} v$, then 
\begin{align*} 
&\int_{\R^d} \hat{f}_{\T^0} \cdot v - \nabla \phi_\h \cdot \nabla v \dx \nonumber \\
=& \int_{\R^d} \hat{f}_{\T^0} \cdot (v - v_{R}) \dx + \int_{\R^d} \hat{f}_{\T^0} \cdot (v_{R} - \Pi_\h v_{R}) \dx  
- \int_{\R^d} \nabla \phi_\h \cdot (\nabla v - \nabla \Pi_\h v_{R})  \dx \nonumber \\
=:&\, T_1 + T_2 + T_3.
\end{align*}
Arguing as in the proof of Lemma~\ref{lem:approxEst_alt} we immediately obtain 
\[
    T_1 + T_3 \lesssim 
    \Big( 
    \RO \| \hat{f}_{\T^0} \|_{L^2(\Omega \setminus B_{\RO/2})}
    + (1 + M_\h^{L_2}) \| \nabla \phi_\h \|_{L^2} \Big) \, \| \nabla v \|_{L^2}.
\]

To estimate $T_2$ we note that $\int_{\Omega} f_\h \cdot (v_{R} - \Pi_\h v_{R}) \dx = 0$ for all $f_\h$ in the finite element space. Let 
\[
    \hat{f}_{\T^0}^0
    := 
    \sum_{\ell \in \mathcal{N}^0 \setminus \partial\Omega} \hat{f}_{\T^0}(\ell) \zeta_\ell^{\rm c}(x), 
\]
i.e. we simply set the nodal values on the boundary to zero. Then we have $\hat{f}_{\T^0}^0 = \hat{f}_{\T^0}$ in all elements $T$ except those that touch the boundary. For the latter it is straightforward to prove that $\|\hat{f}_{\T^0}^0 - \hat{f}_{\T^0}\|_{L^2(T)} \lesssim  \|\hat{f}_{\T^0}\|_{L^2(T)}$. This allows us to estimate 
\begin{align*}
    T_2 &= \int_{\Omega} (\hat{f}_{\T^0} - \hat{f}_{\T^0}^0) \cdot (v_{R} - \Pi_\h v_{R}) \dx \\ 
    &\lesssim 
    \| \hat{f}_{\T^0} - \hat{f}_{\T^0}^0 \|_{L^2}
    \| v_{R} \|_{L^2} 
    \lesssim 
    \RO \|  \hat{f}_{\T^0} \|_{L^2(\Omega \setminus B_{\RO/2})}\|\nabla v\|_{L^2},
\end{align*}
where we have used the Poincar\'{e}--Friedrichs inequality in the final step. 
%
Indeed, we have grossly overestimated here, but there is no advantage in a sharper estimate.  

In summary we we obtain a simpler (less sharp) {\it a posteriori} error bound for the estimator,  
\begin{align}\label{eq:L2Proj_EstPhih}
	\| \nabla \phi - \nabla \phi_{\rm h} \|_{L^2}
	\lesssim  R_\Omega \|\hat{f}_{\mathcal{T}^0}\|_{L^2(\Omega\setminus B_{\RO/2})}  + (M_\h^{L_2}+1)\|\nabla \phi_\h\|_{L^2}.
\end{align}

In light of the analysis in \S~\ref{sec:dataosc} the first term, $\|\hat{f}_{\mathcal{T}^0}\|_{L^2(\Omega\setminus B_{\RO/2})}$, is naturally interpreted as a data-error term, i.e., it is reasonable to define a modified oscillation factor 
\[
    F_{\rm osc}' := 
     \frac{\| \hat{f} - \hat{f}_{\T^0} \|_{(\dot{H}^1)^*} + \RO \| \hat{f}_{\T^0} \|_{L^2(\Omega \setminus B_{\RO/2})}}{\| \hat{f} \|_{(\dot{H}^1)^*}}.
\]

With this alternative bound the main result, Theorem \ref{th:mainresult}, could be reformulated as
\begin{equation} \label{eq:main:first-L2}
	(1+F_{\rm osc}')^{-1} \|\nabla \phi_\h\|_{L^2}
	\lesssim
	\| \delta \E(\uH) \|_{(\UsH)^*}
	\lesssim
	(1+F_{\rm osc}' + M_\h^{L_2})
	\|\nabla \phi_\h\|_{L^2}.
\end{equation}
The estimate \eqref{eq:L2Proj_EstPhih} will also be numerically verified in Figure \ref{fig:cg}. However, our numerical experiments (see in particular Figure~\ref{fig:Singleape_difftau}) show that it is {\em in practise} important to obtain a good resolution of the estimator $\phi_\h$; hence we  have chosen to retain the estimate \eqref{eq:combined_phi_phih_estimator_alt} in the adaptive algorithm~\ref{alg:adaptMesh}.

\subsection{Stress of the QM Model}
\label{sec:stress}
As our final remark on the algorithms derived in the foregoing section, we will make the connection between the mechanical notion of stress and the a posteriori estimator defined through $\phi$ in \eqref{eq:defnvarphi}. 
This is motivated by the stress based formulation of the atomistic/continuum coupling method and corresponding stress based a posteriori estimators \cite{Or:2011a, PRE-ac.2dcorners, 2012-ARMA-cb, Wang:2017, Liao2018}. For a general discussion of atomistic stress we refer to \cite{AdTa:2010}.

To derive a QM stress we extend the technique used in \cite{2012-lattint, 2012-ARMA-cb, 2014-bqce} to QM models, we restrict the discussion to  the homogeneous lattice $\L \equiv \Lhom$. To map between a defective reference configuration $\L$ and the corresponding homogeneous lattice $\Lhom$ one can use \cite[Lemma D.1]{chen19}, and extend this discussion to defective lattices.

After a straightforward computation following \cite{2012-ARMA-cb, 2014-bqce}, we have the identity
\begin{eqnarray}\label{eq:stress_force}
\int_{\R^d}\Sigma(\uH)(x) :  \nabla v(x)\dx = \int_{\R^d} \-f(\uH)(x) \cdot v(x)\dx, \quad \forall v \in \UsH,
\end{eqnarray}
where the {\em stress} $\Sigma(\uH)$ is defined by
\begin{eqnarray*}
\Sigma(\uH)(x) := \sum_{\ell \in \L}  \sum_{\rho \in \Rg} \chi_{\ell,\rho}(x) V_{\ell, \rho}(D\uH(\ell)) \otimes \rho 
\end{eqnarray*}
with $V_{\ell}\big(Du\big) :=  E_{\ell}(y_0 + u)$ and ``smeared bonds'' $\chi_{\ell,\rho}(x) := \int_0^1 \zeta(\ell+t\rho-x)\dt$. For the sake of completeness we give the derivation of $\Sigma(\uH)$ in  Appendix~\ref{sec:apd_stress}.

Formally, $\Sigma(\uH)$ defines an analogue of the second Piola stress tensor for the QM model. We note that the QM model has an infinite interaction range $\Rg$. However, thanks to the locality results Lemma \ref{lemma-thermodynamic-limit}, QM stress $\Sigma$ is exponentially localised. This makes a direct  connection to the atomistic stress.

According to the Helmholtz-Hodge decomposition \cite{Or:2011a, paula2016}, $\Sigma(\uH)$ can be decomposed as a sum of two orthogonal components:
\begin{equation}
\label{eq:HHD}
\Sigma(\uH) = \nabla \phi + \nabla\times \psi,
\end{equation}
with $\phi \in \UsH$, $\psi\in \UsH$.
$\nabla\phi\in L^2$ is called the ``curl-free" component, and $\nabla\times \psi \in  L^2$ is divergence-free in the weak sense, i.e., $\int_{\R^d}\nabla\times\psi(x) : \nabla v(x) \dx = 0$.

Combining Theorem \ref{th:mainresult} and Lemma \ref{th:ctsphi}, we have
\begin{align*}
\eta_\h^2(\uH) \lesssim \| \delta \E(\uH) \|_{-1}^2 \lesssim \|\nabla \phi\|_{L^2}^2\leq \|\nabla\phi\|_{L^2}^2 + \|\nabla\times\psi\|_{L^2}^2=\|\Sigma(\uH)\|_{L^2}^2.
\end{align*}
Therefore, $\|\Sigma(\uH)\|_{L^2}$ provides an upper bound for the approximation error.

In \eqref{eq:HHD}, we can uniquely define $\phi$ by $\Delta \phi = \nabla \cdot \Sigma(\uH)$ (in the weak sense). On the other hand, we can choose an arbitrary divergence-free component $\nabla\times\psi$ in $\Sigma(\uH)$ to satisfy \eqref{eq:stress_force}. Therefore, $\Sigma(\uH)$ in the sense of \eqref{eq:var_QM} is not unique, and we will consider the following problem to obtain a uniquely defined QM stress tensor,
\begin{equation}\label{eq:minSigma}
\bar{\psi} \in \arg\min_{\psi \in  \UsH} \big\{ \|\Sigma(\uH)\|_{L^2} = \|\nabla\phi + \nabla\times\psi\|_{L^2} \big\}.
\end{equation}

A straightforward calculation and the orthogonality of the two components of Helmholtz-Hodge decomposition lead to $\nabla\times\bar{\psi} = 0$. Hence, we denote the corresponding uniquely defined QM stress tensor as
\begin{equation}\label{eq:stress_grad}
\Sigma^0(\uH) := \nabla\phi, \quad {\rm for}\; \phi \in \UsH.
\end{equation}

By choosing the unique QM stress tensor through \eqref{eq:stress_grad} and the inverse interpolation operator $(I^{\rm h}_1)^{-1}: \UsH(\Lhom) \rightarrow \UsH(\L)$, we recover the equation \eqref{eq:defnvarphi} which is used in the a posteriori estimates in the previous sections.

\def\Pqm{\mathcal{P}_{\rm QM}}

\section{Conclusions}
\label{sec:conclusion}
\setcounter{equation}{0}
We proposed a residual based {\it a posteriori} error estimator, and designed an accompanying model-adaptive algorithm,  for
QM/MM multi-scale approximations of crystalline solids with embedded defects.
We have shown both theoretically and in three benchmark problems that the estimator provides both upper and lower bounds for the approximation error.  

Both our estimator and our algorithm are in many respects agnostic about the approximations made to the reference electronic structure model, suggesting possible extensions to other approximation schemes and application areas. 

\paragraph{Outlook: Anisotropic geometries.} 
Most but not all steps of our analysis and algorithm are independent of the geometry of the material defect and computational domain, hence we briefly mention where some refinements are required to achieve full generality of the analysis and applicability of the algorithms to more complex defect configurations (e.g. cracks, partials separated by a stacking fault, etc.):


The first potential problem is that the Poincar\'{e} constant (cf. \eqref{eq:trmod}) in anisotropic domains depends on the domain shape. A simple and general class of domains that can still be treated with minor changes to the analysis are those obtained by smooth deformations of a ball. The Poincar\'{e} constant can then be estiamted in terms of the deformation gradient and the volume element. 


The algorithmic challenges are more significiant: Algorithm \ref{alg:main} adjusts only $\RQM$ and $\RMM$ to refine the model. This prevents us from capturing significant anisotropy in the defect core, elastic field, or indeed defect nucleation. To consider such generalisations, we need to evolve the QM/MM and MM/FF interfaces anisotropically. A possible way forward is to think of this as a free interface problem based on the error distribution, which may lead to robust implementation of model adaptivity.  

Both the theoretical and practical aspects discussed above will be explored in future work.

\clearpage 

\appendix

\section{Supplementary Material}
\label{sec:deri_u0}

\subsection{Far-field boundary condition for dislocations:}
\label{sec:dislocation}
For dislocations, we follow the constructions in \cite{chen19, 2013-defects} and prescribe $u_0$ as follows. Let $\L\subset\R^2$, $\hat{x}\in\R^2$ be the position of the dislocation core and $\Upsilon := \{x \in \R^2~|~x_2=\hat{x}_2,~x_1\geq\hat{x}_1\}$ be the ``branch cut'', with $\hat{x}$ chosen such that $\Upsilon\cap\Lambda=\emptyset$.
We define the far-field predictor $u_0$ by
\begin{eqnarray}\label{predictor-u_0-dislocation}
u_0(x):=\ulin(\xi^{-1}(x)),
\end{eqnarray}
where $\ulin \in C^\infty(\R^2 \setminus \Upsilon; \R^d)$ is the solution of continuum linear elasticity (CLE)
\begin{eqnarray}\label{CLE}
\nonumber
\mathbb{C}^{j\beta}_{i\alpha}\frac{\partial^2 u^{\rm lin}_i}{\partial x_{\alpha}\partial x_{\beta}} &=& 0 \qquad \text{in} ~~ \R^2\setminus \Upsilon,
\\
u^{\rm lin}(x+) - u^{\rm lin}(x-) &=& -{\rm b} \qquad \text{for} ~~  x\in \Upsilon \setminus \{\hat{x}\},
\\
\nonumber
\nabla_{e_2}u^{\rm lin}(x+) - \nabla_{e_2}u^{\rm lin}(x-) &=& 0 \qquad \text{for} ~~  x\in \Upsilon \setminus \{\hat{x}\},
\end{eqnarray}
where the forth-order tensor $\mathbb{C}$ is the linearised Cauchy-Born tensor (derived from teh potential $V$, see \cite[Section 7]{2013-defects} for more detail).
\begin{eqnarray}
\xi(x)=x-\burg_{12}\frac{1}{2\pi}
\eta\left(\frac{|x-\hat{x}|}{\hat{r}}\right)
\arg(x-\hat{x}),
\end{eqnarray}
with $\arg(x)$ denoting the angle in $(0,2\pi)$ between $x$ and
$\burg_{12} = (\burg_1, \burg_2) = (\burg_1, 0)$, and
$\eta\in C^{\infty}(\R)$ with $\eta=0$ in $(-\infty,0]$ and $\eta=1$ in
$[1,\infty)$ which removes the singularity.

We mention that for the anti-plane screw dislocation, under the proper assumptions on the interaction range $\Rg$ and the potential $V$, the first equation in \eqref{CLE} simply becomes to $\Delta u^{\rm lin} = 0$ \cite{2017-bcscrew}. The system \eqref{CLE} then has the well-known solution $u^{\rm lin}(x) = \frac{{\rm b}}{2\pi}\arg(x-\hat{x})$, where we identify $\R^2 \cong \C$ and use $\Upsilon-\hat{x}$ as the branch cut for arg.

\subsection{Derivation of \eqref{eq:stress_force}, atomistic stress:}
\label{sec:apd_stress}
We first introduce the so-called localization formula (see \cite{2012-ARMA-cb})
\begin{align}\label{eq:localization}
D_{\rho} \tilde{v}(\ell)=\int_0^1\nabla_{\rho}\tilde{v}(\ell+t\rho)\dt
=\int_{\R^{d}}\int_0^1\zeta(\ell+t\rho-x)\dt\nabla_{\rho}v(x)\dx.
\end{align}
where the (quasi-)interpolation $\tilde{v}$ is defined as
\begin{align*}
	\tilde{v}(x) := (\zeta * v)(x) =  \int_{\R^d} \zeta(x-y) v(y) \d y, \end{align*}
with the nodal interpolant on $\Lhom$
\begin{align*}
	v(x) := \sum_{\ell\in\Lhom} v(\ell) \zeta(\ell-x),  \quad \text{for} ~v \in \UsH(\Lhom).
\end{align*}

In order to make the QM stress more clear, let $E_{\ell}$ be the site energy we defined in Section \ref{sec:tb},  we define $V_{\ell} : (\R^d)^{\L-\ell}\rightarrow\R$ by
\begin{eqnarray*} 
V_{\ell}\big(Du\big) :=  E_{\ell}(y_0 + u),
\end{eqnarray*}
which is possible due to its translational invariance.

It can be shown in \cite[Lemma 10]{2012-lattint} that $\~v|_{\Lhom}\in\UsH(\Lhom)$.
Hence, for any solution $\uH\in\UsH$ of \eqref{problem-e-mix}, by replacing the test function $v$ by $\tilde{v}$, the first variation of \eqref{energy-difference} is given by
\begin{align}\label{eq:var_QM}
\< \delta\E(\uH), \tilde{v} \>&=\sum_{\ell \in \L} \sum_{\rho \in \Rg} V_{\ell, \rho}(D\uH(\ell)) \cdot D_{\rho}\tilde{v}(\ell) \nonumber \\
&=\int_{\R^{d}} \sum_{\ell \in \L} \sum_{\rho \in \Rg} \rho \otimes V_{\ell, \rho}( D\uH(\ell)) \int_0^1\zeta(\ell+t\rho-x)\dt : \nabla v(x) \dx   \nonumber \\
&=: \int_{\R^{d}}\Sigma(\uH)(x) : \nabla v(x) \dx,
\end{align}
where
\begin{eqnarray}\label{eq:QM_stress}
\Sigma(\uH)(x) := \sum_{\ell \in \L} \sum_{\rho \in \Rg} \rho \otimes V_{\ell, \rho}(D\uH(\ell))\int_0^1 \zeta(\ell+t\rho-x)\dt.
\end{eqnarray}
We could also obtain that
\begin{align*}
\< \delta\E(\uH), \tilde{v} \>
=\sum_{\ell \in \L} \sum_{\rho \in \Rg} \big[ V_{\ell-\rho, \rho}(D\uH(\ell-\rho)) - V_{\ell, \rho}(D\uH(\ell)) \big] \tilde{v}(\ell)
= \int_{\R^{d}} \-f(\uH)(x) \cdot v(x) \dx.
\end{align*}
Combined with \eqref{eq:var_QM} leads to the following equation
\begin{eqnarray*}
\int_{\R^d}\Sigma(\uH)(x) :  \nabla v(x)\dx = \int_{\R^d} \-f(\uH)(x) \cdot v(x)\dx, \quad \forall v \in \UsH,
\end{eqnarray*}
which yields \eqref{eq:stress_force} exactly.

\label{sec:apd_proof}

\subsection{Proof of \bf Lemma \ref{th:ctsphi}:}
	In variational form, \eqref{th:ctsphi} reads
	\begin{equation} \label{eq:defnvarphi}
		\int_{\R^d}\nabla \phi(\uH)(x) \cdot \nabla v(x)\dx = \int_{\R^d} \-f(\uH)(x) \cdot v(x)\dx, \quad \forall v \in \dot{H}^1(\R^d).
	\end{equation}
	Existence of $\phi \in \dot{H}^1$ and uniqueness (up to shifts) are straightforward.
	It is, moreover, convenient to define an atomistic grid potential $\phi_\a \in \UsH$, by an analogous discrete Poisson equation,
	\begin{equation} \label{eq:defnphidvar}
	\int_{\R^d}\nabla \phi_\a(\uH)(x) \cdot \nabla v(x)\dx = \int_{\R^d} \-f(\uH)(x) \cdot v(x)\dx, \quad \forall v \in \UsH.
	\end{equation}
	We first prove the equivalence between $\| \nabla\phi \|_{L^2}$ and $\|\nabla \phi_\a\|_{L^2}$. The equations \eqref{eq:defnvarphi}, \eqref{eq:defnphidvar}, and a Galerkin orthogonality argument yield the inequality
	\begin{equation} \label{eq:equi_d_c_pre}
			\| \nabla \phi_{\a} \|^2_{L^2} \leq \|\nabla \phi\|^2_{L^2} = \| \nabla \phi_{\a} \|^2_{L^2} + \|\nabla \phi - \nabla \phi_\a\|^2_{L^2}.
	\end{equation}

Since $\^f\in L^2$, we have $\phi \in H^2_{\rm loc}$, and $\|\nabla^2 \phi\|_{L^2} \leq \|\Delta \phi\|_{L^2}$ which is known as the Miranda-Talenti estimate \cite{Maugeri:2000}. Applying the standard finite element a priori error analysis~\cite[\S~II.6]{braess2007finite}, we can therefore estimate the error term by
\begin{equation} \label{eq:equi_d_c_1}
		\|\nabla \phi - \nabla \phi_\a\|_{L^2} \leq C\|\nabla^2 \phi\|_{L^2} \leq C\|\Delta \phi\|_{L^2} =  C\|\-f\|_{L^2}.
\end{equation}
Here we have used the fact that the mesh size for the atomistic grid is uniformely bounded by a fixed constant.

We now estimate $\|\-f\|_{L^2}$ by $\|\nabla \phi_\a \|_{L^2}$. For any lattice function $v\in\UsH$, we have
\begin{align*}
\|\nabla v\|_{L^2} \leq C\|Dv\|_{\ell_\gamma^2}
	\leq
	C\|v\|_{\ell^2} \leq C\|v\|_{L^2},
\end{align*}
and therefore, for any $f \in\ell^2$, we have the dual bound
\begin{equation} \label{eq:duality_bound_fL2}
	\|f\|_{\ell^2}
	\leq C \| \hat{f} \|_{L^2}
	 = \sup_{v \in \Us^{1,2} \setminus \{0\}} \frac{\int_{\R^d} \hat{f} \cdot v \, dx}{ \|v \|_{L^2}}
	 \leq C \sup_{v\in\Us^{1,2} \setminus\{0\}}\frac{\int_{\R^d} \hat{f}(x)\cdot v(x)\dx}{\|\nabla v \|_{L^2}}
	 = C \| \nabla \phi_\a \|_{L^2}.
\end{equation}

Using also the norm-equivalence $\|\nabla v\|_{L^2} \eqsim \| D v \|_{\ell^2}$ this establishes
\begin{equation}
	\| \nabla \phi \|_{L^2}
	\eqsim \| \nabla \phi_\a \|_{L^2}
	\eqsim \| \hat{f} \|_{(\UsH)^*}
	:=
	\sup_{v\in\Us^{1,2} \setminus \{0\} }\frac{\int_{\R^d} \hat{f}(x)\cdot v(x)\dx}{\|Dv \|_{\ell^2}}.
\end{equation}

We have therefore reduced the statement to proving the equivalence between $\| \hat{f} \|_{(\UsH)^*}$ and $\| f \|_{(\UsH)^*} = \|\delta\E(\uH)\|_{(\UsH)^*}$.
The key observation is that we can interpret $f$ as a quadrature approximation to  $\hat{f}$. If $\mathcal{I}$ denotes the standard $\mathcal{P}_1$ nodal interpolation operator, then according to \eqref{eq:QM_force} we have
\begin{equation}
	\int_{\R^d} \mathcal{I}\big[ \hat{f} v \big] \dx
	= \sum_{\ell \in \Lambda}  \hat{f}(\ell) v(\ell) \int_{\R^d} \zeta_\ell(x)\dx
	= \sum_{\ell \in \Lambda}  c_\ell f(\ell) \zeta_\ell(\ell) {v}(\ell) \frac{1}{c_\ell}
	= \sum_{\ell \in \Lambda} f(\ell) {v}(\ell).
\end{equation}
(Note that this is the key step where the rescaling of the nodal interpolant $\hat{f}$ enters.)

Following from the standard quadratrue estimates, we can obtain
\begin{align}\label{eq:quadratrue}
	\bigg| \int_{\R^d} \hat{f} v \dx - \sum_{\ell \in \Lambda} f(\ell) v(\ell) \bigg| &=
	\bigg| \int_{\R^d} \hat{f} v - \mathcal{I}\big[ \hat{f} v \big] \dx\bigg| \nonumber\\
&\lesssim
	\| \nabla \hat{f} \|_{L^2} \| \nabla v \|_{L^2} \nonumber \\
	\nonumber
	&\lesssim
	\| \hat{f} \|_{L^2}  \| \nabla v \|_{L^2}  \\
	\nonumber
	&\lesssim
	\|f \|_{\ell^2} \| \nabla v \|_{L^2}.
\end{align}
So in summary, according to the last two inequalities and the norm-equivalence \eqref{eq:norm-eq}, we have shown that both
\begin{align*}
	\| f \|_{(\UsH)^*} &\lesssim \| \hat{f} \|_{(\UsH)^*} + \| \hat{f} \|_{L^2},
	\qquad \text{and}  \\
	\| \hat{f} \|_{(\UsH)^*}  &\lesssim \| f \|_{(\UsH)^*}  + \| f \|_{\ell^2}.
\end{align*}
In the first case, we use duality to bound $\| \hat{f} \|_{L^2} \lesssim \| \hat{f} \|_{(\UsH)^*}$; and in the second case we use duality to bound $\|f\|_{\ell^2} \lesssim \|f\|_{(\UsH)^*}$; cf. \eqref{eq:duality_bound_fL2}. Combing the resulting estimates we get the desired norm-equivalence.

\subsection{Numerical supplements}
\label{apped:ns}
\begin{figure}[H]
	\centering
	\subfigure[Micro-crack]{
		\label{fig:mcgeom}
		\includegraphics[height=4cm]{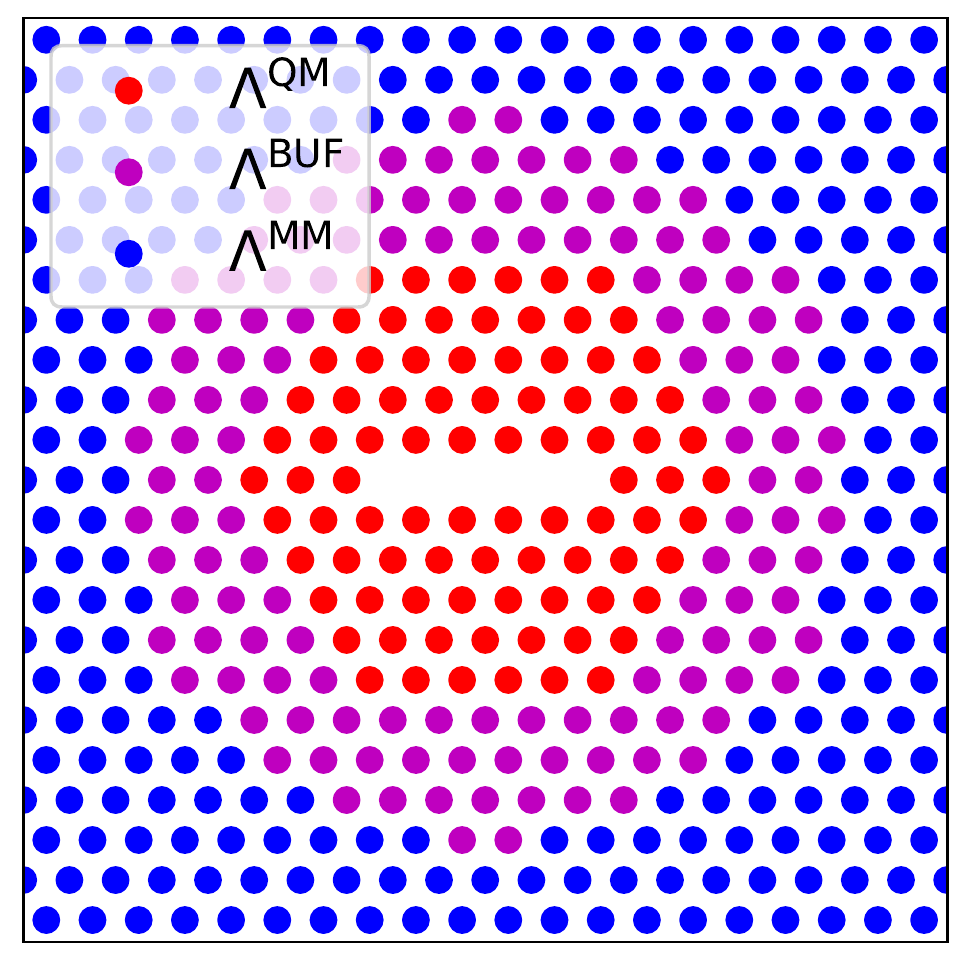}}
	\subfigure[Edge-dislocation]{
		\label{fig:edgeom}
		\includegraphics[height=4cm]{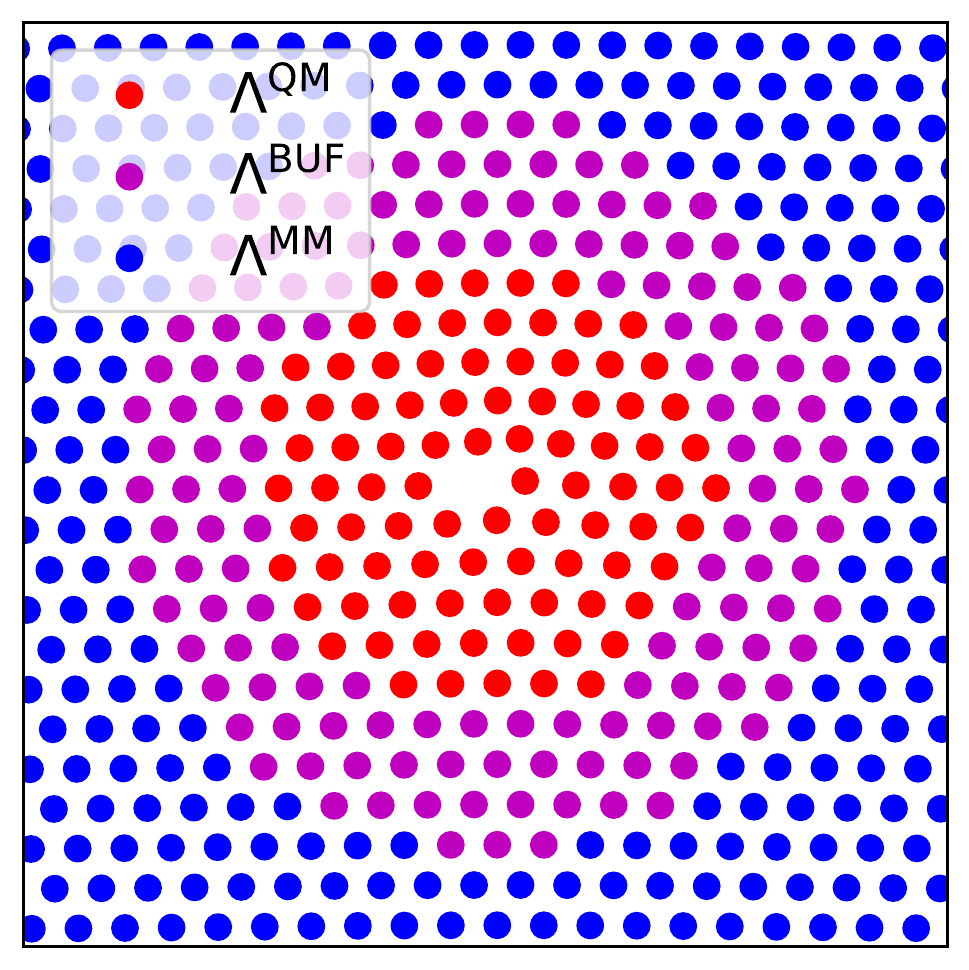}}
	\caption{QM/MM decompositions for micro-crack and edge dislocation examples.}
	\label{fig:qmmmgeom}
\end{figure}

\begin{figure}[H]
	\centering
	\subfigure[Point defect]{
		\label{fig:pdmesh-apd}
		\includegraphics[height=4cm]{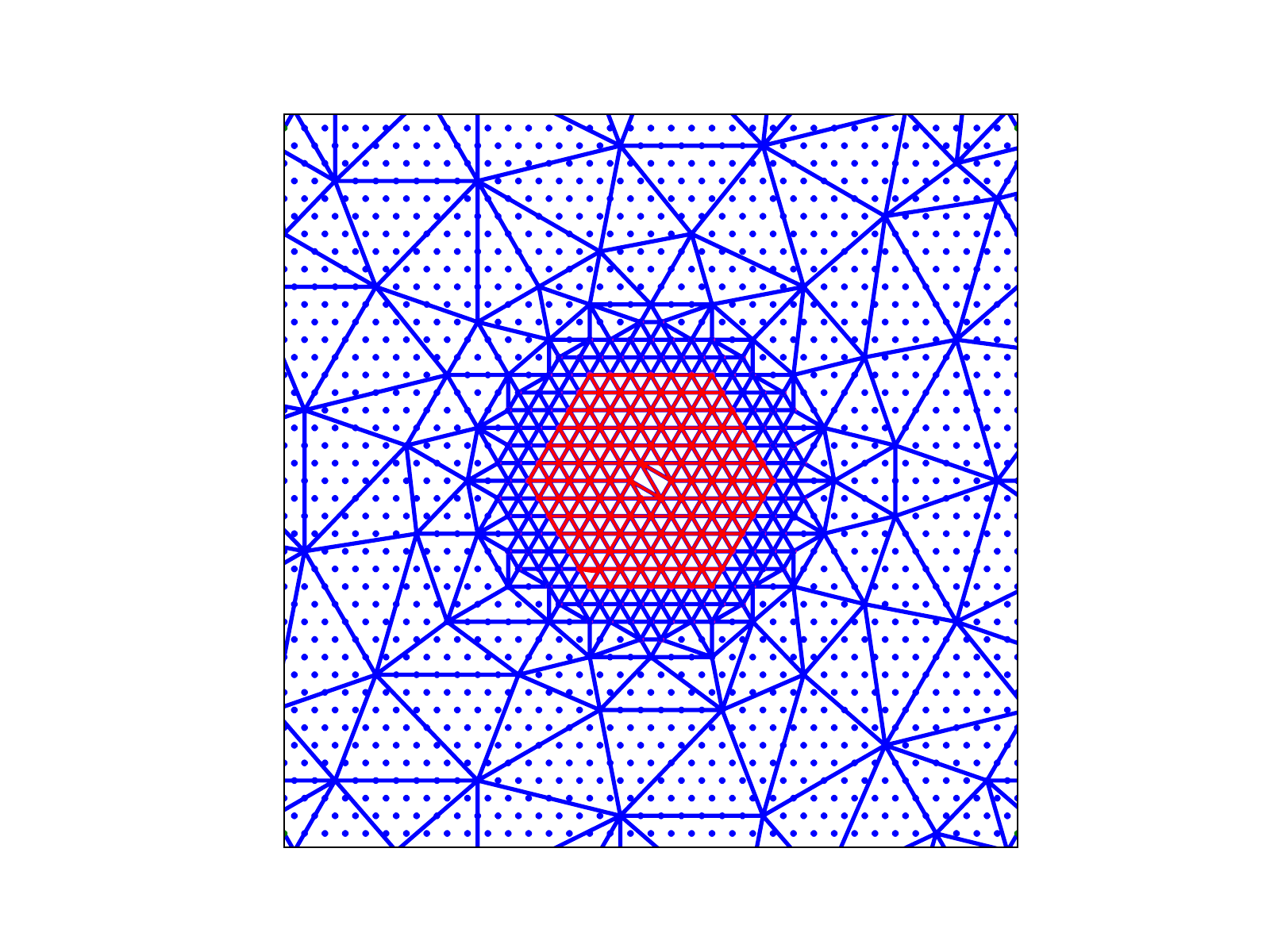}}
	\subfigure[Micro-crack]{
		\label{fig:mcmesh-apd}
		\includegraphics[height=4cm]{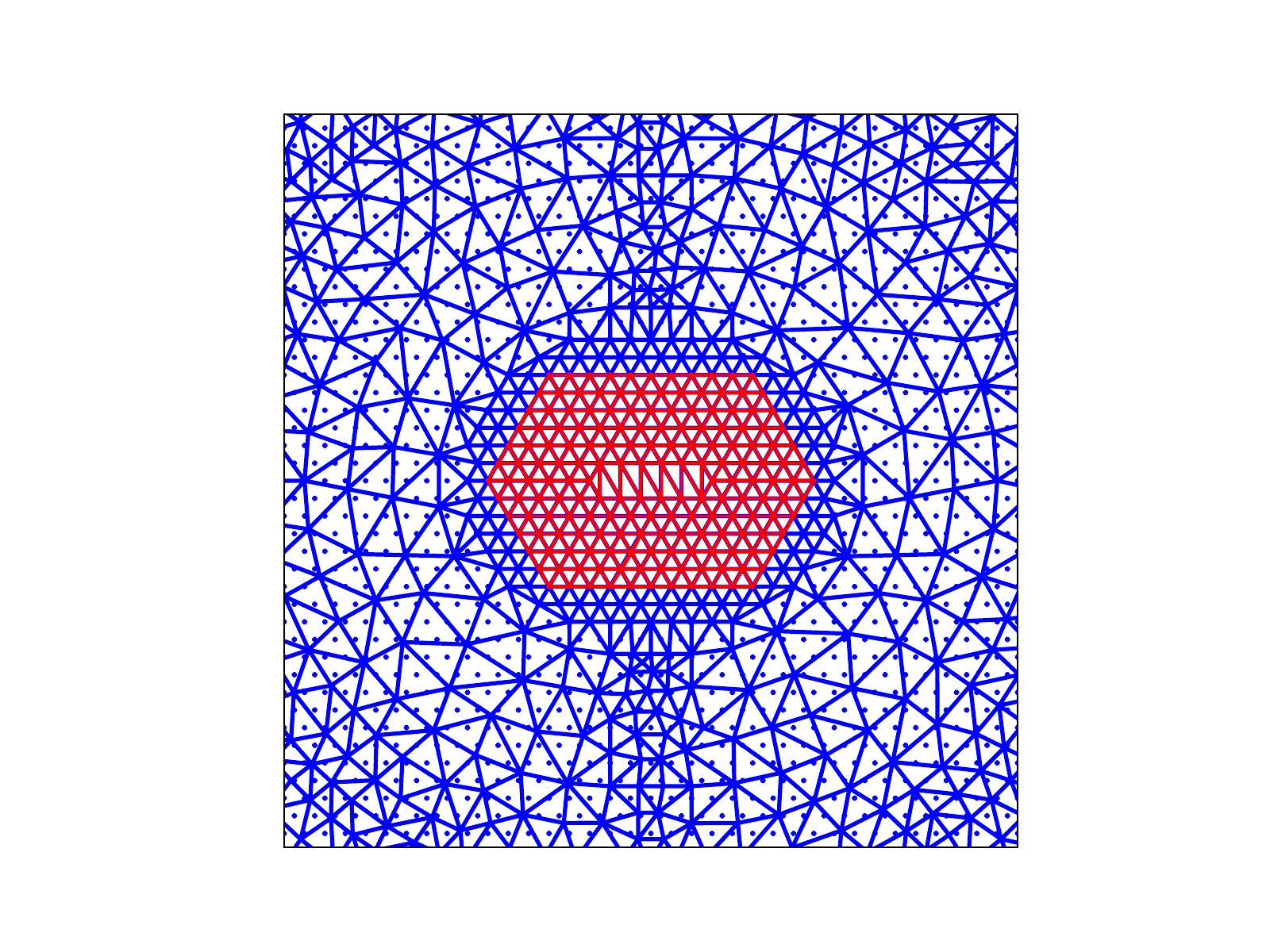}}
	\subfigure[Edge-dislocation]{
		\label{fig:edmesh-apd}
		\includegraphics[height=4cm]{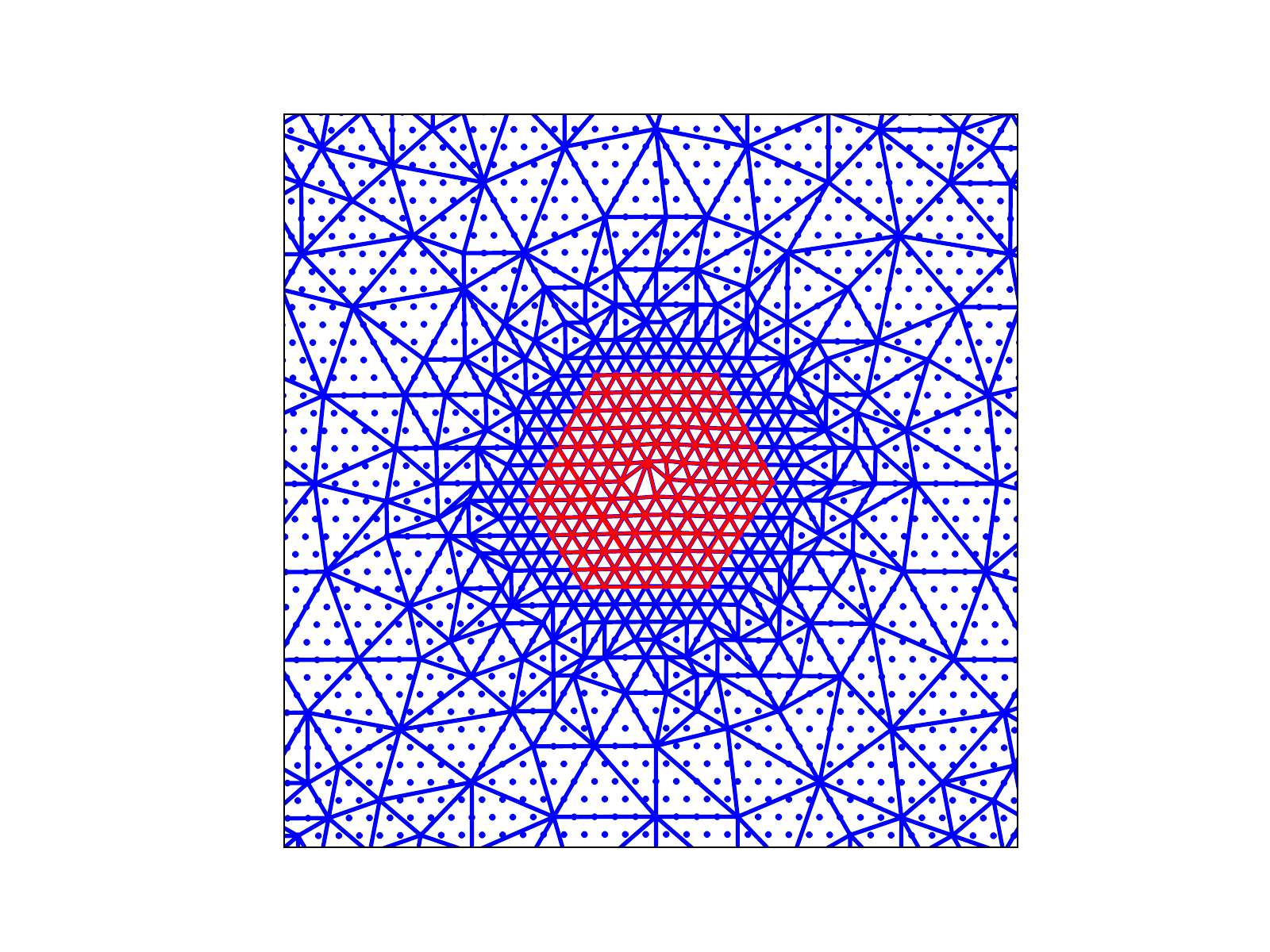}}
	\caption{Combined triangulation $\T^0$  constructed in  \S~\ref{sec:discrete} for three examples introduced in  \S~\ref{sec:numexp}.}
	\label{fig:qmmmmesh-apd}
\end{figure}

\begin{figure}[H]\label{fig:diseoet}
	\centering
	\subfigure[Relative error $\rho_{\h,\Omega}/\eta_\h$ with increasing $R_\Omega$ while fixing $\RQM$ and $\RMM$.]{
		\label{fig:distrun}
		\includegraphics[height=4.5cm]{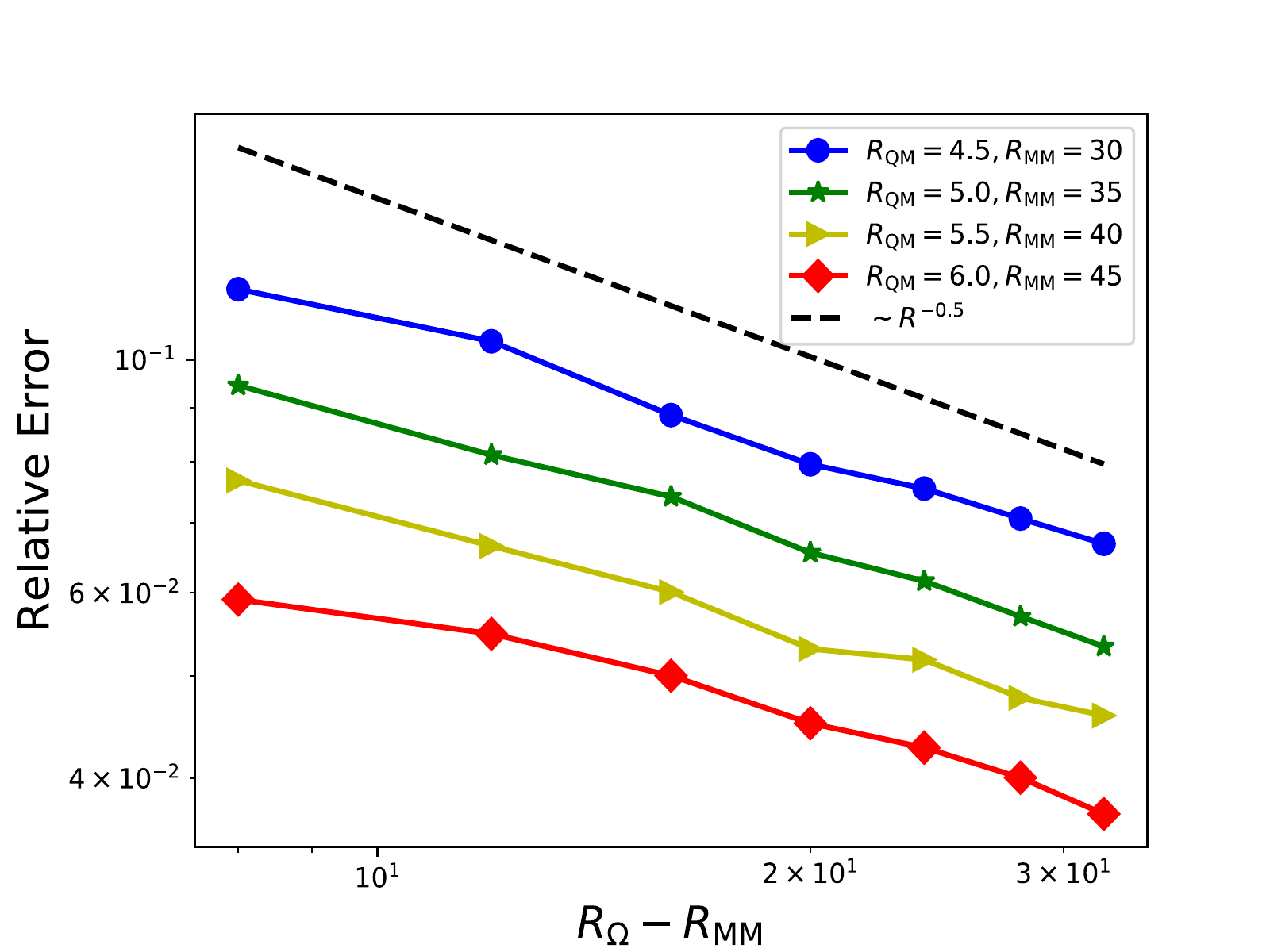}}
	\hspace{0.2cm}
	\subfigure[The values of residual indicator $\rho_{\h, \T}$, the estimator $\eta_{\h}$ and the data-oscillation during the mesh refinement.]{
		\label{fig:discg}
		\includegraphics[height=4.5cm]{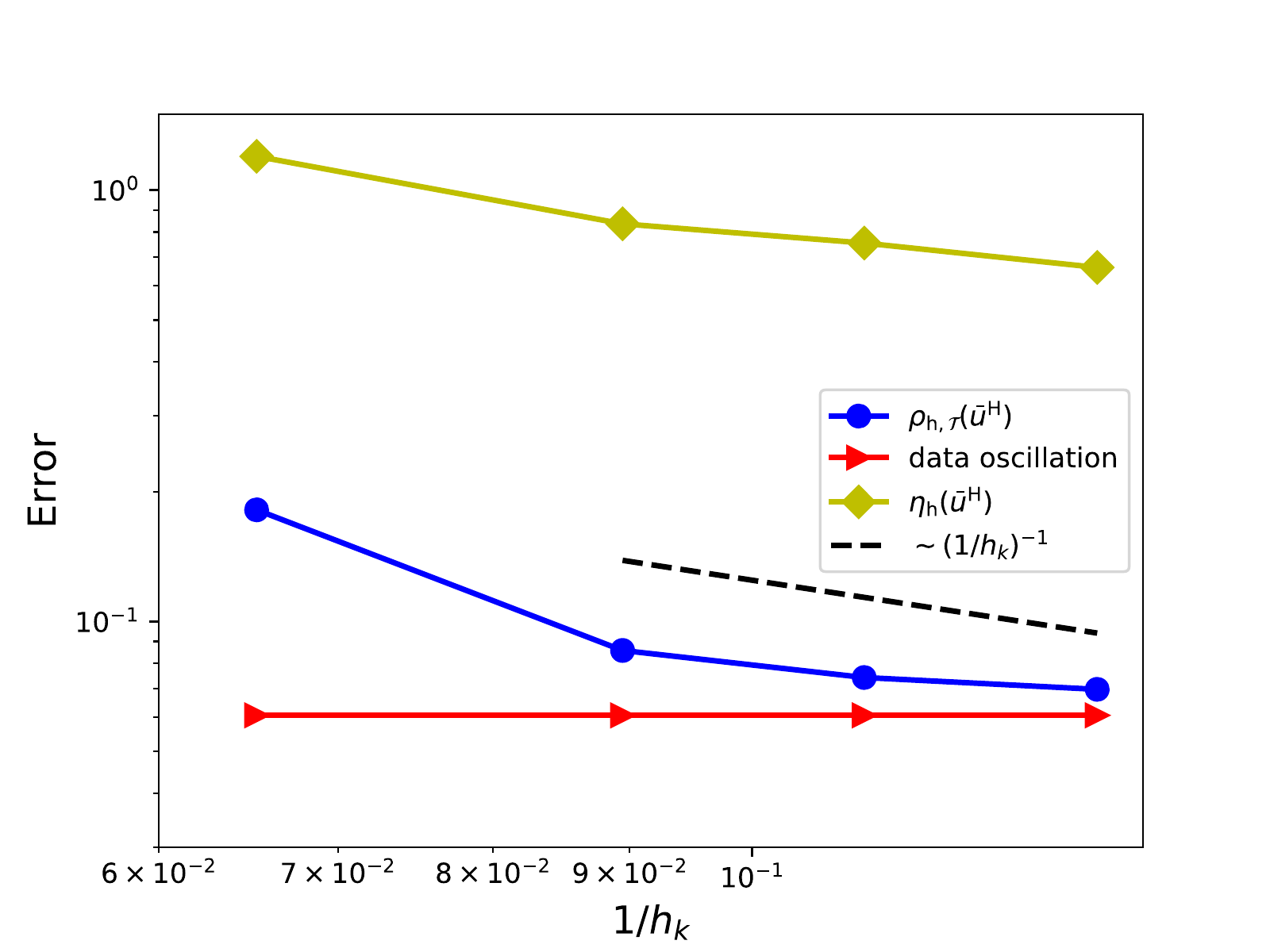}}
	\caption{(Edge-dislocation) Comparison of the truncation indicator, the residual indicator and the data-oscillation to the estimator.}
\end{figure}

We compare the truncation error, the discretization error, and the data oscillation within $\Omega$, with respect to error estimators $\eta_\h(\uH)$ for the edge dislocation. In Figure \ref{fig:distrun}, the truncation error decrease as $\RQM$ and $\RMM$ increase, and also as the width of the FF region $R_\Omega - R_{\rm MM}$ increase. We also note that the truncation error is relatively small compared with $\eta_\h(\uH)$ with sufficiently large $\RQM$ and $\RMM$. In Figure \ref{fig:discg}, we observe that, during the adaptive process of Algorithm \ref{alg:adaptMesh}, the discretization error decreases as we refine the mesh, and the part of data oscillation $\|\hat{f}-\hat{f}_{\T^0}\|_{(\dot{H}^1)^*(\Omega)}$ is relatively small if the initial mesh $\T^0$ is good enough.

\subsection{Proof of \bf Proposition \ref{prop:dataosc}}
\label{sec:dataosc_proof}
%
According to the definition of $\hat{f}_{\T^0}$ \eqref{eq:fhat_T0} and the fact $\int_{\R^d}\hat{f}\dx=\int_{\Omega}\hat{f}_{\T^0}\dx = 0$,  we have 
\begin{align}
\|\hat{f}-\hat{f}_{\T^0}\|_{(\dot{H}^1)^*} 
& = \sup_{\|\nabla v\|_{L^2}=1} \Big(\int_{\R^d \setminus \Omega}\hat{f} \cdot  (v-\bar{v}) \dx + \int_{\Omega}(\hat{f} - \hat{f}_{\mathcal{T}^0}) \cdot  (v-\bar{v}) \dx \Big) \\
& =:\sup_{\|\nabla v\|_{L^2}=1} \Big(T_1 + T_2\Big),
\end{align}
where $\bar{v}=\mint_{\Omega} v\dx$. We can estimate $T_1$ as
\begin{align}\label{eq:T_1}
 T_1 = \int_{\R^d \setminus \Omega}\hat{f} \cdot  (v-\bar{v}) \dx = \int_{\R^d \setminus \Omega}\omega(x)\hat{f} \cdot \omega^{-1}(x) (v-\bar{v}) \dx \lesssim \|\omega(x)\hat{f}\|_{L^2(\R^d \setminus \Omega)}\|\nabla v\|_{L^2}
\end{align}
with 
\begin{eqnarray}\label{assump}
\omega(x) = \left\{
\begin{array}{llll}
\displaystyle
|x|\ln |x|  & \text{if} ~~ d = 2,
\\[2ex]
\displaystyle
\sqrt{1+x^2}  & \text{if} ~~ d \geq 3.
\end{array} \right.
\end{eqnarray}
where  the Cauchy-Schwarz inequality and the weighted Poincar\'{e} inequality \cite[Corollary 16]{pauly2009functional} have been used. We first focus on the case $d\geq 3$  and the result for $d=2$ is very similar. According to the assumption \eqref{eq:decay}, we can further bound $T_1$ in \eqref{eq:T_1} by,
\begin{align}\label{eq:fwfR}
T_1\lesssim \|\omega(x) \hat{f} \|_{L^2(\R^d\setminus \Omega)}\|\nabla v\|_{L^2} \lesssim \RO^{-\alpha_1+d/2-1}\|\hat{f}\|_{(\dot{H}^1)^*}\|\nabla v\|_{L^2}.
\end{align}

We add and substract the constant $C_{\hat{f}_{\T^0}}$ into $T_2$ to have
\begin{align}
\label{eq:T_2}
T_2 = \int_{\Omega}(\hat{f}-\hat{f}_{\T^0})\cdot (v-\bar{v})\dx = - \int_{\Omega}C_{\hat{f}_{\T^0}}(v-\bar{v})\dx + \int_{\Omega}(\hat{f}-\hat{f}_{\T^0}+C_{\hat{f}_{\T^0}})\cdot (v-\bar{v})\dx. 
\end{align}
We note that the first term on the right hand side vanishes since $C_{\hat{f}_{\T^0}}$ is a constant and $\bar{v} = \mint_{\Omega} v\dx$. Substitute \eqref{eq:QM_force} and \eqref{eq:fhat_T0} into the second term, and denote the interpolant of $f_\ell$ on the finite element mesh $\T^0$ by $\hat{f}_{\N^0} := \sum_{\ell \in \N^0} f_{\ell}\zeta^{\rm c}_{\ell}(x)$, we split $T_2$ to two groups,
\begin{align*}
\int_{\Omega}(\hat{f}-\hat{f}_{\T^0}+C_{\hat{f}_{\T^0}})\cdot (v-\bar{v})\dx 
 = & \int_{\Omega} \Big(\hat{f} - \sum_{\ell \in \N^0} \tilde{f}_{\ell} \zeta^{\rm c}_{\ell}(x)\Big)\cdot (v-\bar{v}) \dx  \\
 = & \int_{\Omega}\Big( \sum_{\ell \in \N^0} (f_{\ell}-\tilde{f}_{\ell}) \zeta^{\rm c}_{\ell}(x)\Big)\cdot (v-\bar{v}) \dx \\
 & + \int_{\Omega} (\hat{f} - \hat{f}_{\N^0}) \cdot (v-\bar{v}) \dx \\
 =:& T_{\rm 21} + T_{\rm 22}.
\end{align*} 

As discussed in \S~\ref{sec:discrete}, 
$f_{\ell}-\tilde{f}_{\ell}$ decays exponentially with respect to $\rcut$ for each $\ell\in\N^0$. Using the Poincar\'{e} inequality,
we have
\begin{align}\label{eq:Tf}
\T_{21}  \lesssim & \int_{\Omega} \max_{x\in\Omega}(\#\{\ell\in\N^0~|~\zeta_{\ell}^c(x)\neq 0\}) \exp(-\eta_1\rcut) \cdot  (v-\bar{v})\dx \nonumber \\
 \lesssim & R_\Omega^{1+d/2}\exp(-\eta_1\rcut)\|\nabla v\|_{L^2},
\end{align}
where the overlapping number $\max_{x\in\Omega}(\#\{\ell\in\N^0~|~\zeta_{\ell}^c(x)\neq 0\})$ is bounded for a shape-regular $\T^0$.

We now turn our attention to $T_{22}$. By the Cauchy-Schwarz inequality and the weighted Poincar\'{e} inequality, can get
\begin{eqnarray}\label{eq:T22}
T_{22} =& \int_{\Omega}\omega(x)(\hat{f}-\hat{f}_{\N^0}) \cdot \omega^{-1}(x)(v-\bar{v})\dx \nonumber \\
\lesssim &\|\omega(x)(\hat{f}-\hat{f}_{\N^0})\|_{L^2(\Omega\setminus\Omega^{\rm QM})}\|\nabla v\|_{L^2(\Omega)},
\end{eqnarray}
where we use the fact that $\LQM\subset\N^0$ and $\zeta_{\ell}^c=c_\ell\zeta_{\ell}$ for $\ell\in\LQM$.

\def\IntpS{\mathcal{I}^{\rm S}}

To estimate the difference between interpolations of forces on the fine mesh $\T$ and the coarse mesh $\T^0$, we introduce the $C^{2}$-conforming interpolation $\check{f}$ and use the triangle inequality,
\begin{align}\label{eq:osc}
&\| \omega(x) (\hat{f} - \hat{f}_{\N^0})\|_{L^2(\Omega \backslash \Omega^{\rm QM})} \nonumber \\
 \lesssim & \| \omega(x) ( \-f - \check{f} ) \|_{L^2(\Omega \backslash \Omega^{\rm QM})} +
 \| \omega(x) ( \check{f} - \-f_{\N^0}) \|_{L^2(\Omega \backslash \Omega^{\rm QM})}\nonumber \\
  \lesssim &   \| \omega \nabla^2\check{f}\|_{L^2(\Omega \backslash \Omega^{\rm QM})} + \|\omega h_{0}^2 \nabla^2\check{f}  \|_{L^2(\Omega \backslash \Omega^{\rm QM})} \nonumber \\
  \lesssim & 
   \| \omega \check{f}\|_{\ell^2(\L^{\Omega}\setminus\LQM)} \lesssim  \Big(\RQM^{-\alpha_2+d/2-1} + \RMM^{-\alpha_1+d/2-1}\Big)\|\hat{f}\|_{(\dot{H}^1)^*},
\end{align}
where $h_{0}(x):= \mathrm{diam}(T), ~\text{for }~ x\in T\in \T^{0}$. We note that in the third inequality we use the assumptions \eqref{eq:mesh}, \eqref{eq:decay} and the $\ell^2/L^2$ norm equivalence.

Combining \eqref{eq:fwfR}, \eqref{eq:Tf}, \eqref{eq:osc}, and let $\beta_{i} = \alpha_{i} - d/2 - 1$ for $i=1,2$, we have
\begin{align*}
\|\hat{f}-\hat{f}_{\T^0}\|_{(\dot{H}^1)^*} \lesssim \Big(\RO^{-\beta_1}+ \RQM^{-\beta_2} + \RMM^{-\beta_1}\Big)\|\hat{f}\|_{(\dot{H}^1)^*}+R_\Omega^{1+d/2}\exp(-\eta_1\rcut).
\end{align*}
By analogous analysis, we could obtain the following estimate for $d=2$, 
\begin{align*}
\|\hat{f}-\hat{f}_{\T^0}\|_{(\dot{H}^1)^*} \lesssim \log\RO\Big(\RO^{-\beta_1}+\RQM^{-\beta_2} + \RMM^{-\beta_1}\Big)\|\hat{f}\|_{(\dot{H}^1)^*}+R_\Omega^{2}\exp(-\eta_1\rcut).
\end{align*}
which yields the stated result.

\subsection{Decay estimates of the residual forces}
\label{appdx:force_decay}

We mainly give a decay estimate of the residual force $f_{\ell}(\uH)$ for $\ell \in (\LMM \setminus \Lbuf) \setminus \LI$, where $\LI:=\{\ell\in\L~|~\RMM-\Rbuf\leq|\ell|\leq\RMM+\Rbuf\}$ is the MM/FF interface region. We also discuss about the weaker estimate on MM/FF interface region due to the loss of symmetry. We note that for $\ell \in \LQM \cup \Lbuf$ there is no need to discuss its decay estimate due to the construction of $\T^0$ introduced in \S~\ref{sec:discrete}, and for $\ell \in \LFF\setminus\LI$ the decay estimate has already been studied in \cite{chen19}.

Let $E_{\ell}$ be the site energy introduced in \S~\ref{sec:tb},  we define $V_{\ell} : (\R^d)^{\L-\ell}\rightarrow\R$ by,
\begin{eqnarray*} 
V_{\ell}\big(Du\big) :=  E_{\ell}(x_0 + u),
\end{eqnarray*}
where $x_0:\L\rightarrow\R^d$, $x_0(\ell)=\ell$ is the reference configuration. 

For simplicity of the following presentation, let $\mathcal{R^{\rm BUF}_{\ell}}:=B_{\Rbuf}(\ell)\cap\L$ and $\uu(\ell):=u_0(\ell)+\uH(\ell)$. By this definition, the residual force defined in \eqref{eq:forcenew} could be expressed in terms of QM site potentials $V_\ell$ for $\ell \in \L$
\begin{align*}
f_{\ell}(\uH) = \sum_{\rho\in\ell-\L} V_{\ell-\rho, \rho}\big(D\uu(\ell-\rho)) - \sum_{\rho\in\L-\ell} V_{\ell, \rho}\big(D\uu(\ell)\big).
\end{align*} 

For $\ell \in (\LMM \setminus \Lbuf) \setminus \LI$ belonging to the inner MM region, the force corresponding to QM/MM hybrid energy \eqref{eq:hybrid_energy} is defined by $f^{\rm H}_{\ell}(\uH) := \frac{ \partial \E^{\rm H}(u)}{\partial u(\ell)}\Big|_{u=\uH}$, which can also be written as
\begin{align*}
f_{\ell}^{\textrm{H}}(\uH) := \sum_{\rho\in\ell-\Rg} V^{\textrm{MM}}_{\ell-\rho,\rho}\big(D\uu(\ell-\rho)\big) - \sum_{\rho\in \Rg-\ell} V^{\textrm{MM}}_{\ell,\rho}\big(D\uu(\ell)\big),
\end{align*} 
where
\begin{eqnarray*}
V^{\rm MM}_{\ell}\big(Du\big)
:= \sum_{j=0}^k \frac{1}{j!} \delta^j V_{\ell}^{B_{\Rbuf}(\ell)}\big(0\big)[
	\underset{\text{$j$ times}}{\underbrace{Du, \dots, Du}} ] \quad 
\end{eqnarray*}
with $k\in\mathbb{N}$ the order of the Taylor expansion. For $\ell\in\LI$, we note that $f_{\ell}^{\rm H}(\uH)$ is defined similarly, the only difference is that there only exists the interactions from the MM region. We mention that for $\ell \in \LMM \setminus \Lbuf$, which is far away from the defect core, the site potential $V_{\ell}^{B_{\Rbuf}(\ell)}$ is nearly homogeneous due to Lemma \ref{lemma-thermodynamic-limit}, hence in the following we simply use the notation $V_{\ell}^{\rm BUF}$. 

We now give the decay estimate of $f_\ell(\uH)$ for $\ell$ belongs to inner MM region. 

\begin{proposition}\label{prop:decayforces}
If $k=2$ for point defects and $k=3$ for dislocations, then we have
\begin{enumerate}
\item Point defects: $d=2$ or $d=3$, $|f_{\ell}(\uH)| \eqsim (1+|\ell|)^{-3d}$ for $\ell \in \LMM$.
\item Dislocation: $d=3$, 
$|f_{\ell}(\uH)| \eqsim |\ell|^{-4}$ for $\ell \in \LMM$.
\end{enumerate}
\end{proposition}
\begin{proof}
Observing that $f^{\rm H}_{\ell}(\uH)=0$ for each $\ell \in \LMM\setminus \Lbuf$ and recalling the definition of $\tilde{f}_{\ell}$ in \eqref{eq:ftilde}, for $\bm{\sigma}=(\sigma_1, ..., \sigma_{k+1})$, we obtain
\begin{align}\label{eq:forcediff}
f_{\ell}(\uH) - f_{\ell}^{\rm H}(\uH)  = & f_{\ell}(\uH)- \tilde{f}_{\ell}(\uH) + \tilde{f}_{\ell}(\uH) - f_{\ell}^{\rm H}(\uH) \nonumber \\
 = & f_{\ell}(\uH)- \tilde{f}_{\ell}(\uH) + \frac{1}{(k+1)!}\sum_{\rho\in \Rg-\ell}\sum_{\bm{\sigma} \in (\Rg-\ell)^{k+1}}V^{\rm BUF}_{\ell,\rho \bm{\sigma}}\big(0\big)\prod^{k+1}_{j=1} D_{\sigma_j}\uu(\ell)  \nonumber \\
 & + \frac{1}{(k+1)!}\sum_{\rho \in\ell - \Rg} \sum_{\bm{\sigma} \in (\ell - \Rg)^{k+1}} V^{\rm BUF}_{\ell-\rho, \rho\bm{\sigma}}\big(0\big) \prod^{k+1}_{j=1} D_{\sigma_j}\uu(\ell-\rho)
 \nonumber \\
 =:& F_1 + F_2 + F_3.
\end{align} 

 Lemma \ref{lemma-thermodynamic-limit} leads to,
 \begin{align*}
 \big|F_1\big| \lesssim& e^{-\eta \Rbuf}.
\end{align*}

For $\ell$ belongs to the inner MM region, i.e., $\ell \in (\LMM \setminus \Lbuf) \setminus \LI$, combining Cauchy-Schwarz inequality and \eqref{eq:forcediff}, we have
\begin{align*}
 \big|F_2\big| \lesssim& 
 \prod^{k+1}_{j=1} \Big(\sum_{\sigma_j \in \Rg-\ell} e^{-2\gamma|\sigma_j|}\big|D_{\sigma_j}\uu(\ell)\big|^2\Big)^{1/2} = \big|D\uu(\ell)\big|^{k+1}_\gamma.
\end{align*}

The term $F_3$ can be estimated similarly,
hence for $\ell \in (\LMM \setminus \Lbuf) \setminus \LI$, we have
\begin{align}\label{eq:inner}
\big| f_{\ell}(\uH) \big| 
 \lesssim & e^{-\eta \Rbuf} + 
 \big|D\uu(\ell)\big|^{k+1}_\gamma.
\end{align} 

We now discuss about the decay estimate for $\ell \in \LI$. 
Due to the loss of symmetry, we note that the decay estimate is weaker than that inside MM region. For example, for $\ell \in \LI \cap \LMM$, if there exists a $\sigma_j \in \Rg-\ell$ such that $\ell + \sigma_j \in \LFF$, then the term $D_{\sigma_j}\uu(\ell)$ will becomes to $D_{\sigma_j}u_0(\ell)-\uH(\ell)$. 
Hence, similar to the estimate for inner MM region, for $\ell \in \LI$, we have
\begin{align*}
\big| f_{\ell}(\uH) \big| 
 \lesssim & e^{-\eta \Rbuf} +  \big(\big|Du_0(\ell)\big|^{k+1}_\gamma+\big|\uH(\ell)\big|^{k+1}\big).
 \end{align*} 
Indeed we have grossly overestimated here, but there is no advantage in a sharper estimate since we do not really care about the decay estimate on MM/FF interface.

Using the generic decay estimate of $\bar{u}$ and $u_0$ in \cite{chen19}, $|D\bar{u}(\ell)|_\gamma \lesssim (1+|\ell|)^{-d}\log^t(2+|\ell|)$, $t=0$ for point defects, $t=1$ for dislocations and $|Du_0(\ell)| \lesssim |\ell|^{-1}$, together with an additional assumption that $u$ and $\uH$ have the same decay rates, namely $|D^j\uH(\ell)|_\gamma \eqsim |D^j\bar{u}(\ell)|_\gamma$ for $j=0,1$, we derive the stated results by \eqref{eq:inner}.
\end{proof}

\bibliographystyle{plain}
\bibliography{posterioriQMMM.bib}
\end{document}